\documentclass[12pt]{amsart}

\usepackage{amssymb}

\usepackage{enumitem}

\usepackage{amsmath}
\usepackage{latexsym}
\usepackage{extarrows}
\usepackage{enumerate}
\usepackage{txfonts}
\usepackage{mathtools}
\usepackage{bm}
\usepackage{tikz-cd}
\usepackage{xcolor}
\usepackage{hyperref}

\hypersetup{
    hidelinks,
    linktoc=page,
    }

\counterwithin{figure}{section}

\makeatletter
\@namedef{subjclassname@2020}{%
  \textup{2020} Mathematics Subject Classification}
\makeatother

\usepackage[T1]{fontenc}

\newtheorem{theorem}{Theorem}[section]
\newtheorem{corollary}[theorem]{Corollary}
\newtheorem{lemma}[theorem]{Lemma}
\newtheorem{proposition}[theorem]{Proposition}

\theoremstyle{definition}
\newtheorem{definition}[theorem]{Definition}
\newtheorem{remark}[theorem]{Remark}
\newtheorem{example}[theorem]{Example}

\numberwithin{equation}{section}

\newcommand\E{\mathbb{E}}

\newcommand\R{\mathbb{R}}

\newcommand\C{\mathbb{C}}

\newcommand\N{\mathbb{N}}

\newcommand\B{\mathcal{B}}
\newcommand\X{\mathcal{X}}
\newcommand\Y{\mathcal{Y}}

\newcommand\blue{\color{blue}}

\newcommand{\taildoubleheadrightarrow}{%
  \rightarrowtail\mathrel{\mspace{-15mu}}\rightarrow
}
\newcommand{\doubleheadrightleftarrow}{%
  \twoheadleftarrow\mathrel{\mspace{-15mu}}\twoheadrightarrow
}

\renewcommand{\c}[1]{\mathcal{#1}}
\newcommand\op{{\operatorname{op}}}
\newcommand\Hom{\operatorname{Hom}}

\newcommand\Nat{\operatorname{Nat}}

\newcommand\id{\operatorname{id}}
\newcommand\eps{\varepsilon}

\newcommand\Cat{\mathcal{C}}
\newcommand\AbsMes{\mathbf{AbsMbl}}
\newcommand\ConcMes{\mathbf{CncMbl}}
\newcommand\AbsProb{\mathbf{AbsPrb}}
\newcommand\ConcProb{\mathbf{CncPrb}}

\newcommand\CH{\mathbf{CH}}
\newcommand\CHpt{{(\mathrm{pt} \downarrow \CH)}}
\newcommand\LCH{\mathbf{LCH}}
\newcommand\LCHpr{\mathbf{LCH_{p}}}
\newcommand\Polish{\mathbf{Pol}}
\newcommand\CMet{\mathbf{CMet}}
\newcommand\CMetProb{{\mathbf{CMetPrb}}}
\newcommand\CHProb{{\mathbf{CHPrb}}}
\newcommand\PolishProb{{\mathbf{PolPrb}}}
\newcommand\LCHProb{{\mathbf{LCHPrb}}}
\newcommand\CHptProb{(\mathrm{pt} \downarrow \CH)\mathbf{Prb}}
\newcommand\LCHprProb{\mathbf{LCH_{p}Prb}}

\newcommand\Bool{\mathbf{Bool}}
\newcommand\Boolop{\mathbf{Bool^\op}}
\newcommand\SigmaAlg{{\mathbf{Bool}_\sigma}}
\newcommand\SigmaAlgop{{\mathbf{Bool}_\sigma^\op}}
\newcommand\ProbAlg{\mathbf{PrbAlg}}
\newcommand\CStarAlgMult{\mathbf{CC^*Alg_{\mathtt{Mult,nd}}}}
\newcommand\CStarAlgMultop{\mathbf{CC^*Alg_{\mathtt{Mult,nd}}^\op}}
\newcommand\CStarAlgNd{\mathbf{CC^*Alg_{\mathtt{nd}}}}
\newcommand\CStarAlgNdop{\mathbf{CC^*Alg_{\mathtt{nd}}^\op}}
\newcommand\CStarAlg{\mathbf{CC^*Alg}}

\newcommand\CStarAlgUnit{\mathbf{CC^*Alg_1}}
\newcommand\CStarAlgUnitop{\mathbf{CC^*Alg_1^\op}}
\newcommand\CStarAlgUnitInf{(\mathbf{CC^*Alg_1} \downarrow \C)}
\newcommand\CStarAlgUnitInfop{(\mathbf{CC^*Alg_1} \downarrow \C)^\op}
\newcommand\CStarAlgMultTrace{\mathbf{CC^*Alg}^\tau_{\mathtt{Mult,nd}}}
\newcommand\CStarAlgMultTraceop{(\mathbf{CC^*Alg}^\tau_{\mathtt{Mult,nd}})^\op}
\newcommand\CStarAlgNdTrace{\mathbf{CC^*Alg}^\tau_{\mathtt{nd}}}
\newcommand\CStarAlgNdTraceop{(\mathbf{CC^*Alg}^\tau_{\mathtt{nd}})^\op}
\newcommand\CStarAlgTrace{\mathbf{CC^*Alg}^\tau}

\newcommand\CStarAlgUnitTrace{\mathbf{CC^*Alg^\tau_1}}
\newcommand\CStarAlgUnitTraceop{(\mathbf{CC^*Alg^\tau_1})^\op}
\newcommand\CStarAlgUnitInfTrace{(\mathbf{CC^*Alg_1} \downarrow \C)^\tau}
\newcommand\CStarAlgUnitInfTraceop{((\mathbf{CC^*Alg_1} \downarrow \C)^\tau)^\op}
\newcommand\Set{\mathbf{Set}}
\newcommand\Group{\mathbf{Grp}}
\newcommand\StoneCat{\mathbf{Stone}}
\newcommand\StoneCatSigma{\mathbf{Stone_\sigma}}

\newcommand\AbsDelete{\mathbf{AbsNul}}
\newcommand\ConcDelete{\mathbf{CncNul}}
\newcommand\CHDelete{\mathbf{CHNul}}
\newcommand\vonNeumann{\mathbf{CvNAlg}^\tau}
\newcommand\vonNeumannop{(\mathbf{CvNAlg}^\tau)^\op}
\newcommand\Model{\mathbf{Model}}

\newcommand\Baire{\mathcal{B}a}
\newcommand\BaireB{\mathcal{B}a_b}
\newcommand\BaireC{\mathcal{B}a_c}
\newcommand\Borel{\mathcal{B}o}
\newcommand\Mes{\mathtt{Alg}}
\newcommand\BaireMeagFunc{\mathtt{Meager}}
\newcommand\BaireFunc{\mathtt{Bair}}
\newcommand\BaireBFunc{\mathtt{Bair_b}}
\newcommand\BaireCFunc{\mathtt{Bair_c}}
\newcommand\BorelFunc{\mathtt{Bor}}
\newcommand\Loomis{\mathtt{LS}}
\newcommand\Spec{\mathtt{Spec}}
\newcommand\Stone{\mathtt{Conc}}
\newcommand\Inc{\mathtt{Inc}}
\newcommand\StoneFunc{\mathtt{Stone}}
\newcommand\Clopen{\mathtt{Clopen}}
\newcommand\StoneFuncSigma{\mathtt{Stone}_\sigma}
\newcommand\ClopenSigma{\mathtt{Clopen}_\sigma}
\newcommand\Abs{\mathtt{Abs}}
\newcommand\ident{\mathtt{id}}
\newcommand\Func{\mathtt{Func}}
\newcommand\Idem{\mathtt{Proj}}
\newcommand\Vertex{\mathtt{Vertex}}
\newcommand\Riesz{\mathtt{Riesz}}

\newcommand\CFunc{\mathtt{C}}
\newcommand\CoFunc{\mathtt{C}_0}
\newcommand\CbFunc{\mathtt{C_b}}
\newcommand\CcFunc{\mathtt{C_c}}
\newcommand\Linfty{\mathtt{L^\infty}}
\newcommand\Mult{\mathtt{Mult}}
\newcommand\Alex{\mathtt{Alex}}
\newcommand\Unit{\mathtt{Unit}}
\newcommand\Forget{\mathtt{Forget}}
\newcommand\Cast{\mathtt{Cast}}
\newcommand\Inv{\mathtt{Inv}_\Gamma}

\begin{document}


\baselineskip=17pt


\title[Foundational aspects of uncountable measure theory]{Foundational aspects of uncountable measure theory: Gelfand duality, Riesz representation, canonical models, and canonical disintegration}

\author[A. Jamneshan]{Asgar Jamneshan}
\address{Department of Mathematics\\ Koc University \\ 
Istanbul \\
34450, Turkey}
\email{ajamneshan@ku.edu.tr}

\author[T. Tao]{Terence Tao}
\address{Department of Mathematics\\ University of California \\ 
Los Angeles \\
CA 90095-1555, USA}
\email{tao@math.ucla.edu}

\date{\today}

\begin{abstract}
We collect several foundational results regarding the interaction between locally compact spaces, probability spaces and probability algebras, and commutative $C^*$-algebras and von Neumann algebras equipped with traces, in the ``uncountable'' setting in which no separability, metrizability, or standard Borel hypotheses are placed on these spaces and algebras.  In particular, we review the Gelfand dualities and Riesz representation theorems available in this setting.  We also present a canonical model that represents probability algebras as compact Hausdorff probability spaces in a completely functorial fashion, and apply this model to obtain a canonical disintegration theorem and to readily construct various product measures.  These tools are useful in applications to ``uncountable'' ergodic theory (as demonstrated by the authors and others).
\end{abstract}



\maketitle

\tableofcontents

\section{Introduction}\label{sec-intro}

In this paper we establish various foundational results about the measure theory (and also point set topology and functional analysis) of ``uncountable'' spaces: topological spaces that are not required to be separable or Polish, measurable spaces that are not required to be standard Borel, measure spaces that are not required to be standard Lebesgue, and $C^*$-algebras that are not required to be separable.  In other work by us and other authors \cite{roth,jt20,jamneshan2019fz,jst} we use these results to establish various results in ``uncountable'' ergodic theory (in which the acting groups $\Gamma$ are not required to be countable or the underlying probability spaces/ algebras are not required to be separable), which in turn can be applied to various ``uncountable'' systems constructed using ultraproducts and similar devices to obtain combinatorial consequences.  

In this paper we focus on the following (interrelated) families of results:

\begin{itemize}
\item[(i)] The compactification of locally compact Hausdorff spaces, and the Gel- fand dualities between categories of these spaces and various categories of commutative $C^*$-algebras.
\item[(ii)]  Riesz representation theorems on compact and locally compact Hausdorff spaces, leading to various ``Riesz dualities'' between categories of compact or locally compact Hausdorff probability spaces and categories of \emph{tracial} commutative $C^*$-algebras.
\item[(iii)]  Construction of a \emph{canonical model} of probability algebras as compact Hausdorff spaces with good category-theoretic properties (based on combining the above dualities with a ``probability duality'' between abstract probability spaces and tracial commutative von Neumann algebras).
\item[(iv)] Construction of a canonical disintegration of probability measures with respect to a factor space, via the aforementioned canonical model, and the use of this disintegration to construct relatively independent products.
\item[(v)]  Connections with various Stone dualities between categories of Stone spaces and categories of Boolean algebras, focusing in particular on the duality provided by the Loomis--Sikorski theorem, and using this duality to establish an abstract version of the Kolmogorov extension theorem. 
\end{itemize}

Most of the above results are already known in the literature, though sometimes in a different guise; we discuss relevant references at all stages of this paper. Our primary contribution is to synthesize them into an arrangement in which they appear as different aspects of a coherent whole. As the above descriptions indicate, we will rely hereby on the language of category theory to describe, organize, and interpret our results, as well as the results already in the literature.  Indeed, we found that an insistence on ensuring that various operations or identifications can be viewed as functors or natural transformations to be extremely elucidating, for instance clarifying the different versions of the Baire algebra or the Riesz representation theorem that exist in the literature by assigning each such version to a slightly different category.  We highlight the category-theoretic notions of \emph{categorical product}, \emph{natural isomorphism}, and \emph{duality of categories} as being of particular relevance to our investigations.  We review the basic terminology of category theory we will need in Appendix \ref{category-appendix}.

\begin{table}
    \centering
    \begin{tabular}{|l|l|}
    \hline
     $\CH$, $\CStarAlgUnit$ & Definition \ref{def-ch} \\
     \hline
     $\LCH$, $\LCHpr$, $\CHpt$ & Definition \ref{def-alex} \\
     \hline
     $\CStarAlgNd$, $\CStarAlgMult$, $\CStarAlgUnitInf$ & Definition \ref{def-gelfand} \\ 
     \hline
     $\CMet$, $\Polish$, $\ConcMes$ & Definition \ref{bairefunc-def} \\
     \hline 
     $\ConcProb$, $\CMetProb$, $\CHProb$, $\CHptProb$, & Definition \ref{top-prob-cat}\\
     $\PolishProb$, $\LCHProb$, $\LCHprProb$ & \\
    \hline 
     $\CStarAlgTrace$, $\CStarAlgUnitInfTrace$,  $\CStarAlgNdTrace$, $\CStarAlgMultTrace$  & Definition \ref{tracialcstar-def} \\
    \hline
    $\Bool$,  $\SigmaAlg$, $\AbsMes$, $\AbsProb$, $\ProbAlg$ & Definition \ref{abscat-def} \\
    \hline
         $\vonNeumann$ & Definition \ref{vonneumann-def}\\
         \hline
        $\StoneCat$, $\StoneCatSigma$ & Definition \ref{stone-dual}\\
        \hline
        $\ConcDelete$, $\AbsDelete$, $\CHDelete$ & Definition \ref{null-def} \\
         \hline
    \end{tabular}
    \caption{A list of the categories in this paper and the location where they are first defined or introduced.}
    \label{tab:categories}
\end{table}

\subsection{Compactification and Gelfand duality}

In this paper we use the term \emph{Gelfand duality} to refer to a number of duality of categories between categories of compact or locally compact Hausdorff spaces on one hand, and categories of commutative $C^*$-algebras on the other.  To illustrate the most basic example of Gelfand duality, we introduce the compact Hausdorff category $\CH$ and the unital commutative $C^*$-algebra category $\CStarAlgUnit$.

\begin{definition}[$\CH$ and $\CStarAlgUnit$]\ \label{def-ch}
\begin{itemize}
\item[(i)]  A \emph{$\CH$-space} is a compact Hausdorff space $X = (X_\Set, {\mathcal F}_X)$, that is to say a set $X_\Set$ equipped with a topology ${\mathcal F}_X$ that makes the set compact and Hausdorff.  A \emph{$\CH$-morphism} $f \colon X \to Y$ between two $\CH$-spaces is a $\Set$-morphism (i.e., a function) $f_\Set \colon X_\Set \to Y_\Set$ between the underlying sets which is continuous, using the usual $\Set$-composition law.
\item[(ii)]  A \emph{$\CStarAlgUnit$-algebra} is a unital commutative $C^*$-algebra ${\mathcal A}$.  A \emph{$\CStarAlgUnit$-morphism} $\Phi \colon {\mathcal A} \to {\mathcal B}$ is a unital $*$-homomorphism from ${\mathcal A}$ to ${\mathcal B}$.
\item[(iii)]  If $X$ is a $\CH$-space, we define $\CFunc(X)$ to be the $\CStarAlgUnit$-algebra of continuous functions $f \colon X \to \C$ from $X$ to the complex numbers $\C$, endowed with the obvious structure of a unital $C^*$-algebra.  If $T \colon X \to Y$ is a $\CH$-morphism, we define $\CFunc(T) \colon \CFunc(Y) \to \CFunc(X)$ to be the Koopman operator $\CFunc(T)(f) \coloneqq f \circ T$.
\item[(iv)]  If ${\mathcal A}$ is a $\CStarAlgUnit$-algebra, we define $\Spec({\mathcal A})$ (the \emph{Gelfand spectrum} of ${\mathcal A}$) to be the space $\Hom_{\CStarAlgUnit}({\mathcal A} \to \C)$ of $\CStarAlgUnit$-morphisms $\lambda \colon {\mathcal A} \to \C$ from ${\mathcal A}$ to $\C$ (viewing the latter as a $\CStarAlgUnit$-algebra), equipped with the topology induced from the product topology on the space $\C^{\mathcal A}$ of all functions from ${\mathcal A}$ to $\C$; this is a $\CH$-space thanks to the Banach-Alaoglu theorem.  If $\phi \colon {\mathcal A} \to {\mathcal B}$ is a $\CStarAlgUnit$-morphism, then $\Spec(\Phi) \colon \Spec({\mathcal B}) \to \Spec({\mathcal A})$ is the $\CH$-morphism defined by $\Spec(\Phi)(\lambda) \coloneqq \lambda \circ \Phi$ for all $\lambda \in \Hom_{\CStarAlgUnit}({\mathcal A} \to \C)$.
\end{itemize}
\end{definition}

It is a routine matter to verify that $\CH$ and $\CStarAlgUnit$ are categories, and $\CFunc \colon \CH \to \CStarAlgUnitop$ and $\Spec \colon \CStarAlgUnitop \to \CH$ are functors between the indicated categories. 

It is well-known (see, e.g., \cite{negrepontis}, \cite{semadeni} or \cite[Theorem 1.20]{folland-harmonic-analysis}) that the functors $\CFunc, \Spec$ are faithful and full which invert each other up to natural isomorphisms, thus giving a duality of categories which we refer to as the \emph{Gelfand duality} between $\CH$ and $\CStarAlgUnit$.  We summarize all these facts as a single diagram
\begin{center}
\begin{tikzcd}
    \CStarAlgUnitop \arrow[d, tail,"\Spec",two heads,shift left=.75ex]  \\ 
    \CH \arrow[u,"\CFunc",   tail, two heads,shift left=.75ex] 
\end{tikzcd}.
\end{center}
In fact we have the larger, essentially commuting, diagram of Gelfand dualities depicted in Figures\footnote{These figures, as well as several other figures in this paper, can be viewed as ``coordinate charts'' of a single enormous diagram of categories that encompass a large number of types of objects and morphisms that are studied in topology, measure theory, probability theory, and operator algebras.  This unified diagram is far too large and dense to depict in a presentable fashion, so we have opted instead to only reveal portions of it at a time.} \ref{fig:gelfand-dual}, \ref{fig:gelfand-forget}, where (roughly speaking)

\begin{itemize}
 \item $\LCH$ is the category of locally compact Hausdorff spaces, with morphisms required to be continuous;
 \item $\LCH_p$ is the category of locally compact Hausdorff spaces, with morphisms required to be both continuous and proper;
 \item $\CHpt$ is the category of pointed compact Hausdorff spaces, equipped with a distinguished point, and with morphisms required to preserve this point;
 \item $\beta \colon \LCH \to \CH$ is the Stone--{\v C}ech compactification functor;
 \item $\Alex \colon \LCH_p \to \CHpt$ is the Alexandroff (or one-point) compactification functor;
 \item $\CStarAlgNd$ is the category of commutative $C^*$-algebras with morphisms taking values in the target algebra and required to be nondegenerate;
 \item $\CStarAlgMult$ is the category of commutative $C^*$-algebras with morphisms taking values in the multiplier algebra and required to be nondegenerate;
 \item $\CStarAlgUnitInf$ is the category of unital commutative $C^*$-algebras ${\mathcal A}$ endowed with a distinguished unital $*$-homomorphism to $\C$, with the morphisms required to preserve this $*$-homomorphism;
 \item $\CoFunc(X)$ is the space of continuous functions on $X$ which vanish at infinity;
 \item $\CbFunc(X)$ is the space of bounded continuous functions on $X$;
 \item $\Mult$ is the multiplier algebra functor;
 \item $\Unit$ is the functor that adjoins a unit to a $C^*$-algebra ${\mathcal A}$ to create a unital $C^*$-algebra ${\mathcal A} \oplus \C$, with the coordinate $*$-homomorphism $\lambda_* \colon {\mathcal A} \oplus \C \to \C$.
\end{itemize}

We describe these categories and functors in more detail in Section \ref{gelfand-sec}, where we also present various commutativity relations and dualities of categories that are implicit in Figures \ref{fig:gelfand-dual}, \ref{fig:gelfand-forget}, formalized as Theorem \ref{gelfand-dualities}.  Each of these functors and equivalences already occur either implicitly or explicitly in the literature, but to our knowledge this is the first time they have been combined into the above two diagrams.  In particular, we believe that these diagrams clarify an ambiguity in the Gelfand duality literature in which morphisms between locally compact Hausdorff spaces were sometimes, but not always, required to be proper, and morphisms between $C^*$-algebras were sometimes, but not always, required to lie in the target algebra rather than the multiplier algebra.  This ambiguity is resolved by noting that there are \emph{two} natural categories $\LCH$, $\LCHpr$ of locally compact Hausdorff spaces, and \emph{two} natural categories $\CStarAlgNd$, $\CStarAlgMult$ of commutative $C^*$-algebras.

\begin{figure}
    \centering
    \begin{tikzcd}
      \CStarAlgUnitInfop \arrow[d, tail,two heads,"\Spec",shift left=.75ex]  & \CStarAlgNdop \arrow[d, tail,two heads,"\Spec",shift left=.75ex] \arrow[l, tail,"\Unit"']  \arrow[r, tail] & \CStarAlgMultop \arrow[r,tail, "\Mult"]   \arrow[d, tail,"\Spec",two heads,shift left=.75ex] & \CStarAlgUnitop \arrow[d, tail,"\Spec",two heads,shift left=.75ex] \\ 
    \CHpt \arrow[u,"\CFunc", two heads,  tail, shift left=.75ex]  & \LCHpr  \arrow[u,  "\CoFunc",tail,two heads,shift left=.75ex] \arrow[l,tail, "\Alex"] \arrow[r, tail] & \LCH \arrow[ur,  tail, "\CbFunc"] \arrow[r, tail,"\beta"'] \arrow[u,  "\CoFunc",tail, two heads,shift left=.75ex] &  \CH \arrow[u, "\CFunc",tail,two heads,shift left=.75ex] 
\end{tikzcd}
    \caption{Gelfand dualities. Tailed arrows indicate faithful functors; and an arrow with a doubled head indicates a full functor.  Unlabeled functors are forgetful functors. (These conventions remain in force for all other diagrams of categories and functors in this paper.)  This diagram commutes up to natural isomorphisms.}
    \label{fig:gelfand-dual}
\end{figure}

\begin{figure}
    \centering
    \begin{tikzcd}
      \CStarAlgUnitInfop \arrow[d, tail,two heads,"\Spec",shift left=.75ex] \arrow[r,blue,tail]  & \CStarAlgUnitop \arrow[d, tail,"\Spec",two heads,shift left=.75ex] \arrow[r,blue, tail, two heads] \arrow[rr,blue, tail, two heads, bend left] & \CStarAlgNdop \arrow[d, tail,two heads,"\Spec",shift left=.75ex] \arrow[r,  tail]  & \CStarAlgMultop \arrow[d, tail,"\Spec",two heads,shift left=.75ex] &  \\ 
    \CHpt \arrow[r, blue,tail]  \arrow[u,"\CFunc", two heads,  tail, shift left=.75ex]  & 
     \CH \arrow[u, "\CFunc",tail,two heads,shift left=.75ex] \arrow[r, blue, tail, two heads] \arrow[rr,blue, tail, bend right, two heads] & 
    \LCHpr  \arrow[u,  "\CoFunc",tail,two heads,shift left=.75ex] \arrow[r, tail] & \LCH \arrow[u,"\CoFunc",tail, two heads,shift left=.75ex] 
\end{tikzcd}
    \caption{Forgetful functors in the locally compact and $C^*$-algebra categories.  This diagram also commutes up to natural isomorphisms, but does not commute with the previous diagram. Blue arrows indicate casting functors, as per Definition \ref{cast}.  We do not deem the forgetful functor from $\LCHpr$ to $\LCH$ (or from $\CStarAlgNd$ to $\CStarAlgMult$) to be casting, as these functors do not commute with other casting functors we will use later, such as $\BaireBFunc$ and $\BaireCFunc$.}
    \label{fig:gelfand-forget}
\end{figure}

\subsection{Baire algebras and Riesz duality}

We now augment the Gelfand dualities just discussed by endowing the locally compact Hausdorff spaces with a probability measure on one hand, and endowing the commutative $C^*$-algebras with a trace on the other hand, giving rise to a new collection of dualities of categories based on various forms of the Riesz representation theorem, which we shall term ``Riesz dualities''.

In order to describe these Riesz dualities, one must first address a fundamental measure-theoretic question, namely which $\sigma$-algebra one should associate to a given topological space $X$.  In the literature there are three commonly used options to choose from:
\begin{itemize}
\item[(i)]  The \emph{Borel $\sigma$-algebra} $\Borel(X)$, generated by the open (or equivalently, closed) subsets of $X$.
\item[(ii)]  The \emph{$\CbFunc$-Baire $\sigma$-algebra} $\BaireB(X)$, generated by the bounded complex-valued\footnote{We always endow $\R$ and $\C$ with the Borel $\sigma$-algebra.} continuous functions $\CbFunc(X)$ of $X$ (or equivalently, the space of arbitrary continuous functions into $\C$).  
\item[(iii)]  The \emph{$\CcFunc$-Baire $\sigma$-algebra} $\BaireC(X)$, generated by the compactly supported complex-valued continuous functions $\CcFunc(X)$ of $X$ (or equivalently, by the space $\CoFunc(X) \coloneqq \overline{\CcFunc(X)}$ of continuous complex-valued functions that vanish at infinity).  We also refer to $\BaireC(X)$ as the \emph{$\CoFunc$-Baire $\sigma$-algebra}.  In Proposition \ref{baire-prop} we also establish the well-known fact that $\BaireC(X)$ is generated by the compact $G_\delta$ subsets of $X$.
\end{itemize}

When $X$ is a compact metrizable space, the three $\sigma$-algebras $\Borel(X)$, $\BaireB(X)$, $\BaireC(X)$ agree, and we also clearly have $\BaireB(X)=\BaireC(X)$ for compact Hausdorff spaces $X$; thus in these cases we can refer to both $\BaireB(X)$ and $\BaireC(X)$ simply as the Baire $\sigma$-algebra $\Baire(X)$.  In general we only have the obvious inclusions
\begin{equation}\label{baire-include}
\BaireC(X) \subseteq \BaireB(X) \subseteq \Borel(X);
\end{equation}
see Remark \ref{baire-remark} for further discussion.  Note that once one leaves the $\CH$ setting, there is no consensus in the literature as to which of $\BaireB(X)$, $\BaireC(X)$ should be referred to as \emph{the} Baire $\sigma$-algebra; the $\CbFunc$-Baire algebra $\BaireB(X)$ is favored for instance in \cite[Volume 2]{bogachev2006measure},  \cite{fremlinvol4}, \cite{dudley01}, \cite{hewitt2012abstract}, while the $\CcFunc$-Baire algebra $\BaireC(X)$ is favored in \cite{halmos-measure-theory}, \cite{royden-real-analysis}.  From our investigations we have concluded that the choice of $\sigma$-algebra should be determined by the category one has chosen to work in.  Specifically:

\begin{itemize}
\item In the category $\Polish$ of Polish spaces, the Borel $\sigma$-algebras $\Borel(X)$ are the most natural to use.
\item In the category $\LCH$ of locally compact Hausdorff spaces, the $\CbFunc$-Baire $\sigma$-algebras $\BaireB(X)$ are the most natural to use.
\item In the category $\LCHpr$ of locally compact Hausdorff spaces with proper morphisms, the $\CcFunc$-Baire $\sigma$-algebras $\BaireC(X)$ are the most natural to use.
\item In the category $\CH$ of compact Hausdorff spaces or the category $\CHpt$ of pointed compact Hausdorff spaces, the Baire $\sigma$-algebras $\Baire(X) = \BaireB(X) = \BaireC(X)$ are the most natural to use.
\item In the category $\CMet$ of compact metrizable spaces, the $\sigma$-algebras $\Borel(X) = \Baire(X) = \BaireB(X) = \BaireC(X)$ agree, and one can use them interchangeably.
\end{itemize}

With these choices we obtain functors $\BorelFunc \colon \Polish \to \ConcMes$, $\BaireFunc \colon \CH \to \ConcMes$, $\BaireBFunc \colon \LCH \to \ConcMes$, $\BaireCFunc \colon \LCHpr \to \ConcMes$ to the category $\ConcMes$ of (concrete) measurable spaces, as detailed in Definition \ref{bairefunc-def}; see Figure \ref{fig:baire}.  These functors also enjoy other pleasant category-theoretic properties, for instance being compatible with various product constructions; see Proposition \ref{prod-top}.  These choices are compatible with the folklore philosophy that Baire $\sigma$-algebras are ``less pathological'' than their Borel counterparts when working in ``uncountable'' settings in which the spaces are not assumed to be separable, metrizable, or Polish.  We caution that with the Baire algebra, individual points and other compact sets may become non-measurable, but this turns out to be surprisingly much less of a difficulty than one might initially imagine, particularly if one adopts an ``abstract'', ``point-free'' or ``pointless'' approach to measure theory (see Section \ref{abs-prob-sec}). See also \cite[Remark 5.8]{EFHN} for a comparative analysis of the Borel and Baire algebras in the context of Riesz representation theorems and product space constructions.

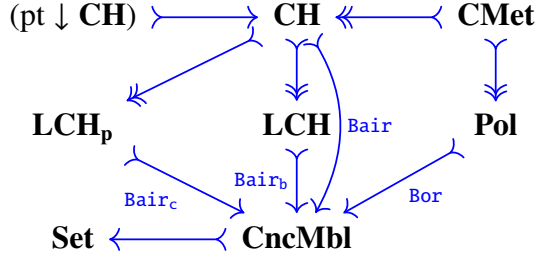
\begin{figure}
    \centering
    \begin{tikzcd}
    \CHpt \arrow[r,blue,tail,] & \CH \arrow[dl, blue,  tail, two heads] \arrow[dd,bend left, blue,tail,"\BaireFunc",shift left=0.25ex] \arrow[d,blue,tail, two heads] & \CMet \arrow[l,blue,tail, two heads] \arrow[d,blue,tail, two heads] \\
    \LCHpr \arrow[dr, blue, "\BaireCFunc"', tail] & \LCH \arrow[d, blue, "\BaireBFunc"', tail] & \Polish \arrow[dl, blue, tail, "\BorelFunc"] \\
    \Set & \ConcMes \arrow[l,blue,tail]
\end{tikzcd}
    \caption{Functors from topological categories to the concrete measurable category $\ConcMes$, which in turn has a forgetful functor to the category $\Set$ of sets. This diagram commutes (as is required as per the casting conventions in Definition \ref{cast}).}
    \label{fig:baire}
\end{figure}

Now that we have fixed the choice of $\sigma$-algebra to place on spaces in each of the topological categories, one can define the notion of a \emph{Radon probability measure}\footnote{We will not attempt to set up a Riesz representation theory for Polish spaces $X$, as these spaces need not be locally compact and so the spaces $\CcFunc(X), \CoFunc(X)$ can be quite degenerate. See \cite{varadarajan-riesz} for some exploration of Riesz representation type theorems in the absence of a hypothesis of local compactness.} on $\CH$-spaces, $\CHpt$-spaces, $\LCH$-spaces, and $\LCHpr$-spaces.    In the literature these Radon measures are usually defined on Borel sets and required to be inner regular with respect to compact sets; with our ``Baire-centric'' philosophy, the measures are instead defined on $\CoFunc$-Baire sets and are inner regular with respect to compact\footnote{A $G_\delta$ set is a countable intersection of open sets, and an $F_\sigma$ set is similarly a countable union of closed sets.} $G_\delta$ sets.  With this setup, it becomes possible to systematically attach Radon probability measures to the spaces in the categories $\CH$, $\CHpt$, $\LCH$, $\LCHpr$ to obtain categories $\CHProb$, $\CHptProb$, $\LCHProb$, $\LCHprProb$ of various types of locally compact Hausdorff spaces equipped with a Radon probability measure; see Definition \ref{top-prob-cat}.  For the categories $\CH, \CHpt$ the Radon hypothesis is in fact automatic (see Proposition \ref{automatic}) and may thus be omitted.  On the dual side, one can similarly attach a ``trace'' to the algebras in the categories $\CStarAlgUnit$, $\CStarAlgUnitInf$, $\CStarAlgMult$, $\CStarAlgNd$ to obtain categories $\CStarAlgUnitTrace$, $\CStarAlgUnitInfTrace$, $\CStarAlgMultTrace$, $\CStarAlgNdTrace$ of various types of commutative $C^*$-algebras equipped with a trace.  This is very much in line with the philosophy of noncommutative probability, in which a noncommutative probability space is often defined as some sort of $C^*$-algebra equipped with a trace, though in our case we are restricting attention solely to commutative $C^*$-algebras.  In Theorem \ref{rrt} below we then establish the fundamental Riesz representation theorems relating the categories
$\CHProb, \CHptProb, \LCHProb, \LCHprProb$ to their counterparts
$\CStarAlgUnitTrace, \CStarAlgUnitInfTrace, \CStarAlgMultTrace, \CStarAlgNdTrace$; our main tools for this will be several existing versions of the Riesz representation theorem (and the closely related Daniell-Stone representation theorem) in the literature.  As a consequence we obtain completely analogous versions of the diagrams of categories in Figures \ref{fig:gelfand-dual}, \ref{fig:gelfand-forget}; see Figures \ref{fig:riesz-dual}, \ref{fig:riesz-forget}.  A precise formulation of this statement is given in Theorem \ref{riesz-dualities}.

\begin{figure}
    \centering
    \begin{tikzcd}
      \CStarAlgUnitInfTraceop \arrow[d, tail,two heads,"\Riesz",shift left=.75ex]  & \CStarAlgNdTraceop \arrow[d, tail,two heads,"\Riesz",shift left=.75ex] \arrow[l, tail,"\Unit"']  \arrow[r, tail] & \CStarAlgMultTraceop \arrow[r,tail, "\Mult"]   \arrow[d, tail,"\Riesz",two heads,shift left=.75ex] & \CStarAlgUnitTraceop \arrow[d, tail,"\Riesz",two heads,shift left=.75ex] \\ 
    \CHptProb \arrow[u,"\CFunc", two heads,  tail, shift left=.75ex]  & \LCHprProb  \arrow[u,  "\CoFunc",tail,two heads,shift left=.75ex] \arrow[l,tail, "\Alex"] \arrow[r, tail] & \LCHProb \arrow[ur,  tail, "\CbFunc"] \arrow[r, tail,"\beta"'] \arrow[u,  "\CoFunc",tail, two heads,shift left=.75ex] &  \CHProb \arrow[u, "\CFunc",tail,two heads,shift left=.75ex] 
\end{tikzcd}
    \caption{Riesz dualities. This diagram commutes up to natural isomorphisms. There are forgetful functors to the corresponding categories in Figure \ref{fig:gelfand-dual}.}
    \label{fig:riesz-dual}
\end{figure}

\begin{figure}
    \centering
    \begin{tikzcd}
      \CStarAlgUnitInfTraceop \arrow[d, tail,two heads,"\Riesz",shift left=.75ex] \arrow[r,blue,tail]  & \CStarAlgUnitTraceop \arrow[d, tail,"\Riesz",two heads,shift left=.75ex] \arrow[r,blue, tail, two heads] \arrow[rr,blue, tail, bend left, two heads] & \CStarAlgNdTraceop \arrow[d, tail,two heads,"\Riesz",shift left=.75ex] \arrow[r,  tail]  & \CStarAlgMultTraceop \arrow[d, tail,"\Riesz",two heads,shift left=.75ex] &  \\ 
    \CHptProb \arrow[r, blue,tail]  \arrow[u,"\CFunc", two heads,  tail, shift left=.75ex]  & 
     \CHProb \arrow[u, "\CFunc",tail,two heads,shift left=.75ex] \arrow[r, blue, tail, two heads] \arrow[rr,blue, tail, bend right, two heads] & 
    \LCHprProb  \arrow[u,  "\CoFunc",tail,two heads,shift left=.75ex] \arrow[r, tail] & \LCHProb \arrow[u,  "\CoFunc",tail, two heads,shift left=.75ex] 
\end{tikzcd}
    \caption{Forgetful functors in the locally compact probabilistic and tracial $C^*$-algebra categories.  This diagram also commutes up to natural isomorphisms, but does not commute with the previous diagram.}
    \label{fig:riesz-forget}
\end{figure}

\subsection{Canonical models of abstract probability spaces}

Given a probability space $X = (X_\ConcMes, \mu_X) = (X_\Set, \Sigma_X, \mu_X)$, one can form the \emph{probability algebra}\footnote{This is a special case of the more familiar notion of a \emph{measure algebra}, which corresponds to the setting in which $X$ is a measure space instead of a probability space.} $X_\ProbAlg = (\Sigma_X/_{\mathcal{N}_X}, \mu_X)$, where ${\mathcal N}_X \coloneqq \{ E \in \Sigma_X \colon \mu_X(E) = 0\}$ is the null ideal, $\Sigma_X/_{\mathcal{N}_X}$ is the quotient algebra with respect to the $\sigma$-ideal ${\mathcal N}_X$ (which is well defined as a $\sigma$-complete Boolean algebra, though it need not be represented concretely as a $\sigma$-algebra of sets), and $\mu_X \colon \Sigma_X/_{\mathcal{N}_X} \to [0,1]$ is the descent of the measure $\mu_X \colon \Sigma_X \to [0,1]$ (here we abuse notation and write $\mu_X$ for both the concrete measure and its descent on the associated probability algebra).   Informally, one should view $X_\ProbAlg$ as a ``point-free'' or ``pointless'' abstraction of $X$ in which the null sets have been ``deleted''.  Every measure-preserving map $T \colon X \to Y$ between probability spaces $X,Y$ then gives rise to a \emph{$\ProbAlg$-morphism}\footnote{This is essentially the same concept as a \emph{measure space homomorphism} from \cite[Definition 5.1]{furstenberg2014recurrence}.} $T_\ProbAlg \colon X_\ProbAlg \to Y_\ProbAlg$ (by convention, we implicitly define $\ProbAlg$ as an opposite category in which the direction of all morphism arrows are reversed, see Definition \ref{abscat-def} for details), which remains unchanged if one modifies $T$ on a null set; see Definition \ref{abscat-def} for precise definitions.

A large part of ergodic theory can be viewed as taking place on probability algebras, by replacing any concrete measure-preserving transformation $T \colon X \to X$ with its abstract probability space counterpart $T_\ProbAlg \colon X_\ProbAlg \to X_\ProbAlg$.  This ``point-free'' approach to ergodic theory seems particularly well suited for studying actions of uncountable (discrete) groups, as by deleting the null sets in advance, one can avoid to a large extent the standard difficulty that an uncountable union of null sets is null. See our previous paper \cite{jt19} for an example of this philosophy.

However, in some applications one would like to be able to reverse the abstraction process, and represent an abstract measure-preserving action by a concrete one, preferably with some additional regularity properties (such as continuity).  If one insists on fixing the concrete model in advance, such a representation is not always possible, see, e.g., \cite[\S 343, 344]{fremlinvol3} or \cite{gtw}, where in \cite{gtw} an interesting example of a natural action (of the automorphism group of infinite dimensional Gaussian measure) is given which can only act in an abstract fashion on the underlying measure space, and cannot be described in terms of a Borel action.   

If one imposes suitable ``countability'' hypotheses on the group and measure space, however, one can \emph{model} an abstract group action by a concrete one\footnote{We stress that an implicit requirement in Theorem \ref{countable-model} is that the sought model satisfies also a separability hypothesis which is needed in applications of such separable models, for example for disintegration of measures and ergodic decomposition.}, where the action is now given by continuous maps.  Here is a typical such theorem:

\begin{theorem}[Continuous model for countable abstract systems] \label{countable-model} Let $\Gamma$ be a group, and let $X_\ProbAlg$ be a probability algebra.  Assume furthermore:
\begin{itemize}
\item[(a)] $\Gamma$ is at most countable.
\item[(b)] The $\sigma$-complete Boolean algebra associated to $X_\ProbAlg$ is separable.
\end{itemize}
Suppose that $\Gamma$ acts on $X_\ProbAlg$ by $\ProbAlg$-morphisms $T^\gamma_{X_\ProbAlg} \colon X_\ProbAlg \to X_\ProbAlg$ for $\gamma \in \Gamma$. Then there exists a Cantor probability space $X^* = (X^*, \mu_{X^*})$ (with $\mu_{X^*}$ a Borel probability measure) and an action of $\Gamma$ on $X^*$ by measure-preserving homeomorphisms $T^\gamma_{X^*} \colon X^* \to X^*$, and a $\ProbAlg$-isomorphism ${\mathfrak A} : X_\ProbAlg \to X^*_\ProbAlg$ such that
$$ (T^\gamma_{X^*})_\ProbAlg \circ {\mathfrak A} = {\mathfrak A} \circ T^\gamma_{X_\ProbAlg}$$
for $\gamma \in \Gamma$.
\end{theorem}

\begin{proof} This is a special case of \cite[Theorem 2.15]{glasner2015ergodic}.
\end{proof}

Informally, the above theorem asserts that under the ``countability'' hypotheses (a), (b), an abstract measure-preserving system can be \emph{modeled} by a concrete and continuous measure-preserving system (on a Cantor space).  The model provided by this theorem is not completely canonical; however, the full version of \cite[Theorem 2.15]{glasner2015ergodic} asserts, roughly speaking, that any pair of abstract measure-preserving systems $(X_\ProbAlg, T_{X_\ProbAlg})$,  $(Y_\ProbAlg, T_{Y_\ProbAlg})$ connected by a factor map $\pi \colon X_\ProbAlg \to Y_\ProbAlg$ can be \emph{simultaneously} modeled by compatible continuous models.  We refer the reader to \cite{glasner2015ergodic} for a more precise statement.  In the case when $X_\ProbAlg,Y_\ProbAlg$ come from standard probability spaces $X, Y$ and $\Gamma$ is at most countable, one can also invoke a well known theorem of von Neumann \cite{vonneumann} to model the $\ProbAlg$-morphisms by concrete measurable maps on the indicated spaces $X$, $Y$; see, e.g., \cite[Theorem F.9]{EFHN}.

For applications to uncountable ergodic theory, it is desirable to remove countability hypotheses such as (a), (b) from the above type of theorem, and also make the model completely canonical.  This will be achieved in Section \ref{canonical-sec}, in which we construct a \emph{canonical model functor}
\begin{center}
\begin{tikzcd}
    \ProbAlg \arrow[r,"\Stone",  tail, two heads] & \CHProb 
\end{tikzcd}
\end{center}
that assigns to each probability algebra $X$ a pair $$\Stone(X)  = (\Stone(X)_\CH, \mu_{\Stone(X)})$$ consisting of
 a compact Hausdorff space $\Stone(X)_\CH$ equipped with a Radon probability measure $\mu_{\Stone(X)}$ that models $X$ in the sense that probability algebra $\Stone(X)_\ProbAlg$ is (naturally) isomorphic to $X$, such that every $\ProbAlg$-morphism $T \colon X \to Y$ is assigned a continuous measure-preserving map $\Stone(T) \colon \Stone(X) \to \Stone(Y)$ in a completely functorial fashion; see Theorem \ref{canon-model} for a precise statement.   The functor $\Stone$ turns out to be full and faithful, thus it identifies the category $\ProbAlg$ of probability algebras with a subcategory of the much more structured category $\CHProb$ of compact Hausdorff probability spaces.  The functoriality of $\Stone$ is convenient for ergodic theory applications, as it automatically allows one to transfer any dynamical structure on the $\ProbAlg$ algebras to their $\CHProb$ counterparts via the canonical model. For instance, we now have an uncountable version of Theorem \ref{countable-model}:

\begin{theorem}[Continuous model for uncountable abstract systems] \label{uncountable-model} Let $\Gamma$ be a group, and let $X_\ProbAlg$ be a probability algebra.  
Suppose that $\Gamma$ acts on $X_\ProbAlg$ by $\ProbAlg$-morphisms $T^\gamma_{X_\ProbAlg} \colon X_\ProbAlg \to X_\ProbAlg$ for $\gamma \in \Gamma$. Then there is action of $\Gamma$ on $\Stone(X)$ by measure-preserving homeomorphisms $\Stone(T^\gamma_X) \colon \Stone(X) \to \Stone(X)$ and a $\ProbAlg$-isomorphism ${\mathfrak A} \colon X_\ProbAlg \to \Stone(X)_\ProbAlg$ such that
$$ \Stone(T^\gamma_{X})_\ProbAlg \circ {\mathfrak A} = {\mathfrak A} \circ T^\gamma_{X_\ProbAlg}$$
for $\gamma \in \Gamma$.
\end{theorem}

There are several ways to construct the canonical model $\Stone$, but the easiest way to proceed is via Riesz duality, and specifically to set
$$ \Stone(X) \coloneqq \Riesz( \Linfty(X)_{\CStarAlgUnitTrace} )$$
where $\Linfty(X)_{\CStarAlgUnitTrace}$ is the space of bounded (abstractly) measurable functions on the probability algebra $X$, viewed as a unital tracial commutative $C^*$-algebra (i.e., a $\CStarAlgUnitTrace$-algebra).  One can view this construction in terms of a further duality, namely a ``probability duality'' between the category $\ProbAlg$ of probability algebras and the category $\vonNeumann$ of commutative tracial von Neumann algebras, with the functor $\Linfty \colon \ProbAlg \to \vonNeumannop$ being one of the two functors witnessing this equivalence of categories; see Figure \ref{fig:canon}.  Versions of this construction have implicitly appeared in the literature in several places \cite{segal}, \cite{DNP}, \cite{ellis}, \cite[\S 12.3, 13.4]{EFHN}, with the model referred to as the \emph{Stone model} in \cite{EFHN}.  One can also proceed by applying Kakutani duality \cite{kakutani-l} to the Banach lattice $L^1(X)$ rather than Gelfand duality to the $C^*$ algebra $\Linfty(X)$; see \cite{dieudonne-counter}.  In \cite{halmos-dieudonne} the canonical model $\Stone(X)$ is referred to as the \emph{Kakutani space} of $X$.

  In Section \ref{loomis-sec} we will also give an equivalent alternate construction of $\Stone$ using the Loomis--Sikorski theorem (which can be viewed as an instance of Stone duality rather than Gelfand or Riesz duality).  A version of this alternate construction also implicitly appears in \cite{fremlinvol3}, \cite{doob-ratio}.  The canonical model functor $\Stone$ also obeys certain universality properties analogous to those enjoyed by the Stone--{\v C}ech functor $\beta$; see Propositions \ref{canon-model}, \ref{factor-of-stone}.
  
 \begin{remark} The Banach spaces $\Linfty(X)$ are almost never separable, and so the canonical model spaces $\Stone(X)$ are also almost never separable, even when the original probability algebra $X$ is separable.  As such, the canonical model can only be constructed in this uncountable framework, and thus presents an \emph{advantage} of this framework over the more traditional countable setting of ergodic theory (even if one was initially only interested in separable spaces), in analogy to how the Stone--{\v C}ech compactification can only be applied in similarly ``uncountable'' frameworks in which the topological spaces one works with are not required to obey any separability, metrizability or countability axioms.
 \end{remark}

In Theorem \ref{countable-model}, the model spaces $X^*$ were not arbitrary topological spaces, but had the structure of a Cantor space.  In a similar vein, the model spaces $\Stone(X)$ constructed by our canonical model have the structure of an (extremally disconnected) Stone space (also called \emph{Stonean spaces}), and furthermore enjoy a remarkable property which we call the \emph{strong Lusin property}: every bounded Baire-measurable function is equal \emph{almost everywhere} (as opposed to merely outside of a set of small measure) to a \emph{unique} continuous function.  See Proposition \ref{model-basic}.  Also, it turns out that the null Baire-sets of $\Stone(X)$ are precisely the Baire-meager sets, see Remark \ref{lusin-equiv}.

\subsection{Canonical disintegration}

One application of the canonical model functor $\Stone$ is to provide a canonical and functorial way to disintegrate a probability measure with respect to a factor map.  Disintegration theorems for measures go back to the work of Rohklin \cite{rohlin}.  There are many arrangements of this theorem; we present here one from \cite{simmons}.  See also \cite[Theorem 5.8]{furstenberg2014recurrence} or \cite[Th. 5]{dieudonne-counter} for a similar statement.

\begin{theorem}[Rohklin disintegration theorem]  Let $(X,\mu_X)$ and $(Y,\mu_Y)$ be probability spaces, and let $\pi \colon X \to Y$ be a measurable map such that $\pi_* \mu_X = \mu_Y$.  Assume furthermore:
\begin{itemize}
\item[(a)] $X$ is universally measurable and $\mu$ is a Borel measure.
\item[(b)] There is a measurable injective map from $Y$ into a standard Borel space.
\end{itemize}
Then for $\mu_Y$-almost every $y \in Y$ one can find a Borel probability measure $\mu_y$ on $\pi^{-1}(\{y\})$ such that one has the identity
$$ \int_X f(x) g(\pi(x))\ d\mu_X(x) = \int_Y \left(\int_X f(x)\ d\mu_y(x)\right) g(y)\ d\mu_Y(y)$$
for all bounded measurable $f \colon X \to \C$, $g \colon Y \to \C$ (in particular the integral $\int_X f(x)\ d\mu_y(x)$ is a measurable function of $y$).  Furthermore, this assignment $y \mapsto \mu_y$ is unique up to $\mu_Y$-almost everywhere equivalence.
\end{theorem}

Continuing the spirit of the ``uncountable'' approach to ergodic theory, we would like to remove hypotheses such as (a) and (b) from this theorem.  As stated, the theorem can fail without these hypotheses; see, e.g., \cite{dieudonne-counter}, \cite[p. 624]{doob}, \cite[p. 210]{halmos-measure-theory}.  However, we can recover a disintegration (with additional uniqueness and topological properties) as long as we pass to the canonical model to perform the disintegration:

\begin{theorem}[Canonical disintegration]\label{canon-disint}  Let $X, Y$ be $\ProbAlg$-spaces, and let $\pi \colon X \to Y$ be a $\ProbAlg$-morphism.  Then there is a unique Radon probability measure $\mu_y$ on $\Stone(X)_\CH$ for each $y \in \Stone(Y)$ which depends continuously on $y$ in the vague topology in the sense that $y \mapsto \int_{\Stone(X)_\CH} f\ d\mu_y$ is continuous for every $f \in \CFunc(\Stone(X))$, and such that
\begin{equation}\label{disint-form}
\int_{\Stone(X)} f(x) g(\Stone(\pi)(x))\ d\mu_{\Stone(X)}(x) = \int_{\Stone(Y)} \left(\int_{\Stone(X)_\CH} f\ d\mu_y\right) g\ d\mu_{\Stone(Y)}
\end{equation}
for all $f \in \CFunc(\Stone(X))$, $g \in \CFunc(\Stone(Y))$.  Furthermore, for each $y \in \Stone(Y)$, $\mu_y$ is supported on the compact set $\Stone(\pi)^{-1}(\{y\})$, in the sense that $\mu_Y(E)=0$ whenever $E$ is a measurable set disjoint from $\Stone(\pi)^{-1}(\{y\})$. (Note that this conclusion does \emph{not} require the fibers $\Stone(\pi)^{-1}(\{y\})$ to be measurable.)
\end{theorem}

We prove this theorem in Section \ref{disint-sec}.   Among other things, this disintegration gives a ``concrete'' way to construct relatively independent products of $\ProbAlg$-spaces, in the spirit of \cite[\S 5.5]{furstenberg2014recurrence}; see Theorem \ref{rel-prod}.  The use of the canonical model to perform canonical disintegration also appears in 
\cite[\S 3]{halmos-dieudonne}, \cite[Theorem 2.3]{ellis}.
In \cite[\S 4]{halmos-dieudonne} it is claimed that the canonical disintegration yields an ergodic decomposition for abstract measure-preserving actions of arbitrary groups $\Gamma$, but this claim is incorrect as stated; we provide a counterexample in Appendix \ref{halmos-sec}.

\subsection{Connection to the Loomis--Sikorski theorem}

In Section \ref{loomis-sec} we show that if one ``removes'' the probability measures from the canonical model functor
\begin{center}
\begin{tikzcd}
    \ProbAlg \arrow[r,"\Stone",  tail, two heads] & \CHProb 
\end{tikzcd}
\end{center}
one obtains an analogous \emph{Loomis--Sikorski functor}
\begin{center}
\begin{tikzcd}
    \AbsMes \arrow[r,"\Loomis",  tail, two heads] & \CHDelete 
\end{tikzcd}
\end{center}
that takes an abstract measurable space $X$, and obtains a concrete model $\Loomis(X)$ of this space, which has the structure of a compact Hausdorff space equipped with a null ideal of the Baire $\sigma$-algebra.  The original $\SigmaAlg$-algebra $X$ is then naturally isomorphic to the Baire $\sigma$-algebra of $\Loomis(X)$, quotiented by the given null ideal.  In fact, $\Loomis(X)$ has the structure of a special type of Stone space which we call a $\StoneCatSigma$-space (a Stone space in which every Baire set\footnote{In descriptive set theory, a \emph{Baire-measurable} set is one that has the property of Baire (that is, differing from an open set by a meager set). To avoid potential confusion, we stress that when writing Baire set (or Baire meager set), then we always mean Baire measurable sets (or Baire measurable sets that are also meager) in the sense introduced in Section \ref{baire-sec}, rather than in the descriptive set theory sense (which we will not use here).} differs from a clopen set by a Baire-meager set), and the null ideal is also the ideal of Baire-meager sets.  The existence and basic properties of the functor $\Loomis$ is a fully functorial form of the Loomis--Sikorski theorem \cite{loomis1947, sikorski}, and we establish it using Stone dualities relating the categories $\Bool, \SigmaAlg$ of Boolean algebras and $\sigma$-complete Boolean algebras with the categories $\StoneCat, \StoneCatSigma$ of Stone spaces respectively; see Figure \ref{fig:altcon} for how these dualities relate to the functors $\Loomis$ and $\Stone$.  We then use $\Loomis$ to give an alternate construction of $\Stone$ that proceeds via Stone duality instead of Riesz duality.  We remark that some very closely related constructions also appear in \cite{fremlinvol3}, although our more ``Baire-centric'' presentation places a greater emphasis on the role of the Baire $\sigma$-algebra, and also ties the construction to the operator-algebraic formalism of $C^*$-algebras and von Neumann algebras rather than the order-theoretic formalism of Riesz algebras.

As a byproduct of this analysis we are also able to clarify the nature of categorical products in the categories $\AbsMes, \StoneCatSigma, \CHDelete$ (or of the categorical coproduct in $\SigmaAlg$), in particular revealing some subtle differences between the product $\prod^\AbsMes$ on abstract measurable spaces, and the product $\prod^\ConcMes$ of concrete measurable spaces (the former being a strict subset of the latter in a category-theoretic sense).  As one manifestation of this distinction, we establish in Theorem \ref{kolmo} a version of the Kolmogorov extension theorem in the abstract measurable category $\AbsMes$ that does not require any regularity hypotheses on the factor spaces, in contrast with the classical version of this theorem in the concrete measurable category $\ConcMes$, which fails in general unless one imposes hypotheses such as the standard Borel property on the factor spaces.

Stone duality is the restriction of a more general duality between (coherent) locales and distributive lattices (locales are the basic structures in point-free topology). 
We will not pursue such generality here since the cost of introducing more categories to the already long list of categories that are employed in this paper would outweigh the restricted benefits such generality would have to our aim of connecting the separate categorical aspects of measure-theoretic dynamics\footnote{This would also duplicate existing efforts, as the recent paper \cite{pavlov2020} already provides a detailed overview on Stone-type dualities with a list of relevant references (to which we refer the interested reader).}.   
However, there are many parallels between our efforts to develop solid foundations for point-free measure theory (with a view towards ergodic theory) and locale theory.  
We would like to point out several existing parallel efforts in the same direction.  
Fremlin's treatise \cite{fremlinvol3} develops systematically (abstract) measure theory on general measure algebras. 
Pavlov establishes a Gelfand-type duality for commutative (not necessarily tracial) von Neumann algebras in \cite{pavlov2020}. 
The main result of \cite{pavlov2020} proves that the following five categories are equivalent: (1) the opposite category of commutative von Neumann algebras; (2) compact strictly localizable enhanced measurable spaces; (3) measurable
locales; (4) hyperstonean locales; (5) hyperstonean spaces. This provides a ``von Neumann duality'' which is to measurable spaces as Gelfand dualities are to locally compact spaces, probability dualities are to probability algebras, and Riesz dualities are to locally compact probability spaces. However, we will not discuss further these dualities here. 
We point out that deriving the equivalence between the categories in (1) and (2) is essentially\footnote{Pavlov's notion of a "measurable locale" is closely related to our notion of an "abstract measurable space" with the difference that Pavlov requires completeness of the underlying $\SigmaAlg$-algebra and quotients out by the "universal" $\sigma$-ideal.} identical to our second construction of the canonical model with the help of Stone duality and the Loomis--Sikorski theorem. 

We thank the anonymous referee for pointing out the following interesting references at the intersection of point-free topology and measure theory \cite{coquand,chen,henry,simpson,vickers,vickers1} and  continuous logic \cite{yaacov,yaacov1,ibarlucia}. We hope to investigate further these connections in future work.

\subsection{Casting conventions}\label{sec-cast}

In this section, we introduce the concept of a \emph{casting functor}, which we use as a device to formally identify objects in different categories.  It is inspired by the use of casting conventions in programming languages, and is conventional in nature, that is to say  our choice of functors which we declare to be casting primarily serves us to organize better the numerous categories related to each other in this paper. As for our category-theoretical notation, we refer the reader to Appendix \ref{category-appendix}. 

A \emph{diagram of categories}, such as the one depicted in the various figures in this paper, is a collection of categories together with some functors between these categories.  Such a diagram is \emph{commutative} if for any pair of categories $\Cat, \Cat'$ in the diagram, there is at most one\footnote{Strictly speaking, it would be more natural from a category-theoretic perspective to require the functor from $\Cat$ to $\Cat'$ to merely be unique up to natural isomorphisms (which are also required to satisfy a coherence condition with respect to the other functors in the diagram); however we shall abuse notation in this paper by identifying various "canonically isomorphic" objects in order not to deal with this ambiguity.} functor from $\Cat$ to $\Cat'$ that can be obtained by composing finitely many of the functors in the diagram.  The diagrams in our figures are not always commutative, but the subdiagram consisting of just the functors depicted by {\blue blue arrows} will always be commutative.  We exploit this commutativity via the following useful notational convention:

\begin{definition}[Casting operators]\label{cast}  Define a \emph{casting functor} (or \emph{casting operator}) to be any one of the following functors:
\begin{itemize}
    \item[(i)]  A functor depicted in {\blue blue} in any of the diagrams of categories in this paper.
    \item[(ii)]  The identity functor $\ident_\Cat$ on any category $\Cat$. 
    \item[(iii)] The vertex functors ${ \Vertex} \colon (\Cat \downarrow X) { \to} \Cat$, ${ \Vertex} \colon (X \downarrow \Cat) \to \Cat$ that map a $(\Cat \downarrow X)$-object $Y\to X$ or $(X\downarrow \Cat)$-object $X\to Y$ to its vertex object $X_\Cat$, and any morphism $f$ in $(\Cat \downarrow X)$ or $(X \downarrow \Cat)$ to the corresponding $\Cat$-morphism $f_\Cat$.
    \item[(iv)] Any finite composition of functors from the above list.
\end{itemize}
The casting functors in this paper are chosen to form a commutative diagram; thus for any two categories $\Cat,\Cat'$ there is at most one casting functor ${\Cast}_{\Cat \to \Cat'} \colon \Cat {\to} \Cat'$ from the former to the latter.  If such a casting functor exists, we say that $\Cat$ can be \emph{casted} to $\Cat'$, and for any $\Cat$-object $X = X_\Cat$ we define the \emph{cast} of $X$ to $\Cat'$ to be the corresponding object in $\Cat'$, we write $X_{\Cat'}$ for ${\Cast}_{\Cat \to \Cat'}(X)$, and refer to $X_{\Cat'}$ as the \emph{cast} of $X$ to $\Cat'$ (and $X_\Cat$ as a \emph{promotion} of $X_{\Cat'}$ to $\Cat$).  We may cast or promote morphisms in $\Cat$, different notions of products and coproducts in $\Cat$ to $\Cat'$ in a similar fashion.  Thus for instance a $\Cat'$-morphism has at most one promotion to a $\Cat$-morphism if the casting functor is faithful.  (Informally, one should view the $\Cat'$-cast or $\Cat'$-promotion of a mathematical structure associated to $\Cat$ as the ``obvious'' corresponding $\Cat'$-structure associated to the $\Cat$-structure, with the choice of casting functors in the diagrams in this paper formalizing what ``obvious'' means.)

When a mathematical expression or statement requires an object or morphism to lie in $\Cat$, but an object or morphism in another category $\Cat'$ appears in its place, then it is understood that a casting operator from $\Cat'$ to $\Cat$ is automatically applied.  In particular, if a statement is said to ``hold in $\Cat$'' or ``be interpreted in $\Cat$'', or if an object or morphism is to be understood as a $\Cat$-object or a $\Cat$-morphism, then the appropriate casting operators to $\Cat$ are understood to be automatically applied.  We will sometimes write $X =_\Cat Y$ to denote the assertion that an identity $X=Y$ holds in $\Cat$.  

If one composes a named functor $\Func$ on the left or right (or both) with forgetful casting functors, the resulting functor will also be called $\Func$ when there is no chance of confusion (or if the ambiguity is irrelevant). 
\end{definition}

There is significant overlap\footnote{Indeed, from the perspective of Definition \ref{cast}, a common ``abuse of notation'' in mathematics is to interpret every forgetful functor as a casting functor.} between the concepts of a casting functor and a forgetful functor, but with the conventions we adopt in this paper, not every casting functor is forgetful, and not every forgetful functor is casting.  Similarly, most of the casting functors we will use are faithful in nature, but there is a key exception, namely the abstraction functors $\Abs$ that map concrete spaces to their abstract counterparts.

The following examples will help illustrate this casting convention (the definitions of the categories are given in the body of the paper, see Table \ref{tab:categories} for the location of these definitions). 

\begin{example}\
\begin{itemize}
\item[(i)] If $X = X_\CHProb = (X_\Set, {\mathcal F}_X, \mu_X)$ is a compact Hausdorff probability space, then $X_\CH = (X_\Set, {\mathcal F}_X)$ is the associated compact Hausdorff space, $X_\ConcMes = \BaireFunc(X_\CH)=(X_\Set, \Sigma_X)$ is the associated measurable space, $X_\ConcProb = (X_\ConcMes, \mu_X)$ is the associated concrete probability space, $\Sigma_X = \Baire(X)$ is the associated $\SigmaAlg$-algebra, $X_\ProbAlg = (\Baire(X)/_{{\mathcal N}_X}, \mu_X)$ is the associated probability algebra, 
and $X_\Set$ is the set $X$ with no additional structures. 
\item[(ii)]  If $T \colon X \to Y$ is a $\ConcProb$-morphism, then $T_{\SigmaAlg} : Y_{\SigmaAlg} \to X_{\SigmaAlg}$ is the associated pullback map.
\item[(iii)] If $X$ is a $\CH$-space and $Y$ is an $\AbsProb$-space, then $X \times^{\AbsMes} K$ denotes the abstract measurable space $X_\AbsMes \times^{\AbsMes} Y_\AbsMes$.
\item[(iv)]  If $X \in \ConcMes$ and $Y \in \CH$, we write ``$T \colon X \to Y$ is a $\AbsMes$-morphism from $X$ to $Y$'' as shorthand for ``$T \colon X_\AbsMes \to Y_\AbsMes$ is an $\AbsMes$-morphism from $X_\AbsMes$ to $Y_\AbsMes$''.
\item[(v)]  If $f \colon X \to Y$ is an $\AbsProb$-morphism, we write ``$f$ is an $\AbsMes$-epimorphism'' as shorthand for ``$f_\AbsMes \colon X_\AbsMes \to Y_\AbsMes$ is an $\AbsMes$-epimorphism''. 
\item[(vi)] Let $X = (X,\X,\mu) \in \ConcProb$, $Y = (Y,\Y,\nu) \in \ConcProb$ be concrete probability spaces, and let $f \colon X \to Y$ be a measurable map.  Then $f$ is a $\ConcMes$-morphism, and can be promoted to a $\ConcProb$-morphism if and only if $f_* \mu = \nu$.
\item[(vii)] If $X \in \ConcProb$, $Y, Z \in \CH$, $f \colon X_\ConcMes \to Y_\ConcMes$ is a $\ConcMes$-morphism, $\pi \colon Y \to Z$ is a $\CH$-morphism, and $g \colon X_\AbsProb \to Z_\AbsProb$ is an $\AbsProb$-morphism, we say that the identity $\pi \circ f = g$ holds in $\AbsMes$, and write $\pi \circ f =_\AbsMes g$, if $\pi_\AbsMes \circ f_\AbsMes = g_\AbsMes$. 
\item[(viii)]  If $f_1, f_2 \colon X \to Y$ are $\ConcProb$-morphisms that agree almost everywhere, then they agree in $\ProbAlg$: $f_1 =_{\ProbAlg} f_2$, that is to say the $\ProbAlg$-morphisms $(f_1)_\ProbAlg \colon X_\ProbAlg \to Y_\ProbAlg$ and $(f_2)_{\ProbAlg} \colon X_\ProbAlg \to Y_\ProbAlg$ agree.  (The converse implication can fail; see \cite[Examples 5.1, 5.2]{jt19}.) 
\end{itemize}
\end{example}

\section{Gelfand dualities}\label{gelfand-sec}

In this section we construct the various categories and functors in Figure \ref{fig:gelfand-dual}, and verify that the diagram commutes up to natural isomorphism (mostly by appealing to existing literature).  We begin by constructing the Alexandroff compactification (also known as the \emph{one-point compactification}), which in our formalism is given by a functor $\Alex \colon \LCHpr \to \CHpt$.

\begin{definition}[Alexandroff compactification]\label{def-alex}\ 
\begin{itemize}
\item[(i)] An \emph{$\LCH$-space} is a locally compact Hausdorff space $X = (X_\Set, {\mathcal F}_X)$.  An \emph{$\LCH$-morphism} $f \colon X \to Y$ between $\LCH$-spaces $X = (X_\Set, {\mathcal F}_X)$, $Y = (Y_\Set, {\mathcal Y})$ is a continuous function $f_\Set \colon X_\Set \to Y_\Set$, with the usual $\Set$-composition law. 
\item[(ii)] $\LCHpr$ is the subcategory of $\LCH$ consisting of the same class of spaces (thus every $\LCH$-space is an $\LCHpr$-space and vice versa), and the \emph{$\LCHpr$-morphisms} consisting of those $\LCH$-morphisms $T \colon X \to Y$ which are proper (thus the pullback $T^*(K) \coloneqq T^{-1}_\Set(K)$ of any compact subset $K$ of $Y$ is compact in $X$).
\item[(iii)] $\CHpt$ is the coslice category of $\CH$ with respect to a point $\mathrm{pt}$, as defined in Definition \ref{def-slice}.  Thus, a \emph{$\CHpt$-space} $X = (X_\CH, *)$ is a $\CH$-space $X_\CH$ equipped with a distinguished $\CH$-morphism $* \colon \mathrm{pt} \to \CH$ (by abuse of notation we use $*$ to refer simultaneously to all distinguished morphisms of all $\CHpt$-spaces).  A \emph{$\CHpt$-morphism} $T \colon X \to Y$ between $\CHpt$-spaces $X,Y$ is a $\CH$-morphism $T \colon X_\CH \to Y_\CH$ such that $T \circ * = *$, with the usual $\Set$-composition law.
\item[(iv)] If $X$ is an $\LCHpr$-space, the $\CHpt$-space $\Alex(X)$ is defined as the disjoint union $\Alex(X)_\Set \coloneqq X_\Set \sqcup \{\infty\}$, with distinguished morphism $*$ mapping $\mathrm{pt}$ to $\infty$, equipped with the topology ${\mathcal F}_{\Alex(X)}$ consisting of sets that are either open in $X$, or are the complement in $\Alex(X)_\Set$ of a compact set in $X$.  If $T \colon X \to Y$ is an $\LCHpr$-morphism, then $\Alex(T) \colon \Alex(X) \to \Alex(Y)$ is the map defined by $\Alex(T)_\Set(x)\coloneqq T(x)$ when $x \in X$ and $\Alex(T)_\Set(\infty) \coloneqq \infty$.
\end{itemize}
\end{definition}

It is clear that $\LCH, \LCHpr$ are categories with a faithful functor from $\LCHpr$ to $\LCH$.  One easily verifies that if $T \colon X \to Y$ is an $\LCHpr$-morphism then $\Alex(T) \colon \Alex(X) \to \Alex(Y)$ is continuous; from this it is not difficult to verify that $\Alex \colon \LCHpr \to \CHpt$ is a faithful functor between the indicated categories.  Note that without the properness hypothesis in the definition of $\LCHpr$-morphism, $\Alex$ would fail to be a functor taking values in $\CHpt$.  For instance, formally applying $\Alex$ to the non-proper zero map $0 \colon \R \to \R$ would lead to a discontinuous map $\Alex(0) \colon \R \sqcup \{\infty\} \to \R \sqcup \{\infty\}$ which mapped all real numbers $\R$ to $0$ and mapped $\infty$ to $\infty$, which is not a $\CHpt$-morphism.  One also observes an obvious faithful forgetful functor from $\CHpt$ to $\CH$, and obvious forgetful faithful and full functors from $\CH$ to $\LCHpr$ and from $\LCHpr$ to $\LCH$. 

\begin{remark}  The functor $\Alex$ is not full.  For instance, the $\CHpt$-morphism from $\Alex(\R)$ to $\Alex(\R)$ that maps all of $\Alex(\R) = \R \sqcup \{\infty\}$ to $\infty$ does not arise from applying $\Alex$ to an $\LCHpr$-morphism.  
\end{remark}

Now we define the additional $C^*$-algebra categories and their Gelfand dualities\footnote{We refer the interested reader to \cite[Chapter 1]{folland-harmonic-analysis} for the basic background in operator algebras required in this paper.}.

\begin{definition}[Gelfand dualities]\label{def-gelfand}\ 
\begin{itemize}
  \item[(i)]  If ${\mathcal A}$ is a commutative $C^*$-algebra, we use $\Mult({\mathcal A})$ to denote its multiplier algebra, that is to say the space of pairs $(L,R)$ of bounded operators on ${\mathcal A}$ obeying the double centralizer condition $aL(b)=R(a)b$ for all $a,b \in {\mathcal A}$.  As is well-known, this has the structure of a $\CStarAlgUnit$-algebra.  If $f = (L,R) \in \Mult({\mathcal A})$, we write $fb$ for $L(b)$ and $af$ for $R(a)$, thus $(af)b=a(fb)$.  Note that ${\mathcal A}$ can be identified with a subalgebra of $\Mult({\mathcal A})$.
  \item[(ii)]  A \emph{$\CStarAlgNd$-algebra} is a commutative $C^*$-algebra.  A \emph{$\CStarAlgNd$-morphism} $\Phi \colon {\mathcal A} \to {\mathcal B}$ between $\CStarAlgNd$-algebras ${\mathcal A}, {\mathcal B}$ is a $*$-homomorphism $\Phi \colon {\mathcal A} \to {\mathcal B}$ which is \emph{non-degenerate}\footnote{Such morphisms were referred to as \emph{proper homomorphisms} in \cite{pedersen}.} in the sense that the linear span of $\Phi({\mathcal A}){\mathcal B} \coloneqq \{ \Phi(a) b \colon a \in {\mathcal A}, b \in {\mathcal B}\}$ is dense in ${\mathcal B}$.  Composition is given by the usual $\Set$-composition.
  \item[(iii)]  A \emph{$\CStarAlgMult$-algebra} is a commutative $C^*$-algebra.  
  A \emph{$\CStarAlgMult$-morphism} $\Phi \colon {\mathcal A} \to {\mathcal B}$ between $\CStarAlgMult$-spaces ${\mathcal A}, {\mathcal B}$ is a $*$- homomorphism\footnote{In particular, we caution that $\CStarAlgMult$ is \emph{not} a concrete category, because the morphisms $\Phi \colon {\mathcal A} \to {\mathcal B}$ are \emph{not} described by concrete functions ${\mathcal A}$ to ${\mathcal B}$, but only by functions from ${\mathcal A}$ to the larger set $\Mult({\mathcal B})$. Similarly, the composition law for these morphisms is \emph{not} the usual $\Set$-composition, but is instead defined indirectly by requiring \eqref{mult-ident}.} $\tilde \Phi \colon {\mathcal A} \to \Mult({\mathcal B})$ which is non-degenerate in the sense that the linear span of $\tilde \Phi({\mathcal A}){\mathcal B} \coloneqq \{ \tilde \Phi(a) b \colon a \in {\mathcal A}, b \in {\mathcal B}\}$ is dense in 
  ${\mathcal B}$.  It is known (see, e.g., \cite[Corollary 2.51]{Raeburn-Williams}) that $\tilde \Phi$ can be uniquely extended to a $\CStarAlgUnit$-morphism $\Mult(\Phi) \colon \Mult({\mathcal A}) \to \Mult({\mathcal B})$.  The composition $\Psi \circ \Phi \colon {\mathcal A} \to {\mathcal C}$ of two $\CStarAlgMult$-morphisms $\Phi \colon {\mathcal A} \to {\mathcal B}$, $\Psi \colon {\mathcal B} \to {\mathcal C}$ is then defined to be the unique $\CStarAlgMult$-morphism for which
  \begin{equation}\label{mult-ident}
  \Mult(\Psi \circ \Phi) = \Mult(\Psi) \circ \Mult(\Phi).
  \end{equation}
  (The existence and uniqueness of this morphism follows from \cite[Proposition 1]{huef}.)
  \item[(iv)]  $\CStarAlgUnitInf$ is the slice category of $\CStarAlgUnit$ over $\C$, as defined in Definition \ref{def-slice}.  Thus, a \emph{$\CStarAlgUnitInf$-algebra} is a $\CStarAlgUnit$-algebra ${\mathcal A}_{\CStarAlgUnit}$ equipped with a $\CStarAlgUnit$ -morphism $* \colon {\mathcal A}_{\CStarAlgUnit} \to \C$.  A \emph{$\CStarAlgUnitInf$-morphism} $\Phi \colon {\mathcal A} \to {\mathcal B}$ between $\CStarAlgUnitInf$-algebras ${\mathcal A}, {\mathcal B}$ is a $\CStarAlgUnit$-morphism $\Phi_\CStarAlgUnit \colon {\mathcal A}_\CStarAlgUnit \to {\mathcal B}_\CStarAlgUnit$ such that $* \circ \Phi_\CStarAlgUnit = *$.
  \item[(v)] The functors $\CFunc \colon \CH \to \CStarAlgUnit$, $\Spec \colon \CStarAlgUnit \to \CH$ induce functors $\CFunc \colon \CHpt \to \CStarAlgUnitInf$, $\Spec \colon \CStarAlgUnitInf \to \CHpt$ as per Example \ref{range-ex} (identifying $\C$ with $\CFunc(\mathrm{pt})$ and $\mathrm{pt}$ with $\Spec(\C)$).
  \item[(vi)]  If $X$ is an $\LCHpr$-space, $\CoFunc(X)$ is the $\CStarAlgNd$-algebra of continuous functions $f \colon X \to \C$ that vanish at infinity (i.e., for any $\eps>0$ there is a compact subset $K$ of $X$ outside of which one has $|f| \leq \eps$).  If $T \colon X \to Y$ is an $\LCHpr$-morphism, $\CoFunc(T) \colon \CoFunc(Y) \to \CoFunc(X)$ is the Koopman operator $\CoFunc(T)(f) \coloneqq f \circ T$; this is easily verified to be a $\CStarAlgNd$-morphism (note it is essential here that $T$ is proper).
  \item[(vii)]  If $X$ is an $\LCH$-space, $\CoFunc(X)$ is the $\CStarAlgMult$-algebra of continuous functions $f \colon X \to \C$ that vanish at infinity, and $\CbFunc(X)$ is the $\CStarAlgUnit$-algebra of bounded continuous functions $f \colon X \to \C$. Note that $\CbFunc(X)$ can be identified with $\Mult(\CoFunc(X))$.  If $T \colon X \to Y$ is an $\LCH$-morphism, we define $\CoFunc(T)(f) \colon \CoFunc(Y) \to \CoFunc(X)$ and $\CbFunc(T)(f) \colon \CbFunc(Y) \to \CbFunc(X)$ to be the Koopman operators $\CoFunc(T)(f) \coloneqq f \circ T$ for $f \in \CoFunc(Y)$ and $\CbFunc(T)(f) \coloneqq f \circ T$ for $f \in \CbFunc(Y)$ (here we use the identification of $\CbFunc(X)$ and $\Mult(\CoFunc(X))$ to define $\CoFunc(T)$).
  \item[(viii)]  If ${\mathcal A}$ is a $\CStarAlgNd$-algebra, $\Spec({\mathcal A})\coloneqq \Hom_{\CStarAlgNd}({\mathcal A} \to \C)$ is the set of $\CStarAlgNd$-morphisms from ${\mathcal A}$ to $C$ (i.e., non-zero *-homomorphisms from ${\mathcal A}$ to $\C$), with the topology induced from $\C^{\mathcal A}$, and with $\Spec(T) \lambda = \lambda \circ T$ for any $\CStarAlgNd$-morphisms $T \colon {\mathcal A} \to {\mathcal B}$ and $\lambda \colon {\mathcal B} \to \C$.  Similarly with $\CStarAlgNd$ replaced by $\CStarAlgMult$ throughout.
  \item[(ix)]  If ${\mathcal A} \in \CStarAlgNd$, we define $\Unit({\mathcal A}) \in \CStarAlgUnitInf$ to be the $\CStarAlgUnit$-algebra $\Unit({\mathcal A})_{\CStarAlgUnit} \coloneqq {\mathcal A} \oplus \C$ of formal sums $a+c1$ with $a \in {\mathcal A}, c \in \C$, together with the coordinate $\CStarAlgUnit$-morphism from ${\mathcal A} \oplus \C$ to $\C$ that maps $a+c1$ to $c$.  If $T \colon {\mathcal A} \to {\mathcal B}$ is a $\CStarAlgNd$-morphism, we define the $\CStarAlgUnitInf$-morphism $\Unit(T) \colon \Unit({\mathcal A}) \to \Unit({\mathcal B})$ by defining $\Unit(T)_\CStarAlgUnit(a + c1) \coloneqq Ta+c1$ for $a \in {\mathcal A}$, $c \in \C$.
  \item[(x)]  We define the \emph{Stone--{\v C}ech compactification functor} $\beta \colon \LCH \to \CH$ to be the functor $\beta \coloneqq \Spec \circ \CbFunc$.
  \item[(xi)] We have obvious forgetful functors from $\CStarAlgUnitInf$ to $\CStarAlgUnit$, from $\CStarAlgUnit$ to $\CStarAlgNd$ (note that any unital *-homomorphism is automatically non-degenerate), and from $\CStarAlgNd$ to $\CStarAlgMult$.
\end{itemize}
\end{definition}

We then have

\begin{theorem}[Gelfand dualities]\label{gelfand-dualities}  The categories in Figures \ref{fig:gelfand-dual}, \ref{fig:gelfand-forget} are indeed categories, and the functors in these figures are indeed functors between the indicated categories, with the indicated faithfulness and fullness properties.  Furthermore, both of these diagrams commute up to natural isomorphisms. (In particular, each pair of vertical functors generates a duality of categories.)
\end{theorem}

\begin{proof}  The verification of the category and functor axioms are routine, as are the faithfulness for the horizontal functors in both figures.   
As mentioned in the introduction, the duality of categories between $\CH$ and $\CStarAlgUnit$ is proven in \cite{negrepontis} (or  \cite[Theorem 1.20]{folland-harmonic-analysis}), which then implies the duality of categories between $\CHpt$ and $\CStarAlgUnitInf$ by abstract nonsense (as well as the obvious identifications  $\mathrm{pt} \equiv \Spec(\C)$ and $\C \equiv \CFunc(\mathrm{pt})$).  The duality of categories between $\LCHpr$ and $\CStarAlgNd$ can be found in \cite{pedersen} or \cite[Theorem 1.31]{folland-harmonic-analysis}.  The duality of categories between $\LCH$ and $\CStarAlgMult$ is established in \cite[Theorem 2]{huef-raeburn-williams}. Note that these dualities ensure that the vertical functors in Figures \ref{fig:gelfand-dual}, \ref{fig:gelfand-forget} are full and faithful. 

The commutativity (up to natural isomorphisms) of the three squares in Figure \ref{fig:gelfand-forget} follows easily from the observation that $\CFunc(X) = \CoFunc(X)$ when $X$ is a $\CH$-space.  This also gives the commutativity of the middle square of Figure \ref{fig:gelfand-dual}.  For the commutativity of the functors in the right square of Figure \ref{fig:gelfand-dual}, one uses the definition of $\beta$ and the identification of $\CbFunc$ with $\Mult \circ \CoFunc$, which one can routinely verify to be a natural isomorphism.  The latter identification also ensures that $\CbFunc$ is faithful. 
Finally, to verify the commutativity of the left square of Figure \ref{fig:gelfand-dual} it suffices to establish a natural isomorphism between $\CFunc \circ \Alex \colon \CH \to \CStarAlgUnitInfop$ and $\Unit \circ \CoFunc \colon \CH \to \CStarAlgUnitInfop$, but this follows easily after noting that every function $f \in \CFunc(\Alex(X))$ for a $\CH$-space $X$ can be uniquely expressed as $f = f' + c1$ for some $f' \in \CoFunc(X)$ (which we identify with an element of $\CFunc(\Alex(X))$ in the obvious fashion) and some $c \in \C$.
\end{proof}

\begin{example}  Let $0 \colon \N \to \N$ be the zero $\LCH$-morphism on the $\LCH$-space $\N = \{0,1,\dots\}$.  The $\CStarAlgMult$-morphism $\CoFunc(0) \colon \CoFunc(\N) \to  \CoFunc(\N)$ can be identified with the $*$-homomorphism $\widetilde{\CoFunc(0)} \colon \CoFunc(\N) \to \CbFunc(\N)$ (identifying $\CbFunc(\N)$ with the multiplier algebra of $\CoFunc(\N)$) defined by $\widetilde{\CoFunc(0)}(a) \coloneqq a(0) 1$ for any $a \in \CoFunc(\N)$.  Note that $\widetilde{\CoFunc(0)}$ does not take values in $\CoFunc(\N)$, which reflects the fact that $0$ is not a proper map and thus not an $\LCHpr$-morphism; it also reflects the non-concrete nature of the category $\CStarAlgMult$.  The $\CH$-morphism $\beta(0) \colon \beta \N \to \beta \N$ is the constant zero map, and the $\CStarAlgUnit$-morphism $\CbFunc(0) \colon \CbFunc(\N) \to \CbFunc(\N)$ is defined by $\CbFunc(0)(a) \coloneqq a(0) 1$ for any $a \in \CbFunc(\N)$. Note that $\CbFunc(\N)$ is also naturally isomorphic to $\CFunc(\beta \N)$.
\end{example}

\begin{remark}  The functors in Figure \ref{fig:gelfand-dual} do not all commute with the functors in Figure \ref{fig:gelfand-forget}, but are instead related to each other by various natural transformations.  For instance, there is a natural monomorphism from the identity functor $\ident_\LCH \colon \LCH \to \LCH$ to $\Forget_{\CH \to \LCH} \circ \beta \colon \LCH \to \LCH$, reflecting the canonical inclusion of an $\LCH$-space $X$ in its Stone--{\v C}ech compactification $\beta X$.  Closely related to this is the well-known fact (see, e.g., \cite[Chapter 10]{walker2012stone}) that $\beta \colon \LCH \to \CH$ is left-adjoint to $\Forget_{\CH \to\LCH} \colon \CH \to \LCH$.
However, there is no such adjoint relationship for the Alexandroff compactification, as there are no $\LCHpr$-morphisms from non-compact spaces to compact spaces.  On the other hand, one can construct a natural epimorphism from the functor $\beta \circ \Forget_{\LCHpr \to \LCH} \colon \LCHpr \to \CH$ to the functor $\Forget_{\CHpt \to \CH} \circ \Alex \colon \LCHpr \to \CH$, reflecting the canonical projection from the Stone--{\v C}ech compactification to the Alexandroff compactification; we leave the details to the interested reader.
\end{remark}

\begin{remark}  Every $\LCH$-morphism $T \colon X \to Y$ between $\LCH$ spaces extends to a $\CH$-morphism $\beta(T) \colon \beta X \to \beta Y$ between the associated Stone--{\v C}ech compactifications, but not every such $\CH$-morphism from $\beta X$ to $\beta Y$ arises from an $\LCH$-morphism from $X$ to $Y$; that is to say, the functor $\beta$ is not full.  For instance, if $p$ is an element of $\beta \N \backslash \N$ (i.e., a non-principal ultrafilter) then the constant $\CH$-morphism from $\beta \N$ to $\beta \N$ that maps all elements of $\beta \N$ to $p$ does not arise from any $\LCH$-morphism on $\N$.  Applying Gelfand duality, we conclude that every $\CStarAlgMult$-morphism $\Phi \colon {\mathcal A} \to {\mathcal B}$ between $\CStarAlgMult$-algebras induces a $\CStarAlgUnit$-morphism $\Mult(\Phi) \colon \Mult({\mathcal A}) \to \Mult({\mathcal B})$ which extends the *-homomorphism $\tilde \Phi \colon {\mathcal A} \to \Mult({\mathcal B})$, but not every $\CStarAlgUnit$-morphism from $\Mult({\mathcal A})$ to $\Mult({\mathcal B})$ arises in this fashion (i.e., $\Mult$ is not full).  For instance, with $p$ as before, and identifying $\Mult(\CoFunc(\N))$ with $\CbFunc(\N)$, the map $\Psi \colon \CbFunc(\N) \to \CbFunc(\N)$ defined by
$$ \Psi(f) \coloneqq (\lim_{n \to p} f(n)) 1$$
(where $\lim_{n \to p} f(n)$ denotes the limiting value of $f$ along the ultrafilter $p$) is a $\CStarAlgUnit$-endomorphism on $\CbFunc(\N)$ that does not arise from applying $\Mult$ to any $\CStarAlgMult$-morphism on $\CoFunc(\N)$, basically because the restriction of $\Psi$ to $\CoFunc(\N)$ vanishes and is therefore not non-degenerate.
\end{remark}

For future reference we record some variants of Urysohn's lemma in these categories.

\begin{proposition}[Urysohn properties]\label{urysohn}
Let $X$ be an $\LCH$-space and $K\subseteq X_\Set$ be compact. In (i),(ii),(iii) further assume that $K\subseteq U$ for some open $U\subseteq X_\Set$. 
\begin{itemize}
\item[(i)] There exists an open $V\subseteq X_\Set$ with compact closure $\overline{V}$ such that $K\subseteq V\subseteq \overline{V}\subseteq U$.  
\item[(ii)] (Urysohn's lemma)  There exists $f\in C_c(X)$ with $0\leq f\leq 1$ such that $f(x)=1$ for all $x\in K$ and $f(x)=0$ for all $x\in U^c$.  
\item[(iii)] There exists a compact $G_\delta$-set $\tilde{K}$ such that $K\subseteq \tilde{K}\subseteq U$.  
\item[(iv)]  $K$ is $G_\delta$ if and only if there exists $f\in \CoFunc(X)$ with $0\leq f\leq 1$ such that $K=f^*\{1\}$. 
\end{itemize}
\end{proposition}

\begin{proof}
Claims (i) and (ii) are standard facts (e.g., see \cite[Theorems 2.7 and 2.10]{rudinrealcomplex}).  
As for (iii), use (i) and (ii) to find a continuous function $f:X\to [0,1]$ such that $f(x)=0$ for all $x\in K$ and $f(x)=1$ for all $x\in V^c$ where $K\subseteq V\subseteq \overline{V}\subseteq U$ and $\overline{V}$ is compact.  Then the claim follows with
$$ \tilde K \coloneqq f^*\left[0,\frac{1}{2}\right] = \bigcap_{n=2}^\infty f^*\left[0,\frac{1}{2}+\frac{1}{n}\right).$$
Finally, we show (iv). If $K=\bigcap_{n=1}^\infty U_n$ is a $G_\delta$ set for some decreasing open $U_n \subseteq X_\Set$, then by (ii) there exist $f_n\in C_c(X)$, $0\leq f_n\leq 1$ such that $f_n(x)=1$ for all $x\in K$ and $f_n(x)=0$ for all $x\in U_n^c$ for all $n\geq 1$. Then $\sum_{n=1}^\infty 2^{-n} f_n$ converges in the Banach space $\CoFunc(X)$ to an element $f \in \CoFunc(X)$
with $f^*\{1\}=K$ and $0 \leq f \leq 1$. 

Conversely, if there exists $f \in \CoFunc(X)$ with $0 \leq f \leq 1$ and $f^*\{1\}=K$, then 
$$K = \bigcap_{n=1}^\infty f^*\left(1-\frac{1}{n},1+\frac{1}{n}\right)$$
is a $G_\delta$ set.
\end{proof}

As one application of Urysohn's lemma, we can classify the monomorphisms and epimorphisms in $\CH, \LCH, \LCHpr, \CHpt$.

\begin{proposition}[Morphisms of locally compact categories]\label{top-prop}
Let $\Cat$ be one of the categories $\CH, \LCH, \LCHpr, \CHpt$.
\begin{itemize}
\item[(i)] A $\Cat$-morphism is a $\Cat$-monomorphism if and only if it is injective. 
\item[(ii)] A $\Cat$-morphism is a $\Cat$-epimorphism if and only if it has dense image.  If $\Cat=\CH, \LCHpr, \CHpt$, it is also true that a $\Cat$-morphism is a $\Cat$-epimorphism if and only if it is surjective.
\item[(iii)] If $\Cat=\CH, \LCHpr, \CHpt$, then every $\Cat$-bimorphism is a $\Cat$-isomorphism. 
\end{itemize}
\end{proposition}

Note that the canonical embedding of $\N$ into $\beta \N$ is an $\LCH$-morphism which is injective and has dense image, but is not surjective, which shows that the claim (iii) and the second claim in (ii) cannot be extended to $\Cat = \LCH$.

\begin{proof} 
Since $\Cat$ is a category of sets,  by Lemma \ref{epimorph} and Example \ref{set-example}, every injective (resp. surjective) $\Cat$-morphism is a  $\Cat$-monomorphism (resp. epimorphism). 
By the identification of elements of a $\Cat$-space $X$ with the $\Cat$-morphisms from a point (or two points, in the case $\Cat = \CHpt$) to $X$, every $\Cat$-monomorphism is also injective. This yields (i). 

Now we show (ii).  By  continuity and the Hausdorff property, every $\Cat$-morphism with dense image is a $\Cat$-epimorphism.  Next, we show that any $\Cat$-epimorphism has a dense image (cf. \cite[Proposition 10.18]{walker2012stone}). Suppose for contradiction that $T \colon X\to Y$ is a $\Cat$-epimorphism with non-dense image $\overline{T(X)}\neq Y$.  By Proposition \ref{urysohn}(ii), one can find a non-trivial continuous function $f \colon Y \to [0,1]$ which vanishes on $\overline{T(X)}$.  The graphing functions $S_0, S_1 \colon Y \to Y \times [0,1]$ defined by $S_0(Y) \coloneqq (y,0)$ and $S_1(Y) \coloneqq (y,f(y))$ are then distinct $\Cat$-morphisms such that $S_0 \circ T = S_1 \circ T$, contradicting the hypothesis that $T$ is an $\Cat$-epimorphism.  This gives the first part of (ii).  To conclude, we need to show that for $\Cat = \CH, \LCHpr, \CHpt$, that every $\Cat$-morphism with dense image is surjective.  For $\Cat = \CH, \CHpt$ this follows since the image is compact.    
For $\Cat=\LCHpr$, let $T \colon X\to Y$ be an $\LCHpr$-morphism with dense image, and let $y\in Y$. Let $K$ be a compact neighborhood of $y$ in $Y$, then the set $T( T^{-1}(K) )$ is compact (by the proper continuous nature of $T$) and contains $y$ in its closure (as $T$ has dense image), hence $y \in T(T^{-1}(K)) \subseteq T(X)$.  Thus $T$ is surjective as required.

Finally, the assertions in (iii) follow from (i), (ii) and the well-known fact that any proper bijective continuous function is a homeomorphism. 
\end{proof}

As is well-known, the Stone--{\v C}ech and Alexandroff compactifications serve as universal ``maximal'' and ``minimal'' compactifications of an $\LCH$-space (or $\LCHpr$-space) $X$.  We can formalize these statements in category-theoretic language as follows.  Define a \emph{compactification} of an $\LCH$-space $X$ as a pair $(\tilde{X},\iota_X)$ such that $\tilde{X}\in \CH$ and $\iota_X: X\to \tilde{X}_\LCH$ is an $\LCH$-monomorphism with dense image. By Proposition \ref{top-prop}, $\iota_X$ is an $\LCH$-bimorphism, hence we could equivalently define a compactification as a pair $(\tilde{X},\iota_X)$ such that $\tilde X\in \CH$ and $\iota_X:X\to \tilde{X}_\LCH$ is an $\LCH$-bimorphism.  The class $\mathbf{Compact}(X)$ of all such compactifications of $X$ forms a partially ordered set (and so can be viewed as a small category), with ordering $(\tilde{X}, \iota_X) \leq (\tilde{X}', \iota_{X'})$ whenever there is a $\CH$-morphism $\tilde \pi \colon \tilde X \to \tilde{X}'$ such that $\iota_{X'} \circ \tilde \pi_{\LCH} = \iota_X$. 
The Stone-\v{C}ech functor $\beta$ gives one compactification $(\beta X, \iota_{X,\beta})$ in $\mathbf{Compact}(X)$, where $\iota_{X,\beta} \colon X \to (\beta X)_\LCH$ is the canonical inclusion; the Alexandroff functor gives another compactification $(\Alex(X), \iota_{X,\Alex})$ in $\mathbf{Compact}(X)$, where by abuse of notation $\Alex(X)$ is $\Alex$ applied to $X$ viewed as an $\LCHpr$-space, and $\iota_{X,\Alex} \colon X \to \Alex(X)_\LCH$ is the canonical inclusion.  It is then routine to verify that these two compactifications are the least and greatest elements in $\mathbf{Compact}(X)$; see Figure \ref{fig:stone-cech}.  Using Gelfand duality, one can also identify compactifications of an $\LCH$-space $X$ (up to natural isomorphisms) with unital subalgebras of $\CbFunc(X)$ that contain $\CoFunc(X) \oplus \C$, somewhat in the spirit of the fundamental theorem of Galois theory; we leave the details to the interested reader. 

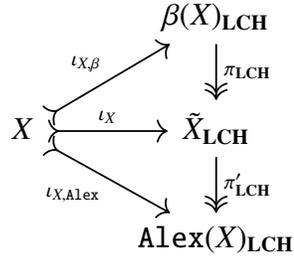
\begin{figure}
    \centering
    \begin{tikzcd}
    & \beta(X)_\LCH \arrow[d,two heads,"\pi_\LCH"] \\
    X \arrow[ur, tail,"\iota_{X,\beta}"] \arrow[r, tail,"\iota_X"] \arrow[dr, tail,"\iota_{X,\Alex}"'] & \tilde X_\LCH \arrow[d,two heads,"\pi'_\LCH"] \\ 
    & \Alex(X)_\LCH
\end{tikzcd}
    \caption{The universal properties of the Stone-\v{C}ech and Alexandroff compactifications. For any compactification $(\tilde X, \iota_X)$ of an $\LCH$-space $X$, there are unique $\CH$-morphisms $\pi, \pi'$ that make the above diagram commute (in $\LCH$).}
    \label{fig:stone-cech}
\end{figure}

\section{Baire algebras}\label{baire-sec}

In this section we describe all the categories and functors depicted in Figure \ref{fig:baire}.

\begin{definition}[Topological and measurable categories and functors]\label{bairefunc-def}\ 
\begin{itemize}
    \item[(i)]  A \emph{$\CMet$-space} is a compact metrizable space $X = (X_\Set, \mathcal{F}_X)$. A \emph{$\CMet$-morphism} is a continuous function between $\CMet$-spaces.  
    \item[(ii)]  A \emph{$\Polish$-space} is a Polish space $X = (X_\Set, {\mathcal F}_X)$ (i.e., a separable topological space that is completely metrizable). A \emph{$\Polish$-morphism} is a continuous function between $\Polish$-spaces. 
    \item[(iii)]  A \emph{$\ConcMes$-space} is a concrete measurable space $X = (X_\Set, \Sigma_X)$, i.e., a set $X_\Set$ endowed with a $\sigma$-complete Boolean algebra $\Sigma_X$ of subsets of $X_\Set$.  A \emph{$\ConcMes$-morphism} is a measurable map between $\ConcMes$-spaces. In (i)-(iii), composition is given by the usual $\Set$-composition law.
    \item[(iv)]  Forgetful functors from $\CMet$ to $\CH, \Polish$ are defined in the obvious fashion.
    \item[(v)]  If $X$ is a $\Polish$-space (resp. $\LCH$-space, $\LCHpr$-space), one defines $\BorelFunc(X) \coloneqq (X_\Set, \Borel(X))$ (resp. $\BaireBFunc(X) \coloneqq (X_\Set, \Baire(X))$, $\BaireCFunc(X) \coloneqq (X_\Set, \BaireC(X))$).  If $T$ is a $\Polish$-morphism (resp. $\LCH$-morphism, $\LCHpr$-morphism), we define $\BorelFunc(T) \coloneqq T$ (resp. $\BaireBFunc(T) \coloneqq T$, $\BaireCFunc(T) \coloneqq T$). (Here we abuse notation by identifying $\Cat$-morphisms $T$ with their underlying $\Set$-morphism $T_\Set$ for the concrete categories $\Cat = \Polish, \LCH, \LCHpr,$ $\ConcMes$.)
    \item[(vi)]  We define $\BaireFunc \coloneqq \BaireCFunc \circ \Forget_{\CH \to \LCHpr} = \BaireBFunc \circ \Forget_{\CH \to \LCH}$.
\end{itemize}
\end{definition}

It is a routine matter to verify that the categories and functors in Figure \ref{fig:baire} are indeed categories and functors with the indicated faithfulness properties, and that the diagram commutes.  Note that it is essential that $\LCHpr$-morphisms $T \colon X \to Y$ be proper in order for $\BaireCFunc$ to be a functor, since otherwise the Koopman operator $f \mapsto f \circ T$ would not map $\CcFunc(Y)$ to $\CcFunc(X)$.

We recall some standard product constructions in these categories (the reader is referred to Appendix \ref{sec-product} for our category-theoretic notation of categorical and non-categorical products and how these can be functorially related to each other):

\begin{proposition}[Products in topological and measurable categories]\label{prod-top}
\text{}
\begin{itemize}
  \item[(i)] The categories $\Set$, $\CH$, $\CHpt$, $\ConcMes$ admit categorical products $\prod^{\Set}$, $\prod^{\CH}$, $\prod^{\CHpt}$, $\prod^\ConcMes$, defined for arbitrarily many factors. 
    \item[(ii)]  The categories $\CMet$, $\Polish$ admit categorical products $\prod^\CMet$, $\prod^\Polish$, defined for at most countably many factors. 
    \item[(iii)] The category $\LCH$ admits categorical products $\prod^{\LCH}$, defined for finitely many factors. 
    \item[(iv)] The category $\LCHpr$ does not admit categorical products, not even for two factors. 
    \item[(v)]  With the exception of the categorical $\ConcMes$-product, all the categorical products listed in (i), (ii), (iii) are related in the sense of Definitions \ref{def-monoidal} and \ref{def-infinit-prod} (see Example \ref{exp-cartesian} and Remark \ref{rem-cartesian}) to each other with respect to the casting functors in Figure \ref{fig:baire}.
    \item[(vi)]  (Weil's theorem)  Every $\CH$-space $K$ is $\CH$-isomorphic to a compact subspace of a product $\prod_{\alpha \in A}^\CH S_\alpha$ of $\CMet$-spaces. In fact one can take each $S_\alpha$ to be a compact subset of $\R$. 
    \item[(vii)]  If $(S_\alpha)_{\alpha \in A}$ is a family of $\CMet$-spaces and $K$ is a closed $\CH$-subspace of $\prod_{\alpha \in A}^\CH S_\alpha$, then $K_\ConcMes$ is the restriction of $(\prod_{\alpha \in A}^\CH S_\alpha)_\ConcMes$ to $K_\Set$ (that is, the measurable sets in $K_\ConcMes$ are precisely the sets of the form $E \cap K_\Set$, where $E$ is measurable in $(\prod_{\alpha \in A}^\CH S_\alpha)_\ConcMes$), even if $K_\Set$ itself fails to be measurable in $(\prod_{\alpha \in A}^\CH S_\alpha)_\ConcMes$.
    \item[(viii)]  All the categorical products listed in (i), (ii) agree with each other in the sense of Definitions \ref{def-monoidal} and \ref{def-infinit-prod} with respect to the casting functors in Figure \ref{fig:baire}.
   \end{itemize}
\end{proposition}

\begin{proof}
The assertions in (i) are standard; for instance, the existence of the categorical product in $\CH$, for instance, follows readily from Tychonoff's theorem.  For $\CMet$ there is the issue of how to assign the metric on a countable product $\prod^\CMet_{n \in \N} X_n$ of the factor metric spaces $X_n = ((X_n)_\Set, d_n)$ in a manner that is compatible with the product topology, but this can be achieved in any number of ways, e.g., by using the metric
$$ d\left( (x_n)_{n \in \N}, (y_n)_{n \in \N} \right) \coloneqq \sum_{n \in \N} 2^{-n} \frac{d_n(x_n,y_n)}{1+d_n(x_n,y_n)}$$
(this construction can also be used to verify that the product of countably many Polish spaces is Polish).  

The assertions in (ii) and (iii) are also well-known (the category $\LCH$ does not admit general infinite products, since if it did, then so would the category of locally compact Hausdorff abelian groups. However, there is no product of countably many copies of the real numbers in the category of locally compact Hausdorff abelian groups, for if there were, then by the universal property of the categorical product, it would become a real Hausdorff topological vector space, contradicting the well known fact that the only locally compact Hausdorff topological vector spaces are finite-dimensional.).  

To verify (iv), suppose there would exist a categorical product of $\R$ with itself, say $X$. Then for any point $y\in \R$, there must exist unique proper maps $\{y\}\times \R\to X$ and $\R\times \{y\}\to X$ by the universal property of the categorical product $X$. 
This implies that $\R\times \R$ embeds properly into $X$, but then the projection maps from $X$ to $\R$ cannot be proper, giving the claim. 

Claim (v) follows from a routine expansion of the definitions, with matters boiling down to establishing easy identities such as $(\prod_{n \in \N}^\CMet X_n)_\CH = \prod^\CH_{n \in \N} (X_n)_\CH$ for a countable sequence $X_n$ of $\CMet$ spaces.

Claim (vi) was established in \cite{weil1937espaces}; for a canonical construction, take $A \coloneqq \CFunc(K)$, set $S_f \coloneqq f(K)$ for $f \in \CFunc(K)$, and identify each point $k \in K$ with the tuple $(f(k))_{f \in \CFunc(K)}$; the required properties are then easily verified using Urysohn's lemma \ref{urysohn}.

Claim (vii) was established in \cite[Lemma 2.1]{jt19}. For Claim (viii), the only non-routine step is in establishing that $\BaireFunc(\prod_{\alpha \in A} X_\alpha) = \prod_{\alpha \in A} \BaireFunc(X_\alpha)$ for any (possibly uncountable) family $(X_\alpha)_{\alpha \in A}$ of $\CH$-spaces, and that $\BorelFunc(\prod_{n \in \N} X_n) = \prod_{n \in \N} \BorelFunc(X_n)$ for any countable family of $\Polish$-spaces.  The first claim follows from (vi), (vii) (and is also proven in \cite[Proposition 2.3]{varadarajan-riesz}), and the second follows from constructing a countable subbase for $\prod_{n \in \N} X_n$ arising from open balls in the individual $X_n$. 
\end{proof}

The compatibility of the $\CH$ and $\ConcMes$ products via the Baire functor $\BaireFunc$ (which, as mentioned above, was first proven in \cite[Proposition 2.3]{varadarajan-riesz}) is one of the major reasons why it is preferable to use the Baire $\sigma$-algebra instead of the Borel $\sigma$-algebra for $\CH$-spaces.  On the other hand, the $\LCH$ product is not compatible with the $\ConcMes$ product even when multiplying just two spaces together.  For instance, if $X$ is a discrete $\LCH$-space with cardinality greater than the continuum, then $X \times_\LCH X$ is also discrete, hence every subset is $\CbFunc$-Baire-measurable (every indicator function is bounded continuous).  In particular, the diagonal $\{ (x,x): x \in X \}$ is $\CbFunc$-Baire-measurable.  On the other hand, it is easy to see that every set $E$ measurable in the product space  $\BaireB(X) \times_{\ConcMes} \BaireB(X)$ has the property that the slices $E_x \coloneqq \{ y \in X: (x,y) \in X \}$ lie in a countably generated $\sigma$-algebra.  By cardinality considerations, the diagonal does not have this property, hence $\BaireB(X \times_\LCH X) \neq \BaireB(X) \times_\ConcMes \BaireB(X)$, demonstrating the incompatibility of the $\LCH$ and $\ConcMes$ products.

We have the following useful descriptions of the $\CbFunc$-Baire and $\CcFunc$-Baire $\sigma$-algebras: 

\begin{proposition}[Characterization of Baire algebras]\label{baire-prop}
\text{}
\begin{itemize}
\item[(i)] Let $X$ be an $\LCH$-space. Then $B\in \BaireB(X)$ if and only if there exist a sequence of real $f_n\in \CbFunc(X), n\in \N$ and $A\in \Borel(\prod^\Polish_{n\in \N} \R)$ such that $B=(\prod_{n\in \N}^\Polish f_n)^*(A)$.  
\item[(ii)] Let $X$ be an $\LCHpr$-space. Then
$B\in \BaireC(X)$ if and only if there exist a sequence of continuous functions $f_n:X\to [0,1], n\in \N$ and $A\in \Borel(\prod^\CMet_{n\in \N} [0,1])$ such that $B=(\prod_{n\in \N}^\CMet f_n)^*(A)$.  Equivalently, $\BaireC(X)$ is generated by all compact $G_\delta$ subsets of $X$.   
\end{itemize}
\end{proposition}

\begin{proof} We begin with (i). The set of all $B$ of the form $B = (\prod_{n \in \N}^\Polish f_n)^*(A)$ for some real $f_n \in \CbFunc(X)$ and $A \in \Borel(\prod^\Polish_{n \in \N} \R)$ is a $\sigma$-algebra that contains the preimages $f^*(E)$ of any Borel subset $E$ of $\C$ by elements $f$ of $\CbFunc(X)$, and thus contains $\BaireB(X)$.  To obtain the converse inclusion, it suffices to show that for any fixed real $f_n\in \CbFunc(X), n\in\N$, the collection $\{A\in \Borel(\prod^\Polish_{n\in \N} \R): (\prod_{n\in \N} f_n) ^*(A)\in \BaireB(X)\}$ is all of $\Borel(\prod^\Polish_{n\in \N} \R)$.  Since this collection is a $\sigma$-algebra, it suffices to show that it contains all closed subsets $F$ of $\prod^\Polish_{n\in \N} \R$.  Let $F$ be such a closed set, and let $g \colon X \to [0,1]$ be the function $g(x) \coloneqq \min( \mathrm{dist}((f_n(x))_{n \in \N},F), 1)$ for $x \in X$ (using a suitable product metric on $\prod^\Polish_{n\in \N} \R$).  Then $g \in \CbFunc(X)$ and $F = g^*\{0\}$, giving the claim.  This proves (i).

Now we establish\footnote{We are indebted to Minghao Pan for pointing out a mistake in the proof of Proposition \ref{baire-prop}(ii) in a preliminary manuscript.}  (ii). The first claim can be shown similarly to (i). By Proposition \ref{urysohn}(iv), every $G_\delta$ set $K$ is of the form $K=f^*\{1\}$ for some $f\in C_0(X)$ which shows that the $\sigma$-algebra generated by the compact $G_\delta$ sets is included in $\BaireC(X)$. Conversely, for any real $f \in \CcFunc(X)$, the level sets $\{f\geq r\}, r\in \R$  belong to the $\sigma$-algebra generated by compact $G_\delta$-sets, and hence the entirety of $\BaireC(X)$ does also (decomposing complex-valued functions in $\CcFunc(X)$ into real and imaginary parts). This gives (ii).
\end{proof}

As is well-known, an illustrative example of the subtleties of uncountable topological spaces is provided by the first uncountable ordinal $\omega_1$.  

\begin{proposition}\label{omega1} 
Let the intervals $[0,\omega_1)$ and $[0,\omega_1]$ be endowed with  the order topology.
\begin{itemize}
\item[(i)] The space $[0,\omega_1)$ is an $\LCH$-space (or $\LCHpr$-space) which is countably compact, sequentially compact, and first countable, but not compact, $\sigma$-compact, paracompact or second countable.  The space $[0,\omega_1]$ is a $\CH$-space which is not first countable.   
\item[(ii)] Both $[0,\omega_1)$ and $[0,\omega_1]$ are  zero-dimensional.   
\item[(iii)] A subset of $[0,\omega_1)$ (resp. $[0,\omega_1]$) is compact if and only if it is complete\footnote{An ordered set is said to be \emph{complete} if every non-empty subset has an infimum and a supremum.}.  Moreover, every compact subset of $[0,\omega_1)$ is $G_\delta$.   
\item[(iv)] Both $[0,\omega_1)$ and $[0,\omega_1]$ are completely normal, but neither is perfectly normal.  
\item[(v)] Every complex continuous function on $[0,\omega_1)$ (resp. $[0,\omega_1]$) is eventually constant.  
Therefore, we have $\CcFunc([0,\omega_1))=\CoFunc([0,\omega_1))$ and $\CFunc([0,\omega_1])=\CbFunc([0,\omega_1))=\CoFunc([0,\omega_1)) \oplus \C$.   
\item[(vi)] One has $\BaireC([0,\omega_1)) = \BaireB([0,\omega_1)) \subsetneq \Borel([0,\omega_1))$ and $\BaireC([0,\omega_1]) = \BaireB([0,\omega_1]) \subsetneq \Borel([0,\omega_1])$. 
\item[(vii)] Both the Stone-\v{C}ech compactification $\beta [0,\omega_1)$ and the Alexandroff compactification $\Alex([0,\omega_1))$ are identifiable\footnote{This should be contrasted with the fact that the Stone-\v{C}ech compactification of $\omega_0$ (the first infinite ordinal) is much larger than its  Alexandroff compactification.} with $[0,\omega_1]$. 
\end{itemize}
\end{proposition}

\begin{proof}
A proof of the properties in (i), (v) can be found in \cite[\S 42]{counterexamplestopology}, and (vii) follows easily from (v).  
As for (ii), notice that we will not be able to construct a strictly decreasing infinite sequence in $[0,\omega_1)$ or $[0,\omega_1]$, and therefore the collection of intervals $(\alpha,\beta]$, $\alpha<\beta$ together with $\{0\}$ forms a base of clopen subsets respectively.  
See \cite[\S 39.7]{counterexamplestopology} for the characterization of compactness in terms of completeness for subsets of $[0,\omega_1)$ and $[0,\omega_1]$.  
This characterization implies that every compact subset of $[0,\omega_1)$ can be viewed as a closed subset of a compact interval in $[0,\omega_1)$.  A compact interval in $[0,\omega_1)$ is second-countable with respect to the subspace topology, and hence is metrizable. It is well-known that closed subsets of metric spaces are $G_\delta$, giving (iii). 

See \cite[\S 39.6]{counterexamplestopology} for a proof that $[0,\omega_1)$ and $[0,\omega_1]$ are completely normal. 
To prove that neither are perfectly normal, it is enough to find a closed set that is not $G_\delta$ respectively. 
It is easy to see that $\{\omega_1\}$ is not $G_\delta$ in $[0,\omega_1]$.  
We show that the set $A=\{\alpha\in [0,\omega_1): \alpha \text{ limit ordinal} \}$   is not $G_\delta$ in $[0,\omega_1)$.  
Let $O$  be an open set including $A$. 
Then for each $\gamma\in A$ there exists $\alpha_\gamma\in \omega_1$ such that $(\alpha_\gamma,\gamma]=[\alpha_\gamma+1,\gamma]\subseteq O$. 
Define $f:A\to [0,\omega_1)$ to be $f(\gamma):=\alpha_\gamma+1$.  By Fodor's Pressing Down Lemma (e.g., see \cite[Lemma  III.6.14]{kunen}), there exist $\alpha\in \omega_1$ and a set $B\subseteq A$ which has nonempty intersection with any unbounded closed subset of $\omega_1$ such that $f(\beta)= \alpha$ for all $\beta\in B$. By construction, $[\alpha,\omega_1)\subseteq O$.
Now let $(O_n)$ be a sequence in $[0,\omega_1)$ such that $A\subseteq O_n$ for all $n$. 
For each $n$ choose $\alpha_n$ such that $[\alpha_n,\omega_1)\subseteq O_n$.  
Then $\alpha_*=\sup\{\alpha_n\}$ is a countable ordinal and $[\alpha_*,\omega_1)\subseteq \bigcap_n O_n$. 
As such a ray $[\alpha_*,\omega_1)$ must include a successor ordinal, $A\neq \bigcap_n O_n$.  

Now we establish (vi).  From (v) we have $\CcFunc(([0,\omega_1)) = \CbFunc(([0,\omega_1))$, hence $\BaireC([0,\omega_1)) = \BaireB([0,\omega_1))$; similarly, from the compactness of $[0,\omega_1]$ one has $\BaireC([0,\omega_1]) = \BaireB([0,\omega_1]) = \Baire([0,\omega_1])$.  By \eqref{baire-include} it remains to show that $\BaireB([0,\omega_1)) \neq \Borel([0,\omega_1))$ and
$\Baire([0,\omega_1]) \neq \Borel([0,\omega_1])$. To establish the first claim, it suffices to show that the set $A$ of all limit ordinals smaller in $\omega_1$  is not Baire-measurable (as a closed set it is clearly Borel-measurable).  If for contradiction $A$ were an element of $\BaireB([0,\omega_1))$, by Proposition \ref{baire-prop} there would exist $f_n\in \CbFunc([0,\omega_1))$, $B\in \Borel(\prod_{n\in \N}^\Polish \R)$ such that $A=(\prod_n f_n)^*(B)$.  By Proposition \ref{omega1}(v), each $f_n$ is eventually constant with some constant value $c_n$ for all ordinals larger or equal than $\alpha_n$.  If $(c_1,c_2,\ldots)\in B$, then $(\prod_n f_n)^*(B)$ includes the interval $[\sup_n\{\alpha_n\},\omega_1)$, and thus cannot be $A$, giving a contradiction.  Hence $(c_1,c_2,\ldots)\not\in B$, in which case there exists some $c_i$ which is not an element of the $i$th projection of $B$, so if $\beta\in (\prod_n f_n)^*(B)$ then $\beta<\alpha_i$, and thus also in this case $A$ cannot be $(\prod_n f_n)^*(B)$, again giving the contradiction.  

It remains to show $\Baire([0,\omega_1]) \neq \Borel([0,\omega_1])$.  But this follows after observing from (v) that every Baire-measurable subset of $[0,\omega_1]$ is either bounded, or has a bounded complement, so in particular the Borel-measurable set $[0,\omega_1)$ (or the complement $\{\omega_1\}$) is not Baire-measurable.  Alternatively, we can also derive this from \cite[Proposition 1.4]{comfort} which establishes the following equivalence: 
If $X$ is an $\LCH$-space (resp. $\LCHpr$-space) and $A\in \BaireC(X)$ is closed, then $A$ is $\sigma$-compact if and only if $A$ is Baire measurable in $\beta(X)$.  
Now $[0,\omega_1)$ is clearly closed in $\BaireC([0,\omega_1))$ but by Proposition \ref{omega1}(i) is not $\sigma$-compact, so $[0,\omega_1)$ is not in $\Baire(\beta(X))=\Baire([0,\omega_1])$ by Proposition \ref{omega1}(vii).  
\end{proof}

\begin{remark}\label{baire-remark} We now discuss when the inclusions in \eqref{baire-include} are strict.  The inclusion $\BaireC(X) \subseteq \BaireB(X)$ is strict when $X$ is an uncountable discrete space.  Proposition \ref{baire-prop}(i), (v) offers an obvious (though not obviously useful) necessary and sufficient condition for $\BaireC(X)=\BaireB(X)$: For all $f\in \CbFunc(X)$ there exist $f_n\in \CcFunc(X), n\in \N$, $A\in \Borel(\prod^\CMet_{n\in \N} X_n)$, where the $X_n\subseteq \R$ are compact, such that $f^*(\{0\})=(\prod_{n\in \N} f_n) ^*(A)$. 
Another merely sufficient condition is that any $f\in \CbFunc(X)$ is the pointwise limit of a sequence of functions in $\CoFunc(X)$. 
This condition is equivalent to saying that $X$ is $\sigma$-compact or that the $C^*$-algebra $\CoFunc(X)$ has a countable approximate identity.   
However, Proposition \ref{omega1}(vi) shows that this condition is not necessary in order to have $\BaireC(X)=\BaireB(X)$.  

Proposition \ref{omega1}(vi) also gives examples in which 
$\BaireB(X)\neq \Borel(X)$.  A sufficient condition for $\BaireB(X)=\Borel(X)$ is that $X$ is perfectly normal\footnote{A topological space $X$ is said to be \emph{perfectly normal} if two disjoint closed sets $E,F$ can be perfectly separated by a continuous function, that is there is $f \colon X\to [0,1]$ such that $f^*(\{0\})=E$ and $f^*(\{1\})=F$. Equivalently, $X$ is perfectly normal if it is normal and every closed set is $G_\delta$.}. However, as the example of an uncountable discrete space showed, being perfectly normal is definitely not enough to also have $\BaireC(X)=\BaireB(X)$. 
\end{remark}

\begin{remark} From the Gelfand dualities in Figure \ref{fig:gelfand-dual}, it is not difficult to show that for an $\LCH$-space $X$ (which we also view as an $\LCHpr$-space) that $\BaireC(X)$ is the restriction of the Baire algebra $\Baire(\Alex(X))$ of the Alexandroff compactification $\Alex(X)$ to $X$, while $\BaireB(X)$ is similarly the restriction of the Baire algebra $\Baire(\beta X)$ of the Stone--{\v C}ech compactification $\beta X$ to $X$.  Thus we see that the two canonical compactifications and two canonical Baire algebras of locally compact Hausdorff spaces are naturally divided up between the two categories $\LCH$, $\LCHpr$.
\end{remark}

\section{Regular measures and \texorpdfstring{$\tau$}{tau}-additivity}

In the theory of both Baire and Borel probability measures it is common to impose additional axioms such as inner or outer regularity, $\tau$-additivity, or the Radon measure property; see, e.g., \cite{knowles}.  We recall the relevant notions.

\begin{definition}[Regularity properties]  Let $X = (X_\ConcMes, {\mathcal F}_X)$ be a $\ConcMes$ space $X_\ConcMes = (X_\Set, \Sigma_X)$ equipped with a topology ${\mathcal F}_X$ on $X_\Set$.  Let $\mu_X$ be a finite measure on $X$.
\begin{itemize}
\item[(i)]  We say that $\mu$ is \emph{$\tau$-additive} in $X$ if
$$ \sup_{\alpha \in A} \mu(O_\alpha) = \mu(O)$$
whenever $(O_\alpha)_{\alpha \in A}$ is a net of open measurable sets $O_\alpha \in \Sigma_X$ which is non-decreasing (thus $O_{\alpha} \subseteq O_\beta$ whenever $\alpha \leq \beta$, and $O \coloneqq \bigcup_{\alpha \in A} O_\alpha$ is also open measurable).
\item[(ii)]  If $\langle \mathrm{adjective} \rangle$ is an adjective pertaining to subsets of $X$ which applies in particular to the empty set, such as ``closed'', ``closed $G_\delta$'', ``compact'', or ``compact $G_\delta$'', we say that $\mu$ is \emph{$\langle\mathrm{adjective}\rangle$ inner regular} in $X$ if
$$ \mu(E) = \sup \{ \mu(F) \colon F \in \Sigma_X, F \subseteq E, F \ \langle\mathrm{adjective}\rangle \}$$
for all $E \in \Sigma_X$. Similarly, if $\langle \mathrm{adjective} \rangle$ be an adjective pertaining to subsets of $X$ which applies in particular to the whole set $X_\Set$, such as ``open'', or ``open $F_\sigma$'', we say that $\mu$ is
\emph{$\langle\mathrm{adjective}\rangle$ outer regular} in $X$ if
$$ \mu(E) = \inf \{ \mu(O) \colon O \in \Sigma_X, E \subseteq O, O \ \langle\mathrm{adjective}\rangle \}$$
for all $E \in \Sigma_X$.  
\item[(iii)]  We say that $\mu$ is \emph{Radon} in $X$ if it is compact $G_\delta$ inner regular.
\end{itemize}
\end{definition}

Using $\mu(E^c) = \mu(X)-\mu(E)$ we easily verify the logical implications
\begin{gather*}
\hbox{Radon} \iff \hbox{compact $G_\delta$ inner regular}\\
\Downarrow \\
\hbox{closed $G_\delta$ inner regular} \iff \hbox{open $F_\sigma$ outer regular} \\
\Downarrow\\
\hbox{closed inner regular} \iff \hbox{open outer regular}
\end{gather*}
in Hausdorff spaces (in which compact sets are closed).  In metrizable spaces we can reverse the second downward arrow (because closed sets are automatically $G_\delta$), and in $\CH$-spaces we can reverse the first downward arrow (because closed sets are automatically compact).  For Borel measures, the notions of compact inner regularity (also known as \emph{tightness}) and open outer regularity are the most frequently employed, but for Baire probability measures the notion of closed $G_\delta$ inner regularity (or equivalently open $F_\sigma$ outer regularity) is of more use.  In particular we will not make much use of the concept of compact inner regularity in this paper.  As we shall shortly see, the property of $\tau$-additivity is automatic in $\CH$-spaces, but can be non-trivial in non-compact spaces.

It is a well-known theorem of Ulam (see, e.g. \cite{halmos-measure-theory}, \cite[Theorem 7.1.4]{dudley01} or \cite[Proposition 4.2]{sunder}) that Borel probability measures on $\CMet$-spaces are automatically Radon. We review several further results (also reasonably well-known) of this type:

\begin{proposition}[Automatic regularity of Borel and Baire measures]\label{automatic}  Let $\Cat$ be one of $\CMet$, $\Polish$, $\CH$, $\CHpt$, $\LCH$, $\LCHpr$, let $X$ be a $\Cat$-space, and let $\mu$ be a probability measure on $X_\ConcMes$ (here we use the casting functors from Figure \ref{fig:baire}).
\begin{itemize}
\item[(i)] $\mu$ is closed $G_\delta$ inner regular and open outer $F_\sigma$ regular. (In particular $\mu$ is closed inner regular and open outer regular.)
\item[(ii)]  If $\Cat = \CMet$, $\CH$, $\CHpt$, $\LCH$, $\LCHpr$, then $\mu$ is Radon in $X$ if and only if it is $\tau$-additive.
\item[(iii)] If $\Cat = \CMet$, $\CH$, $\CHpt$, then $\mu$ is both Radon and $\tau$-additive in $X$.
\end{itemize}
\end{proposition}

\begin{proof}  We begin with (i).  By applying forgetful functors it suffices to check the cases $\Cat = \Polish$, $\LCH$, $\LCHpr$.  For $\Cat=\Polish$ this follows from Ulam's tightness theorem (see, e.g., \cite[Theorem 7.1.4]{dudley01}), noting that in $\Polish$-spaces closed sets are automatically $G_\delta$ due to metrizability.  Now we establish the claim for $\Cat=\LCH$.  It suffices to establish closed $F_\sigma$ inner regularity.  Let $E \in \Sigma_X = \BaireB(X)$, then by Proposition \ref{baire-prop}(i) we have $E = T^* A$ for some $\LCH$-morphism $T \colon X \to Y_\LCH$ and some $\Polish$-space $Y$ (indeed one can take $Y = \prod^\Polish_{n \in \N} \R$).  Applying the $\Cat=\Polish$ case of (i) to the pushforward measure $T_* \mu$, we see that for any $\eps>0$ there is a closed $G_\delta$ subset $F$ of $A$ such that $T_* \mu(A \backslash F) \leq \eps$, hence $\mu(E \backslash T^* F) \leq \eps$. Since $T^* F$ is a closed $G_\delta$ subset of $E$, this establishes closed $G_\delta$ inner regularity when $\Cat = \LCH$.  The case $\Cat=\LCHpr$ is obtained similarly using Proposition \ref{baire-prop}(ii).  We also remark that the $\Cat=\CH$ case was established in \cite[II, 7.1.8]{bogachev2006measure}.

Now we establish (ii).  By applying forgetful functors it suffices to establish the claim for $\Cat = \LCH$, $\LCHpr$.  We begin with the $\Cat = \LCH$ case.  Suppose first that $\mu$ is Radon in $X$, and $O = \bigcup_{\alpha \in A} O_A$ for some non-decreasing net $(O_\alpha)_{\alpha \in A}$ of open Baire sets whose union $O$ is also open Baire.  By the Radon hypothesis, for any $\eps$ there is a compact $G_\delta$ subset $K$ of $O$ such that $\mu(O \backslash K) \leq \eps$.  By compactness, $K$ is covered by a finite number of the $O_\alpha$, hence (by the non-decreasing net hypothesis) one has $K \subseteq O_\beta \subseteq O$ for some $\beta \in A$, which establishes $\tau$-additivity.  Conversely, suppose $\mu$ is $\tau$-additive in $X$, and let $O$ be Baire open in $X$.  Consider the family ${\mathcal F}$ of open Baire subsets $U$ of $O$ with the property that $U \subseteq K \subseteq O$ for some compact $G_\delta$ $K$, ordered by set inclusion. This is a non-decreasing net of open Baire sets, and from Proposition \ref{urysohn} we see that every $x \in U$ is contained in at least one set $U$ from this family ${\mathcal F}$.  From $\tau$-additivity we conclude that
$$ \mu(O) = \sup \{ \mu(U): U \in {\mathcal F} \}$$
and hence
\begin{equation}\label{muo}
 \mu(O) = \sup \{ \mu(K): K \hbox{ compact } G_\delta \},
\end{equation}
which gives the Radon property for Baire open sets $O$.  Now if $E$ is a Baire set and $\eps>0$, we see from (i) that there is an open $F_\sigma$ set $O \supset E$ such that
$$ \mu(O \backslash E) \leq \eps$$
then by \eqref{muo} there is a compact $G_\delta$ set $K \subseteq O$ such that
$$ \mu(O \backslash K) \leq \eps.$$
Applying (i) again we also have an open $F_\sigma$ set $U \supset O \backslash E$ such that 
$$ \mu(U) \leq 2\eps.$$
The set $K \backslash U$ is then a compact $G_\delta$ subset of $E$ with
$$ \mu(E \backslash (K \backslash U)) \leq 3\eps.$$
Since $\eps>0$ is arbitrary, we conclude the Radon property for general Baire sets $E$.

The claim (iii) follows from (i) and (ii) after noting in these categories that closed sets are automatically compact. (A slightly weaker version of this claim, dropping the $G_\delta$ requirement, is also established in \cite[Theorem 7.1.5]{dudley01}.)
\end{proof}

\section{Riesz representation theorems}

We can now introduce the probability theory analogues $\CMetProb$, $\CHProb$, $\CHptProb$, $\PolishProb$, $\LCHProb$, $\LCHprProb$ of the topological categories $\CMet$, $\CH$, $\CHpt$, $\Polish$, $\LCH$, $\LCHpr$, as well as the analogue $\ConcProb$ of $\ConcMes$. We will do this by following a general categorical construction called action categories (see Definition \ref{def-action}). 

\begin{definition}[Topological-probabilistic categories]\label{top-prob-cat}  Let $\Cat = \CMet$, $\CH$, $\CHpt$, $\Polish$, $\LCH$, $\LCHpr$, and let $\Cat\mathbf{Prb}$ be the string formed by appending $\mathbf{Prb}$ to $\Cat$.
\begin{itemize}
\item[(i)] Let the functor $\mathtt{Prb}\colon \Cat\to \Set$ send an object $X$ in $\Cat$ to the set $\mathtt{Prb}(X)$ of Radon probability measures on $X_\ConcMes=(X_\Set,\Sigma_X)$ (using the casting functors from Figure \ref{fig:baire}) and a $\Cat$-morphism $f\colon X\to Y$ to the pushforward map $\mathtt{Prb}(f)\colon \mathtt{Prb}(X)\to \mathtt{Prb}(Y)$  defined by $\mathtt{Prb}(f)(\mu)=f_\ast\mu$, that is to say $f_\ast\mu(E) = \mu(f^*(E))$ for all $E \in \Sigma_X$. 
We define the category $\Cat\mathbf{Prb}$ to be the action category\footnote{This category can be identified with the category whose objects are Radon probability spaces on $\Cat$-spaces and whose morphisms are measure-preserving $\Cat$-morphisms.}   $\Cat\ltimes \mathtt{Prb}$; see Definition \ref{def-action}. 
\item[(ii)]  We can construct the category $\ConcProb$ of concrete probability spaces by a similar construction using the functor $\mathtt{Prb}\colon \Set$ associating to any $\ConcMes$-space the set of probability measures on it (see Example \ref{concprob-ex}). 
\end{itemize}
\end{definition}
 By Proposition \ref{automatic} we see that the requirement that $\mathtt{Prb}(X)$ are Radon probability measures can be dropped when $\Cat = \CMet, \CH$, and replaced with $\tau$-additivity when $\Cat = \LCH, \LCHpr$.  By definition, any Radon probability measure $\mu_X$ on $X_\Cat$ generates a \emph{$\Cat\mathbf{Prb}$-promotion} $(X_\Cat,\mu_X)$ of the $\Cat$-space $X_\Cat$ to a $\Cat\mathbf{Prb}$-space.  We note the subtle difference between an $\LCHProb$-space and an $\LCHprProb$-space: both spaces are locally compact Hausdorff spaces equipped with a Radon probability measure, but in the former case the measure is defined on the $\CbFunc$-Baire $\sigma$-algebra, but in the latter case the measure is defined on the smaller $\CoFunc$-Baire $\sigma$-algebra.  However, the distinction between the two types of Radon probability measure (as well as the Borel measure counterpart) can be erased in practice; see Corollary \ref{canon} below.  We also note that the category $\CHptProb$ of pointed $\CH$ spaces equipped with a probability measure is not the same as the (significantly less interesting) category $(\mathrm{pt} \downarrow \CHProb)$ of pointed $\CHProb$-spaces, as in the latter the distinguished point would be required to support the entire probability measure thanks to the definition of a $\CHProb$-morphism.

The functors in Figure \ref{fig:baire} have analogues in probabilistic categories which we depict in Figure \ref{fig:prob}.  All of these functors will be deemed to be casting functors, as are the forgetful functors from $\Cat\mathbf{Prb}$ to $\Cat$ for 
each category $\Cat$ appearing in Definition \ref{top-prob-cat}.

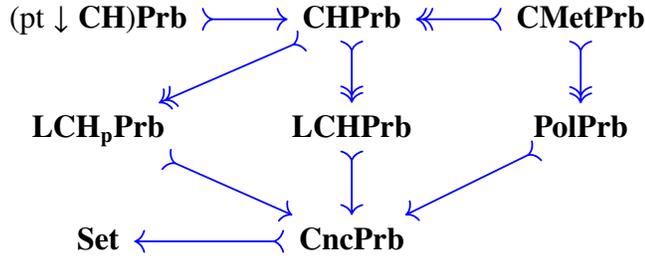
\begin{figure}
    \centering
    \begin{tikzcd}
    \CHptProb \arrow[r,blue,tail] & \CHProb \arrow[dl, blue,  tail,two heads] \arrow[d,blue, tail, two heads] & \CMetProb \arrow[l,blue,tail, two heads] \arrow[d,blue,tail,two heads] \\
    \LCHprProb \arrow[dr, blue, tail] & \LCHProb \arrow[d, blue, tail] & \PolishProb \arrow[dl, blue, tail] \\
    \Set & \ConcProb \arrow[l,blue,tail]
\end{tikzcd}
    \caption{Functors from topological-probabilistic categories to the concrete probabilistic category $\ConcProb$, which in turn has a forgetful functor to the category $\Set$ of sets. Every category here has a forgetful casting functor to its counterpart in Figure \ref{fig:baire}, and the union of these two diagrams together with these functors commutes.}
    \label{fig:prob}
\end{figure}

We now focus on the Riesz representation theory for $\LCHProb$-spaces and $\LCHprProb$-spaces.  We begin with the basic theory of linear functionals on $\CoFunc(X)$ and $\CbFunc(X)$ for $\LCH$-spaces $X$; these notions will end up being identified via Riesz dualities with Radon measures on $X$ and $\beta X$ respectively.

\begin{definition}[Functionals]\label{linfunc}  Let $X$ be an $\LCH$-space (which can also be identified with an $\LCHpr$-space).  A \emph{$\CoFunc$-functional} (resp. \emph{$\CbFunc$-functional}) on $X$ is a complex linear functional $\lambda \colon \CoFunc(X) \to \C$ (resp. $\lambda \colon \CbFunc(X) \to \C$).
\begin{itemize}
\item[(i)] We say that a $\CoFunc$-functional (resp. $\CbFunc$-functional) $\lambda$ is \emph{non-negative} if $\lambda(f) \geq 0$ whenever $f \geq 0$ is a real non-negative element of $\CoFunc(X)$ (resp. $\CbFunc(X)$).
\item[(ii)]  We say that a $\CoFunc$-functional (resp. $\CbFunc$-functional) $\lambda$ is \emph{$\tau$-smooth} if one has $\lim_\alpha \lambda(f_\alpha)=0$ whenever $(f_\alpha)_{\alpha \in A}$ is a net of real elements of $\CoFunc(X)$ (resp. $\CbFunc(X)$) which is non-increasing (thus $f_\alpha(x) \leq f_\beta(x)$ whenever $\alpha \geq \beta$ and $x \in X$) and converges pointwise to zero, thus $\lim_\alpha f_\alpha(x)=0$ for all $x \in X$.
\item[(iii)]  We say that $\CoFunc$-functional (resp.  $\CbFunc$-functional) $\lambda$ is a \emph{$\CoFunc$-state} (resp. \emph{$\CbFunc$-state}) if it is non-negative and has operator norm $1$.
\item[(iv)]  We say that a $\CoFunc$-functional (resp. $\CbFunc$-functional) $\lambda$ is \emph{represented} by a Radon probability measure $\mu_X$ on $X$ (or by the pair $(X,\mu_X)$) if one has $\lambda(f) = \int_X f\ d\mu_X$ for all $f$ in $\CoFunc(X)$ (resp. $\CbFunc(X)$).
\end{itemize}
If $X$ is a $\CH$-space, there is no distinction between $\CoFunc(X)$ and $\CbFunc(X)$, and so we drop the ``$\CoFunc$'' and ``$\CbFunc$'' prefixes in this case.
\end{definition}

Intuitively, a $\tau$-smooth functional is one which ``assigns no mass'' to $\beta X \backslash X$; we formalize this intuition later in Theorem \ref{rrt}(iii). In Examples \ref{ultra-ex}, \ref{dieu-ex}, \ref{unc-ex} below we give examples of $\CbFunc$-states that are not $\tau$-smooth.

\begin{proposition}[Properties of functionals]\label{func-prop}  Let $X$ be an $\LCH$-space (resp. an $\LCHpr$-space).
\begin{itemize}
\item[(i)]  If $\lambda$ is a non-negative $\CbFunc$-functional (resp. $\CoFunc$-functional) on $X$, then it is bounded; in particular it is a scalar multiple of a $\CbFunc$-state (resp. $\CoFunc$-state).
\item[(ii)]  Every $\CoFunc$-state $\lambda$ on $X$ is $\tau$-smooth.
\item[(iii)]  Any Radon probability measure $\mu_X$ on $X$ represents a unique $\tau$-smooth $\CbFunc$-state (resp. $\CoFunc$-state) $\lambda$.
\item[(iv)]  Every $\CoFunc$-state $\lambda$ on $X$ has a unique extension to a state on the Alexandroff compactification $\Alex(X)$.
\item[(v)] Every $\CoFunc$-state $\lambda$ on $X$ has a unique extension to a $\CbFunc$-state on $X$. Furthermore, this extension is $\tau$-smooth.
\end{itemize}
\end{proposition}

\begin{proof} We begin with (i).  It suffices to establish boundedness of $\lambda$ applied to non-negative real $f$ in $\CoFunc(X)$ or $\CbFunc(X)$.  When $\lambda$ is a $\CbFunc$-functional this is immediate from the bounds
$$ 0 \leq \lambda(f) \leq \lambda(1) \|f\|_{\CbFunc(X)} $$
arising from non-negativity.  Now suppose $\lambda$ is a $\CoFunc$-functional.  If $\lambda$ is unbounded for non-negative real $f$, then for each $n \in \N$ there exists non-negative $f_n \in \CoFunc(X)$ with $\|f_n\|_{\CoFunc(X)} \leq 2^{-n}$ such that $\lambda(f_n) \geq 1$.  But then by non-negativity $f \coloneqq \sum_{n=1}^\infty f_n$ is an element of $\CoFunc(X)$ such that 
$\lambda(f) \geq \lambda(\sum_{n=1}^N f_n) \geq N$ for any $N \in \N$, which is absurd.  Thus $\lambda$ is bounded.

To prove (ii), let $\eps>0$, then we can find a $f \in \CoFunc(X)$ with $\|f\|_{\CoFunc(X)} \leq 1$ and  $|\lambda(f)| \geq 1-\eps$.  By multiplying by a phase we may assume $\lambda(f)$ is real and positive, and taking real parts we may assume $f$ is real, then by replacing $f$ with $|f|$ we may assume that $f$ takes values in $[0,1]$.  As $f \in \CoFunc(X)$, there exists a compact subset $K$ of $X$ such that $|f(x)| \leq \frac{1}{2}$ outside of $K$.  By Proposition \ref{urysohn} we may find $\chi \in \CcFunc(X)$ taking values in $[0,1]$ with $\chi=1$ on $K$.  Then for any $g \in \CoFunc(X)$ taking values in $[0,1]$, we have
$$ \lambda(f) + \lambda((1-\chi)g) \leq \|f + (1-\chi)g\|_{\CoFunc(X)} \leq 1$$
and hence $\lambda((1-\chi)g) \leq \eps$.  

Now suppose that $(f_\alpha)_{\alpha \in A}$ is a non-increasing net in $\CoFunc(X)$ whose limit is zero.  We need to show that $\lim_\alpha \lambda(f_\alpha)=0$.  By rescaling we may assume that $f_\alpha$ takes values in $[0,1]$ for at least one $\alpha$, and then for all $\alpha$ after refining the net.  By the previous discussion we have
$$ \lambda((1-\chi) f_\alpha) \leq \eps.$$
Meanwhile, the net $(\chi f_\alpha)_{\alpha \in A}$ of continuous functions has uniform compact support and converges monotonically to zero, hence by Dini's theorem for nets (see, e.g., \cite[p. 239]{kelley}) it converges uniformly.  This implies that $\lim_\alpha \lambda(\chi f_\alpha)=0$, hence
$$ \overline{\lim}_\alpha |\lambda(f_\alpha)| \leq \eps.$$
Since $\eps>0$ is arbitrary, we obtain the claim.

Now we prove (iii).  Define $\lambda(f) \coloneqq \int_X f\ d\mu$ for $f \in \CbFunc(X)$ (resp. $f \in \CoFunc(X)$).  It is clear that $\lambda$ is non-negative and has operator norm at most $1$. From the Radon property we have that for any $\eps>0$ there exists compact $K$ such that $\mu_X(K^c) \leq \eps$, and then by using the cutoff $\chi$ as before one can establish that $\lambda$ has operator norm at least $1-\eps$ for any $\eps>0$, and is hence a state; repeating the previous arguments then also give $\tau$-smoothness. Uniqueness of the represented state $\lambda$ is clear from definition.

Now we prove (iv). Using the identification $\CFunc(\Alex(X)) \equiv \CoFunc(X) \oplus \C$, we can define an extension $\tilde \lambda \colon \CFunc(\Alex(X)) \to \C$ of $\lambda$ by the formula
$$ \tilde \lambda(f + c1) \coloneqq \lambda(f) + c$$
for $f \in \CoFunc(X)$ and $c \in \C$ (where we embed $\CoFunc(X)$ in $\CFunc(\Alex(X))$ in the usual fashion).  It is not difficult to see that $\tilde \lambda$ is non-negative with $\tilde \lambda(1)=1$, hence $\tilde \lambda$ is a state.  Conversely, every state $\tilde \lambda$ on the $\CHpt$-space $\Alex(X)$ has $\tilde \lambda(1)=1$, so the extension is unique by linearity.  

Finally, we show (v). Let $\CbFunc(X)_+$ (resp. $\CoFunc(X)_+$) denote the real nonnegative elements of $\CbFunc(X)$ (resp. $\CoFunc(X)$).  For any $f \in \CbFunc(X)_+$, define
$$ \tilde \lambda(f) \coloneqq \sup \{ \lambda(g): g \in \CoFunc(X)_+, g \leq f \}$$
where we use $g \leq f$ to denote the pointwise domination $g(x) \leq f(x)$ for all $x \in X$.  Since $\lambda$ is a $\CoFunc$-state, we see that $0 \leq \tilde \lambda(f) \leq \|f\|_{\CbFunc(X)}$.  One clearly has superadditivity $\tilde \lambda(f_1+f_2) \geq \tilde \lambda(f_1)+\tilde \lambda(f_2)$ for $f_1,f_2 \in \CbFunc(X)_+$.  Next, observe that if $f_1,f_2 \in \CbFunc(X)_+$ and $g \in \CoFunc(X)_+$ is such that $g \leq f_1+f_2$, then $g=g_1+g_2$ for some $g_1,g_2 \in \CoFunc(X)_+$ with $g_1 \leq f_1$ and $g_2 \leq f_2$; for instance one can take $g_1 \coloneqq \min(f,g_1)$ and $g_2 \coloneqq g-g_1$.  From this we see that we in fact have additivity $\tilde \lambda(f_1+f_2) = \tilde \lambda(f_1) + \tilde \lambda(f_2)$ for nonnegative $f_1,f_2 \in \CbFunc(X)_+$.  We also have the homogeneity property $\tilde \lambda(cf) = c \tilde \lambda(f)$ for $c \geq 0$ and $f \in \CbFunc(X)_+$.  Thus $\tilde \lambda$ extends to a $\CbFunc$-functional on $X$, which we continue to call $\tilde \lambda$.  By construction, $\tilde \lambda$ is non-negative. For any real $f \in \CbFunc(X)$ we then have 
$$ -\|f\|_{\CbFunc(X)} \leq \tilde \lambda(f) \leq \|f\|_{\CbFunc(X)},$$
which implies for any complex $f \in \CbFunc(X)$ and phase $e^{i\theta}$ that
$$ \mathrm{Re} e^{i\theta} \tilde \lambda(f) = \lambda(\mathrm{Re} e^{i\theta} f) \leq \|f\|_{\CbFunc(X)};$$
taking suprema in $\theta$, we conclude that $\tilde \lambda$ has operator norm at most $1$.  Since $\tilde \lambda$ extends $\lambda$ which already had operator norm $1$, we conclude that $\tilde \lambda$ has operator norm exactly equal to $1$ and is hence a $\CbFunc$-state.  If $\tilde \lambda$ were not $\tau$-smooth, then there would exist $\eps>0$ and a non-increasing net $(f_\alpha)_{\alpha \in A}$ of functions $f_\alpha \in \CbFunc(X)_+$ converging pointwise to zero such that $\tilde \lambda(f_\alpha)>\eps$ for all $\eps$.  If we then let $B$ be the collection of all $g \in \CoFunc(X)_+$ such that $\lambda(g) > \eps$ and $g \leq f_\alpha$ for some $\alpha \in A$, ordered by pointwise domination $\leq$, then $(g)_{g \in B}$ is a non-increasing net converging pointwise to zero. Thus $\lambda$ would not be $\tau$-smooth, contradicting (iii).  Thus $\tilde \lambda$ is $\tau$-smooth.

It remains to show that $\tilde \lambda$ is the unique  extension of $\lambda$ to a $\CbFunc$-state.  If $\lambda'$ is another such extension, we see from repeating the proof of (ii) that for any $\eps>0$ there exists $\chi \in \CcFunc(X)$ taking values in $[0,1]$ such that $\lambda'((1-\chi)g) \leq \eps$ and also $\tilde \lambda((1-\chi)g) \leq \eps$ for any $g \in \CbFunc(X)$ taking values in $[0,1]$.  On the other hand $\lambda'$ and $\tilde \lambda$ both agree with $\lambda$ on $\chi g$.  By the triangle inequality we conclude that $|\lambda'(g) - \tilde \lambda(g)| \leq 2\eps$ for all $g \in \CbFunc(X)$ taking values in $[0,1]$, hence on sending $\eps \to 0$ and using linearity we conclude that $\lambda', \tilde \lambda$ are identical.
\end{proof}

We now give the Riesz representation theorems for the categories $\CMet$, $\CH$, $\CHpt$, $\LCHpr$, $\LCH$.  These results are largely contained in prior literature, but are presented here in the notation of this paper.

\begin{theorem}[Riesz representation theorem]\label{rrt} Let $\Cat = \CMet$, $\CH$, $\CHpt$, $\LCHpr$, $\LCH$, and let $X$ be a $\Cat$-space.
\begin{itemize}
\item[(i)]  (Riesz representation theorem)  Every $\CoFunc$-state $\lambda$ on $X$ is represented by a unique promotion of $X$ to a $\Cat\mathbf{Prb}$-space $(X,\mu_X)$. (In other words, for each state $\lambda$ there is a unique Radon measure $\mu_X$ on $X$ such that $\lambda(f) = \int_X f\ d\mu_X$ for all $f \in \CoFunc(X)$.)
\item[(ii)]  (Daniell-Stone representation theorem)  If $\Cat = \LCH$, then every $\tau$-smooth $\CbFunc$-state $\lambda$ on $X$ is represented by a unique promotion of $X$ to a $\Cat\mathbf{Prb}$-space $(X,\mu_X)$.
\item[(iii)]  (Relationship with Stone--{\v C}ech compactification)  If $\Cat = \LCH$ and $\lambda$ is a $\CbFunc$-state on $X$, then there is a unique promotion of $\beta X$ to a $\CH\mathbf{Prb}$-space $(\beta X, \mu_{\beta X})$ such that $\lambda(f\downharpoonright_X) = \int_{\beta X} f\ d\mu_{\beta X}$ for all $f \in \CFunc(\beta X)$ (where $f\downharpoonright_X$ is the restriction of $\beta X$ to $X$, where we identify the latter with a subspace of the former).  Furthermore, the $\CbFunc$-state $\lambda$ is $\tau$-smooth if and only if $\beta X \backslash X$ has zero outer measure in the sense that
$$ \inf \{ \mu_{\beta X}( E ): E \in \Sigma_{\beta X}, E \supset \beta X \backslash X \} = 0.$$
\end{itemize}
\end{theorem}

We refer to \cite{knowles} for a further study of how the Riesz representation theorem interacts with the Stone--{\v C}ech compactification.  For instance, the second part of Theorem \ref{rrt}(iii) is essentially \cite[Theorem 2.4]{knowles}.

\begin{proof}  The claim (i) for $\Cat = \CH$ can be found for instance in \cite[\S 2]{varadarajan-riesz}, \cite[Theorem 3.3]{sunder}, \cite{hartig}, \cite[Theorem 7.4.1]{dudley01}, \cite[Theorem 5.7]{EFHN}, or \cite{garling}.  This implies the cases $\Cat = \CMet, \CHpt$ of claim (i) after applying forgetful casting functors.  We remark that the presentation in \cite{hartig} is particularly compatible with the category-theoretic viewpoint adopted in this paper.

Now we show (i) for $\Cat = \LCHpr$.  This result appears for instance in \cite[Theorem 4.1]{sunder} or \cite[\S 3]{varadarajan-riesz}, but for the convenience of the reader we give a proof here.  We begin with existence.  By Proposition \ref{func-prop}(iv), we can extend $\lambda$ to a state $\tilde \lambda \colon \CFunc(\Alex(X)) \to \C$ on $\Alex(X)$ (viewing $X$ as a subspace of $\Alex(X)$ and $\CoFunc(X)$ as a subalgebra of $\CFunc(\Alex(X))$).  By the $\Cat=\CH$ case of (i), $\tilde \lambda$ is represented by a Radon probability measure $\mu_{\Alex(X)}$ on $\Alex(X)$.  Now let $K$ be a compact $G_\delta$ subset of $X$.  From Proposition \ref{baire-prop} (as well as the Tietze extension theorem and Urysohn's lemma) we see that the $\CoFunc$-Baire $\sigma$-algebras $\BaireC(X), \BaireC(\Alex(X))$ both agree with $\BaireC(K) = \Baire(K)$ when restricted to $K$.  Thus $\mu_{\Alex(X)}$ may be restricted to a finite Baire measure $\mu_K$ on the $\CH$-space $K$.  These measures are compatible with each other in the sense that $\mu_{K'}$ is the restriction of $\mu_K$ to $K'$ whenever $K'$ is a compact $G_\delta$ subset of $K$.  We now define a $\CoFunc$-Baire measure $\mu_X$ on $X$ by
$$ \mu_X( E ) \coloneqq \sup \{ \mu_K(E \cap K) \colon K \subseteq X, \hbox{ compact } G_\delta \}.$$
Since each $\mu_K$ is countably additive of total mass at most one, one easily verifies that $\mu_X$ is countably additive also with total mass at most one.  As $\mu_{\Alex(X)}$ is compact $G_\delta$ inner regular on $X$, each $\mu_K$ is compact $G_\delta$ inner regular on $K$, which then implies that $\mu_X$ is compact $G_\delta$ inner regular.  If $f \in \CcFunc(X)$, then from Proposition \ref{urysohn} $f$ is supported in some compact $G_\delta$ set $K$, and
$$ \int_X f\ d\mu_X = \int_K f\ d\mu_K = \int_{\Alex(X)} f\ d\mu_{\Alex(X)} = \tilde \lambda(f) = \lambda(f).$$
Thus $\lambda$ is represented by $\mu_X$ on $\CcFunc(X)$, and hence also on $\CoFunc(X)$ by taking uniform closures.  Since $\lambda$ is a state, $\mu_X$ must therefore have total mass one, and is thus a Radon probability measure as required.

To show uniqueness, observe that if $\lambda$ is represented by any other Radon probability measure $\mu'_X$ on $X$, then by the uniqueness aspect in the $\Cat=\CH$ case of (i), $\mu'_X$ must agree with $\mu_X$ on each compact $G_\delta$ set $K$, and then by compact $G_\delta$ inner regularity $\mu'_X$ and $\mu_X$ must be identical.  

Now we establish (ii).  By \cite[Theorem 7.8.6]{bogachev2006measure} (see also \cite[Theorem 4.5.2]{dudley01}), there exists a unique $\tau$-additive probability measure $\mu_X$ on $\BaireC(X)$ that represents $\lambda$. The claim now follows from Proposition \ref{automatic}(ii).

Now, we establish (i) for $\Cat = \LCH$.  By Proposition \ref{func-prop}(v), we can extend $\lambda$ to a $\tau$-smooth $\CbFunc$-state $\tilde \lambda$ on $X$, which by (ii) is represented by a Radon probability measure $\mu_X$ on $X$.  Hence the $\CoFunc$-state $\lambda$ is also represented by $\mu_X$.  If $\lambda$ is represented by another Radon probability measure $\mu'_X$, then from dominated convergence and $\tau$-smoothness we see that $\tilde \lambda$ is also represented by $\mu'_x$, hence $\mu_X = \mu'_X$ by (ii), giving uniqueness.

Finally, we establish (iii).  Every function in $\CbFunc(X)$ has a unique extension to $\CFunc(\beta X)$, hence the $\CbFunc$-state $\lambda$ on $X$ can be identified with a state on $\beta X$.  The existence and uniqueness of the promotion $(\beta X, \mu_{\beta X})$ then follows from the $\Cat = \CH$ case of (i).  If $\lambda$ is $\tau$-smooth, then by (ii) $\lambda$ is also represented by a Radon probability measure $\mu_X$ on $X$, hence for any $\eps>0$ there is a compact $G_\delta$ subset $K$ of $X$ with $\mu_X(K) \geq 1-\eps$.  From the $\Cat=\CH$ case of (i) we see that $\mu_X$ and $\mu_{\beta X}$ must agree when restricted to $K$, thus $\mu_{\beta X}(K) \geq 1-\eps$, or equivalently $\mu_{\beta X}(\beta X \backslash K) \leq \eps$.  Hence $\beta X \backslash X$ has zero outer measure.

Conversely, if $\beta X \backslash X$ has zero outer measure, then by the Radon property for every $\eps$ there exists a compact $G_\delta$ subset $K$ of $\beta X$ contained in $X$ such that $\mu_{\beta X}(K) \geq 1-\eps$, or equivalently $\mu_{\beta X}(\beta X \backslash K) \leq \eps$.  From Proposition \ref{urysohn} we also see that $K$ is a compact $G_\delta$ subset of $X$.  Arguing using Dini's theorem as in the proof of Proposition \ref{func-prop}(iii) we conclude that $\lambda$ is $\tau$-smooth.
\end{proof}

As a corollary of the Riesz representation theorem, one can extend Radon measures on the smaller $\sigma$-algebras in \eqref{baire-include} to larger ones in a canonical fashion:

\begin{corollary}[Canonical extension]\label{canon}  Let $X$ be an $\LCH$-space (and hence also an $\LCHpr$-space).
\begin{itemize}
\item[(i)] Any Radon probability measure on $(X_\Set, \BaireC(X))$ has a unique extension to a Radon probability measure on $(X_\Set,\BaireB(X))$.
\item[(ii)] Any Radon probability measure on $(X_\Set,\BaireB(X))$ has a unique extension to a compact inner regular measure on $(X_\Set,\Borel(X))$.
\end{itemize}
\end{corollary}

For $\CH$-spaces this corollary is well known (see, e.g., \cite[Theorem 7.3.1]{dudley01}, \cite[Proposition 5.4]{EFHN}).

\begin{proof}  We begin with (i).  If $\mu_X$ is a Radon probability measure on $(X_\Set, \BaireC(X))$, then by Proposition \ref{func-prop}(iii) it represents a $\CoFunc$-state $\lambda$ on $X$ (viewed as an $\LCHpr$-space).  By Theorem \ref{rrt}(i), $\lambda$ is also represented by a Radon probability measure $\tilde \mu_X$ on $(X_\Set, \BaireB(X))$.  By construction, $\int_X f\ d\mu_X = \int_X f\ d\tilde \mu_X$ for all $f \in \CcFunc(X)$. By Proposition \ref{urysohn}, if $K$ is a compact $G_\delta$ subset of $X$, then $1_K$ can be expressed as the pointwise limit of a decreasing sequence of functions $\CcFunc(X)$, thus by monotone convergence $\mu_X, \mu'_X$ agree on compact $G_\delta$ functions, hence on all $\CoFunc$-Baire functions by the Radon property.  The Radon property also ensures uniqueness of the extension (here we use the fact from Proposition \ref{baire-prop} that compact $G_\delta$ functions are $\CoFunc$-Baire measurable).

Now we prove (ii).  If $\mu'_X$ is a Radon probability measure on $(X_\Set, \BaireB(X))$, then by Proposition \ref{func-prop} it represents a $\CoFunc$-state $\lambda$ on $X$ (viewed as an $\LCH$-space).  By the Riesz--Markov--Kakutani theorem \cite{kakutani}, there is a compact inner regular probability measure $\mu''_X$ on $(X_\Set, \Borel(X))$ which represents $\lambda$.  By arguing as before we see that $\mu''_X, \mu'_X$ agree on compact $G_\delta$ sets, which by the regularity properties implies that $\mu''_X(E) \geq \mu'_X(E)$ for all $\CbFunc$-Baire $E$.  Taking complements we also have $\mu''_X(E) \leq \mu'_X(E)$.  Thus $\mu''_X$ extends $\mu'_X$.  The uniqueness of the extension follows from the uniqueness aspect of the Riesz--Markov--Kakutani theorem.
\end{proof}

The following examples show that the relationship between states and probability measures deteriorates if hypotheses such as the Radon property, $\tau$-additivity, or $\tau$-smoothness are dropped.

\begin{example}[Generalized limit functionals]\label{ultra-ex}  As is well-known, the Hahn-Banach theorem allows one to (non-uniquely) extend the limit functional $\lim \colon \CoFunc(\N) \oplus \C \to \C$ to a generalized limit functional $\lambda \colon \CbFunc(\N) \to \C$ which is a $\CbFunc$-state on $\N$.  Such a state is not $\tau$-smooth: indeed, the sequence of indicator functions $1_{n \geq N}$ for $N \in \N$ is non-decreasing and converges pointwise to zero in $\N$, but $\lambda(1_{n \geq N}) = 1$ does not converge to zero.  In particular, $\lambda$ is \emph{not} represented by any Radon probability measure on $\N$.  Indeed, the restriction of $\lambda$ to $\CoFunc(\N)$ is zero, so any probability measure that could represent $\lambda$ would vanish, which is absurd.  On the other hand, identifying $\CbFunc(\N)$ with $\CFunc(\beta \N)$, we see that $\lambda$ will be represented by a Radon probability measure on $\beta \N$, but this measure will assign full measure to $\beta \N \backslash \N$, so the second part of Theorem \ref{rrt}(iii) will not apply. (Conversely, any Radon probability measure supported on $\beta \N \backslash \N$ generates a generalized limit functional.)
\end{example}

\begin{example}[Dieudonn\'e's measure]\label{dieu-ex} Let ${\mathcal F}$ be the collection of unbounded closed subsets of $[0,\omega_1)$.  We claim that this collection is closed under countable intersections.  Indeed, if $(F_n)_{n \in \N}$ is a sequence of unbounded closed subsets, then $F \coloneqq \bigcap_{n \in \N} F_n$ is closed.  If for contradiction $F$ is bounded by some countable ordinal $\alpha$, then by repeatedly using the unbounded nature of the $F_n$ we can find countable ordinals $\alpha_{j,n} > \alpha$ in $F_n$ for all $j,n \in \N$ such that $\alpha_{j+1,n} > \sup_m \alpha_{j,m}$ for all $j,n \in \N$.  The countable ordinal $\sup_{j,n} \alpha_{j,n}$ is equal to $\sup_j \alpha_{j,n}$ for every $n$, hence is greater than $\alpha$ and  lies in every $F_n$ and hence in $F$, contradicting the choice of $\alpha$.

One can check that each Borel subset of $[0,\omega_1)$ either contains an element of ${\mathcal F}$, or is disjoint from an element of ${\mathcal F}$, but not both, by first verifying this for closed sets and then noting that the claim is preserved by $\sigma$-algebra operations.  Define \emph{Dieudonn\'e measure}\footnote{An early appearance of this example (in the Borel case) is in \cite[\S 53]{halmos-measure-theory} as Exercise 10. In the literature, the example is attributed to Dieudonn\'e (e.g., see \cite{bogachev2006measure,rudinrealcomplex}), and the related reference cited is \cite{dieudonnemeasure}.} $\mu_{[0,\omega_1)}$ on $([0,\omega_1), \Borel([0,\omega_1))$ by setting $\mu_{[0,\omega_1)}(E)$ to equal $1$ when $E$ contains an element of ${\mathcal F}$ and $0$ when $E$ is disjoint from an element of ${\mathcal F}$.  Then the above properties ensure that 
$\mu_{[0,\omega_1)}$ is a probability measure, which then represents a $\CbFunc$-state $\lambda$, which by Proposition \ref{omega1}(v) assigns to each $f \in \CbFunc(X)$ the limiting value of $f$ at $\omega_1$.  If we define \emph{Dieudonn\'e measure} $\mu_{[0,\omega_1]}$ on $[0,\omega_1]$ to be the extension of $\mu_{[0,\omega_1)}$ to $([0,\omega_1], \Borel([0,\omega_1])$ by giving $\{\omega_1\}$ zero mass, we then see that $\mu_{[0,\omega_1]}$ is a Borel probability measure that represents the same state on $[0,\omega_1]$ as the Dirac measure $\delta_{\omega_1}$, despite the two measures differing on Borel sets (although they do agree on Baire subsets of $[0,\omega_1]$, in accordance with Theorem \ref{rrt} and Proposition \ref{automatic}(iii)).  The state $\lambda$ also vanishes on $\CoFunc([0,\omega_1))$, but this is not a contradiction because $\lambda$ is not $\tau$-smooth (and $\mu_{[0,\omega_1)}$ is not $\tau$-additive or Radon).
\end{example}

\begin{example}\label{unc-ex} Let $X$ be an uncountable discrete $\LCH$-space, then $\BaireC(X)$ is the countable-cocountable $\sigma$-algebra (consisting of countable sets and their complements), while $\BaireB(X) = \Borel(X)$ is the discrete $\sigma$-algebra (since every indicator function is bounded continuous).  One can then check that a probability measure on $(X,\BaireC(X))$ is Radon iff it is $\tau$-additive iff it is supported on an at most countable set.  (For instance, the probability measure that assigns $0$ to countable sets and $1$ to cocountable sets has none of these properties.)  Meanwhile, a probability measure on $(X,\BaireB(X)) = (X,\Borel(X))$ is Radon iff it is compact inner regular iff it is $\tau$-additive iff it is supported on an at most countable set.  This is of course consistent with Proposition \ref{automatic} and Corollary \ref{canon}.
\end{example}

Now we can establish the Riesz duality analogues of the Gelfand dualities in Figures \ref{fig:gelfand-dual}, \ref{fig:gelfand-forget}.  If ${\mathcal A}$ is a $C^*$-algebra, define a \emph{state} on ${\mathcal A}$ to be a bounded linear functional $\tau \colon {\mathcal A} \to \C$ which is non-negative (it maps non-negative elements to non-negative reals) and is of operator norm $1$.   Note that this is consistent with the definition of a state for the algebras $\CoFunc(X), \CbFunc(X)$ in Definition \ref{linfunc}(iii).  We need a technical lemma:

\begin{lemma}[Extension of states]\label{ext-states}\   
\begin{itemize}
\item[(i)]  Let ${\mathcal A}$ be a $\CStarAlgMult$-algebra.  Then every state $\tau_{{\mathcal A}}$ on ${\mathcal A}$ has a unique extension $\tau_{\Mult({\mathcal A})}$ to a state on $\Mult({\mathcal A})$.
\item[(ii)]  Let ${\mathcal A}$ be a $\CStarAlgNd$-algebra.  Then every state $\tau_{{\mathcal A}}$ on ${\mathcal A}$ has a unique extension $\tau_{\Unit({\mathcal A})}$ to a state on $\Unit({\mathcal A})$.
\end{itemize}
\end{lemma}

\begin{proof}  We first prove (i).  By Gelfand duality (Theorem \ref{gelfand-dualities}) we may assume that ${\mathcal A} = \CoFunc(X)$ for some $\LCH$-space $X$, in which case we can identify $\Mult({\mathcal A})$ with $\CbFunc(X)$.  The claim now follows from Proposition \ref{func-prop}(v).  One can also avoid Gelfand duality by using approximate units of ${\mathcal A}$ as a substitute for the cutoff functions $\chi$ that arise in the proof of Proposition \ref{func-prop}(v); we leave this alternate argument to the interested reader.

The proof of (ii) is completely analogous, using Proposition \ref{func-prop}(iv) in place of Proposition \ref{func-prop}(v).  Alternatively, one can extend the trace directly via the formula $\tau_{\Unit({\mathcal A})}(a+c1) \coloneqq \tau_{\mathcal A}(a)+c$; we leave the details to the interested reader.
\end{proof}

We can now attach traces to the categories $\CStarAlgUnitInf$, $\CStarAlgUnit$, $\CStarAlgNd$, $\CStarAlgMult$ to obtain new categories  $\CStarAlgUnitInfTrace$, $\CStarAlgUnitTrace$, $\CStarAlgNdTrace$, $\CStarAlgMultTrace$, in a manner dual to how probability measures were attached to the categories $\CHpt$, $\CH$, $\LCHpr$, $\LCH$:

\begin{definition}[Tracial commutative $C^*$-algebra categories]\label{tracialcstar-def}  Let $\Cat$ be equal to $\CStarAlgUnitInf$, $\CStarAlgUnit$, $\CStarAlgNd$, or $\CStarAlgMult$.  Let $\Cat'$ be the Gelfand dual $\Cat = \CHpt$, $\CH$, $\LCHpr$, $\LCH$ to $\Cat$, thus we have  functors $\CoFunc \colon \Cat' \to   \Cat^\op$ and $\Spec \colon \Cat^\op \to \Cat'$ (note that we can write $\CoFunc$ as $\CFunc$ if $\Cat' = \CHpt$, $\CH$). Let $\Cat{\bf Prb}$ be the category defined in Definition \ref{top-prob-cat}.
\begin{itemize}
\item[(i)] A \emph{$\Cat^\tau$-algebra} is a pair ${\mathcal A} = ({\mathcal A}_{ \Cat}, \tau_{\mathcal A})$, where ${\mathcal A}_{ \Cat}$ is a $\Cat$-algebra and $\tau_{\mathcal A} \colon {\mathcal A}_{\Cat} \to \C$ is a state.  
\item[(ii)] A \emph{$\Cat^\tau$-morphism} $\Phi \colon {\mathcal A} \to {\mathcal B}$ between two $\Cat^\tau$-algebras ${\mathcal A} = ({\mathcal A}_{\Cat}, \tau_{\mathcal A})$, ${\mathcal B} = ({\mathcal B}_{\Cat}, \tau_{\mathcal B})$ is a ${\Cat}$-morphism $\Phi_{ \Cat} \colon {\mathcal A} \to {\mathcal B}$ which is required to obey the relation
\begin{equation}\label{taub}
 \tau_{\mathcal B} \circ \Phi_{ \Cat} = \tau_{\mathcal A}
 \end{equation}
if ${\Cat} = \CStarAlgUnitInf$, $\CStarAlgUnit$, $\CStarAlgNd$.  When ${ \Cat} = \CStarAlgMult$ one cannot impose \eqref{taub} because the morphism $\Phi_{ \Cat}$ describes a function $\tilde \Phi_{ \Cat}$ from ${\mathcal A}$ to $\Mult({\mathcal B})$, rather than a function from ${\mathcal A}$ to ${\mathcal B}$. Instead, one instead imposes the slightly different relation
$$ \tau_{\Mult(\mathcal B)} \circ \Mult(\Phi_{ \Cat}) = \tau_{\Mult(\mathcal A)}$$
where the extended states $\tau_{\Mult(\mathcal A)}$, $\tau_{\Mult(\mathcal B)}$ are defined by Lemma \ref{ext-states}.
\item[(iii)] One defines a   forgetful functor from ${ \Cat^\tau}$ to ${ \Cat}$ in the obvious fashion.
\item[(iv)]  If $X = (X_{\Cat'}, \mu_X)$ is a $\Cat'{\bf Prb}$-space, we define $\CoFunc(X)$ to be the ${ \Cat^\tau}$-algebra $\CoFunc(X) \coloneqq (\CoFunc(X_{\Cat'}), \tau)$, where $\tau$ is the $\CoFunc(X_{\Cat'})$-state represented by $\mu_X$.  If $T \colon X \to Y$ is a $\Cat'{\bf Prb}$-morphism, we define $\CoFunc(T) \colon \CoFunc(Y) \to \CoFunc(X)$ to be the unique
${\Cat^\tau}$-morphism from $\CoFunc(Y)$ to $\CoFunc(X)$ with $\CoFunc(T)_{\Cat}$ $= \CoFunc(T_{\Cat'})$.  When $\Cat' = \CH, \CHpt$ we abbreviate $\CoFunc$ as $\CFunc$.
\item[(v)]  If ${\mathcal A} = ({\mathcal A}_{\Cat}, \tau_{\mathcal A})$ is a ${\Cat^\tau}$-algebra, we define $\Riesz({\mathcal A})$ to be the $\Cat'{\bf Prb}$-space $(\Spec({\mathcal A}), \mu)$, where $\mu$ is the unique Radon probability measure on $\Spec({\mathcal A})$ that represents $\tau_{\mathcal A}$ (after using Gelfand duality to identify ${\mathcal A}$ with $\CoFunc(\Spec({\mathcal A}))$), as guaranteed by Theorem \ref{rrt}(i).  If $\Phi :{\mathcal A} \to {\mathcal B}$ is a ${\Cat^\tau}$-morphism, we define $\Spec(\Phi) \colon \Spec({\mathcal B}) \to \Spec({\mathcal A})$ to be the unique $\Cat'{\bf Prb}$-morphism such that $\Spec(\Phi)_{\Cat'} = \Spec(\Phi_{ \Cat})$.
\end{itemize}
\end{definition}

By using Gelfand duality and Theorem \ref{rrt}(i) (and also Lemma \ref{ext-states} in the case $\Cat = \CStarAlgMult$), we can verify that $\Cat^\tau$ is indeed a category, and that the functors $\CoFunc \colon \Cat{\bf Prb} \to (\Cat^\tau)^\op$ and $\Riesz \colon (\Cat^\tau)^\op \to \Cat{\bf Prb}$ form a duality of categories; we refer to these dualities of categories as ``Riesz dualities''.  The horizontal functors on the first row of Figures \ref{fig:gelfand-dual}, \ref{fig:gelfand-forget} extend in an obvious fashion to their tracial counterparts (using Lemma \ref{ext-states} as necessary), which by Riesz duality then allows one to analogously extend the functors on the second row as well to their probabilistic counterparts, and similarly for the ``diagonal'' functor $\CbFunc$.  Routine verification then gives

\begin{theorem}[Riesz dualities]\label{riesz-dualities}  The categories in Figures \ref{fig:riesz-dual}, \ref{fig:riesz-forget} are indeed categories, and the functors in these figures are indeed functors between the indicated categories, with the indicated faithfulness and  fullness properties.  Furthermore, both of these diagrams commute up to natural isomorphisms. (In particular, each pair of vertical functors generates a duality of categories.)
\end{theorem}

\begin{remark} Corollary \ref{canon}(i) can be interpreted category-theoretically as guaranteeing the existence of the ``forgetful functor'' from $\LCHprProb$ to $\LCHProb$ that appears in Figures \ref{fig:riesz-dual}, \ref{fig:riesz-forget}. 
Theorem \ref{rrt}(iii) (and Proposition \ref{func-prop}(iii)) can similarly be interpreted as a guarantee for the existence of the functor $\beta \colon \LCHProb \to \CHProb$.
\end{remark}

\section{Abstract probability theory}\label{abs-prob-sec}

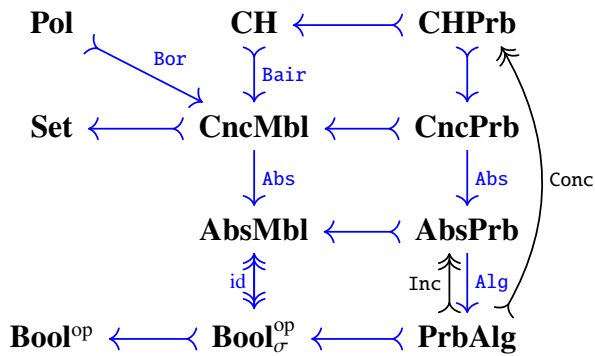
\begin{figure}
    \centering
    \begin{tikzcd}
    \Polish \arrow[dr,blue,tail,"\BorelFunc"] & \CH \arrow[d,blue,tail, "\BaireFunc"] & \CHProb \arrow[l,blue,tail] \arrow[d,blue,tail] \\
     \Set & \ConcMes \arrow[l,blue,tail] \arrow[d,blue,"\Abs"] & \ConcProb \arrow[l,blue,tail] \arrow[d,blue,"\Abs"] \\
     & \AbsMes \arrow[d,blue,two heads] & \AbsProb \arrow[l,blue,tail] \arrow[d,blue,"\Mes"] \\ 
      \Boolop & \SigmaAlgop \arrow[l,blue,tail] \arrow[u,blue, two heads, "\id"] & \arrow[l,blue,tail] \ProbAlg  \arrow[u,"\Inc",shift left=1.25ex, tail,two heads] \arrow[uuu,bend right,"\Stone"',  tail, two heads, shift right = 1.5ex]
    \end{tikzcd}
    \caption{Basic functors between concrete and abstract probabilistic and measurable categories.  All the casting functors (displayed in blue) commute with each other, but the non-casting functors (displayed in black) need not commute with the rest of the diagram. Note that the categories in the first two rows have a faithful casting functor to $\Set$ and are thus concrete categories, while the other categories in this diagram should be viewed as being more abstract in nature. The canonical model functor $\Stone$, which crucially allows one to return from an abstract category to a concrete one, will be constructed in the next section. }
    \label{fig:prob-mes}
\end{figure}

In previous sections we have already seen the categories $\ConcMes, \ConcProb$ of concrete measurable spaces and concrete probability spaces respectively, as well as their compact Hausdorff counterparts $\CH$ and $\CHProb$.  Being concrete, these categories also have faithful forgetful functors to $\Set$.   In this section we introduce some more abstract categories of measurable and probability spaces (and their associated Boolean algebras) that we will use in the sequel.  These categories are summarized in Figure \ref{fig:prob-mes}.

\begin{definition}[Abstract categories]\label{abscat-def}\ 
\begin{itemize}
\item[(i)]  A \emph{$\Bool$-algebra} is an abstract Boolean algebra $\B = ({\mathcal B}, 0, 1, \wedge, \vee, \overline{\cdot})$.  A \emph{$\Bool$-morphism} is a Boolean algebra homomorphism between Boolean algebras, with the usual composition law.
\item[(ii)]  $\SigmaAlg$ is the subcategory of $\Bool$ in which the $\SigmaAlg$-algebras are those $\Bool$-algebras $\B$ which are $\sigma$-complete (every countable family $(E_n)_{n \in \N}$ in $\B$ has a meet $\bigvee_{n \in \N} E_n$ and a join $\bigwedge_{n \in \N} E_n$), and the $\SigmaAlg$-morphisms $\Phi \colon \B \to \B'$ are those $\Bool$-morphisms which preserve countable meets and joins, thus $\Phi\left( \bigvee_{n \in \N} E_n \right) = \bigvee_{n \in \N}\Phi(E_n)$ and $\Phi\left( \bigwedge_{n \in \N} E_n \right) = \bigwedge_{n \in \N}\Phi(E_n)$ for $E_n \in \B$.
\item[(iii)]  $\AbsMes$ is the opposite category to $\SigmaAlg$ (as defined in Definition \ref{opp-cat}). An $\AbsMes$-space is also called an \emph{abstract measurable space}, and an $\AbsMes$-morphism an \emph{abstractly measurable map}. 
\item[(iv)]  Let $\mathcal{A}\in \SigmaAlg$. A probability measure on $\mathcal{A}$ is a function $\mu \colon \mathcal{A} \to [0,1]$ such that $\mu(0)=0$, $\mu(1)=1$, and $\mu\left(\bigvee_{n \in \N} E_n\right) = \sum_{n \in \N} \mu(E_n)$ whenever $E_n$ are pairwise disjoint elements of $\mathcal{A}$ (thus $E_n \wedge E_m=0$ for all distinct $n,m \in \N$). 
We define the functor $\mathtt{Prb}\colon \AbsMes\to \Set$ to assign to a $\sigma$-algebra $\mathcal{A}$ the set $\mathtt{Prb}(\mathcal{A})$ of probability measures on $\mathcal{A}$ and to an $\AbsMes$-morphism $f\colon \mathcal{B}\to \mathcal{A}$ the pushforward map $\mathtt{Prb}(f)\colon \mathtt{Prb}(\mathcal{B})\to \mathtt{Prb}(\mathcal{A})$ defined as the pullback $\mathtt{Prb}(f)(\mu)=\mu\circ f$. 
We define the category $\AbsProb$ of \emph{abstract probability spaces and abstract measure-preserving maps} to be the action category $\AbsMes\ltimes \mathtt{Prb}$ (see Definition \ref{def-action}). 
\item[(v)]  We define the category $\ProbAlg$ of \emph{probability algebras}\footnote{These are special cases of \emph{measure algebras}, in which the measure $\mu$ is not required to map $1$ to $1$.} as a (non-full) subcategory of the category $\AbsProb=\AbsMes\ltimes \mathtt{Prb}$, where we additionally require that the set $\mathtt{Prb}(\mathcal{A})$ only contains probability measures on the $\sigma$-algebra $\mathcal{A}$ that have the additional property that $\mu(E) > 0$ whenever $E \in  \mathcal{A}$ is non-zero.  We let $\Inc$ be the faithful functor from $\AbsProb$ to $\ProbAlg$. 
\item[(vi)]  There are obvious forgetful functors from $\AbsProb$ to $\AbsMes$ and from $\SigmaAlg$ to $\Bool$. 
\item[(viii)]  If $X = (X_\Set, \Sigma_X)$ is a $\ConcMes$-space, we define the abstraction $\Abs(X)$ to be the $\AbsMes$-space $\Sigma_X$, where the $\sigma$-algebra $\Sigma_X$ of $X$ is viewed as an abstract $\sigma$-algebra.  Similarly if $T \colon X \to Y$ is a $\ConcMes$-morphism, we define $\Abs(T)$ to be the $\AbsMes$-morphism $\Abs(T)\colon \Sigma_X\to \Sigma_Y$ corresponding to the $\SigmaAlg$-morphism $T_\SigmaAlg \colon \Sigma_Y\to \Sigma_X$ defined as the pullback map $T_\SigmaAlg(E) \coloneqq T^*(E)$ for $E \in \Sigma_Y$.  The abstraction functor $\Abs$ from $\ConcProb$ to $\AbsProb$ is defined similarly.
\item[(ix)]  If $X = (X_\SigmaAlg, \mu_X)$ is an $\AbsProb$-space, we define $\Mes(X)$ to be the $\ProbAlg$-space $\Mes(X) \coloneqq (X_\SigmaAlg/_{\mathcal{N}_X}, \mu_X)$, where $X_\SigmaAlg/_{\mathcal{N}_X}$ denotes the quotient $\sigma$-algebra with respect to the $\sigma$-ideal ${\mathcal N}_X \coloneqq \{ E \in X_\SigmaAlg \colon \mu_X(E)=0\}$, also called the \emph{null ideal}, and $\mu_X : X_\SigmaAlg/_{\mathcal{N}_X} \to [0,1]$ is the descent of $\mu_X: X_\SigmaAlg \to [0,1]$ to $X_\SigmaAlg/_{\mathcal{N}_X}$ (by an abuse of notation). 
If $T \colon X \to Y$ is an $\AbsProb$-morphism, we define the $\ProbAlg$-morphism $\Mes(T) \colon \Mes(X) \to \Mes(Y)$ by the commutative diagram
\[\begin{tikzcd}
	{\mathtt{Alg}(X)} && {\mathtt{Alg}(Y)} \\
	\\
	X && Y
	\arrow["{\iota_X}"', from=1-1, to=3-1]
	\arrow["{\mathtt{Alg}(T)}", from=1-1, to=1-3]
	\arrow["{\iota_Y}", from=1-3, to=3-3]
	\arrow["T"', from=3-1, to=3-3]
\end{tikzcd}\]
where $\iota_X,\iota_Y$ denote the canonical inclusions.   
\end{itemize}
\end{definition}

It is a routine matter to check that this defines categories and functors as depicted in Figure \ref{fig:prob-mes} (with the exception of $\Stone$), with the indicated faithfulness and fullness properties, and with all the casting functors (depicted in blue) commuting with each other.  

Informally, the abstraction functors $\Abs$ ``abstract away the points'' from a concrete measurable or probability space, and the probability algebra functor $\Mes$ ``deletes the null sets''.  The faithful functor $\Inc \colon \ProbAlg \to \AbsProb$ re-interprets a probability algebra as a special type of abstract probability space, namely one in which there are no non-trivial null sets. 

\begin{remark}  The $\SigmaAlg$-algebra $\tilde {\mathcal A} = \Forget_{\ProbAlg \to \SigmaAlgop}({\mathcal A})$ associated to a $\ProbAlg$-algebra ${\mathcal A}$ has stronger properties than $\sigma$-completeness; as is well-known, these Boolean algebras are in fact complete and obey the countable chain condition (see, e.g., \cite[322G, 322C]{fremlinvol3}).  Also, the requirement that the $\SigmaAlg$-morphism $\tilde \Phi \colon \tilde {\mathcal A} \to \tilde {\mathcal B}$ associated to a $\ProbAlg$-morphism $\Phi \colon {\mathcal B} \to {\mathcal A}$ be $\sigma$-complete can be dropped as it follows automatically from the Boolean homomorphism hypothesis. These facts are easy to establish, but we shall not do so here as they will not be needed in our arguments.
\end{remark}

\begin{remark}
One can also enlarge the category $\ProbAlg$ by replacing the class of $\ProbAlg$-morphisms with the more general class of Markov operators, which is the abstract analogue of the class of probability kernels on $\ConcProb$-spaces.  These categories are studied in \cite[Ch. 13]{EFHN} and \cite{voevodsky} respectively. However, we will not investigate these categories further here.
\end{remark}

For future reference we develop some of the basic category-theoretic properties of these abstract categories, focusing on the classification of monomorphisms and epimorphisms, and the structure of products. We begin with the Boolean categories.

\begin{lemma}[Properties of Boolean categories]\label{bool-props} Let ${ \Cat} = \Bool, \SigmaAlg$. 
\begin{itemize}
    \item[(i)]  A ${ \Cat}$-morphism is a ${ \Cat}$-monomorphism (resp. ${ \Cat}$-epimorphism) if and only if it is injective (resp. surjective).  
    Any ${ \Cat}$-bimorphism is a ${ \Cat}$-isomorphism. 
    \item[(ii)] Dually, an $\AbsMes$-morphism $f$ is a monomorphism (resp. epimorphism) if and only if $f$ is surjective (resp.~injective), and any $\AbsMes$-bimorphism is an $\AbsMes$-isomorphism. 
    \item[(iii)]  If $\Phi$ is a $\ProbAlg$-morphism, then $\Inc(\Phi)_\AbsMes$ is an $\AbsMes$-epimorphism and $\Phi$ is a $\ProbAlg$-epimorphism. 
    \item[(iv)]  The category ${ \Cat}$ admits categorical coproducts for arbitrarily many factors. As a consequence, $\AbsMes$ admits a categorical products for arbitrarily many factors. 
    \item[(v)]  The categorical $\SigmaAlg$-coproduct contains the categorical $\Bool$-coproduct, but does not agree with it (in the sense of Definition \ref{def-monoidal}, where following Example \ref{exp-cartesian}, we interpret $\Bool$ and $\SigmaAlg$ as cocartesian symmetric monoidal categories) with respect to the forgetful functor from $\SigmaAlg$ to $\Bool$. 
\end{itemize}
\end{lemma}

\begin{remark}\label{rem-prod-probalg}
There is no categorical product in $\ProbAlg$ (for reasons similar to Remark \ref{no-univ-prob}), but one can build infinite tensor products in $\ProbAlg$ in the sense of Definition \ref{def-infinit-prod} following a similar construction as in Example \ref{concprob-ex}.  
\end{remark}

\begin{proof} The ``if'' portion of the first part of (i) follows from Proposition \ref{epimorph}, since there is a faithful forgetful functor from ${ \Cat}$ to $\Set$.  The ``only if'' portion of the first part of (i) for monomorphisms follows by observing that elements $E$ of a ${ \Cat}$-algebra $\X$ can be identified with ${ \Cat}$-morphisms from the $2^2$-element ${ \Cat}$-algebra $2^{\{0,1\}}$ to $\X$.  The ``only if'' portion of the first part of (i) for epimorphisms was established in\footnote{We are indebted to Badam Baplan for this reference.} \cite{lagrange}.  The last part of (i) follows from the first part, since any bijective ${ \Cat}$-morphism is clearly a ${ \Cat}$-isomorphism.

The claim (ii) is immediate from (i) applied to ${ \Cat} = \SigmaAlg$.

Now we prove (iii). If $\Phi \colon (\Y,\nu)  \to (\X,\mu)$ is a $\ProbAlg$-morphism, then we have $\nu(\Inc(\Phi)(E)) = \mu(E)$ for all $E \in \X$.  As $(\X,\mu), (\Y,\nu)$ are both $\ProbAlg$-algebras, this implies that $\Inc(\Phi)(E) = 0$ if and only if $E=0$.  Thus $\Inc(\Phi)$ is an injective $\SigmaAlg$-morphism, hence $\Inc(\Phi)$ is an $\AbsMes$-epimorphism by (ii).  By  Lemma \ref{epimorph}, $\Phi$ is thus also a $\ProbAlg$-epimorphism.  

The claim (iv) is a special case of the fact that categories of algebraic structures with a \emph{set} of operations admit all small limits and colimits, in particular, arbitrary categorical products and coproducts, see \cite[Theorem 5.30 and Remark 1.56]{adamek}, for the special case of $\Cat=\Bool, \SigmaAlg$ see, e.g., \cite[Theorems 11.2, 12.12]{koppelberg89}. 

For claim (v), from the universality of both coproducts we always have a canonical $\Bool$-map $f \colon \coprod^{\Bool}_{\alpha \in A} X_\alpha \to \coprod^{\SigmaAlg}_{\alpha \in A} X_\alpha$ for any $\SigmaAlg$-algebras $X_\alpha$.  The fact that this is in fact a canonical $\Bool$-inclusion (i.e., injective) was shown in \cite{lagrange}, we will also demonstrate it in Corollary \ref{bool-cat-extra} after we describe the categorical coproduct in $\SigmaAlg$ more explicitly; we will not need this fact until then.  To show that the two coproducts do not agree, let $\X$ be the Borel $\sigma$-algebra of $[0,1]$, then as is well-known $\X \otimes^\SigmaAlg \X$ can be identified with the Borel $\sigma$-algebra of $[0,1]^2$ (see, e.g., Proposition \ref{stomp}(iv)), while $\X \otimes^\Bool \X$ is instead the smaller $\Bool$-algebra of finite disjoint unions of rectangles $E \times F$ with $E,F \in \X$, and it is easy to see that the inclusion map from the latter to the former is the canonical map and fails to be a $\Bool$-isomorphism.
\end{proof}

The functors $\Mes \colon \AbsProb \to \ProbAlg$ and $\Inc \colon \ProbAlg \to \AbsProb$ do not quite generate an equivalence of categories, but they come close to it:

\begin{lemma}[Passing to the probability algebra]\label{ppa}\ 
\begin{itemize}
\item[(i)] $\Mes \circ \Inc$ is naturally isomorphic to $\ident_{\ProbAlg}$.
\item[(ii)]  There is a natural monomorphism $\iota$ from $\Inc \circ \Mes$ to $\ident_{\AbsProb}$.
\end{itemize}
\end{lemma}

\begin{proof}
If $X = (\X, \mu_X)$ is a $\ProbAlg$-space then the $\AbsProb$-space $\Inc(X) = (\X, \mu_X)$ has no non-trivial null sets, so $\Mes(\Inc(X))$ can be identified with $X$ again. It is then a routine matter to verify that this yields the natural monomorphism required for (i).

If $X = (\X, \mu_X)$ is an $\AbsProb$-space, then $$\Inc(\Mes(X))= (\X/_{\mathcal{N}}, \mu_X)$$ where ${\mathcal N}$ is the null ideal of $\X$.  The quotient map from $\X$ to $\X/_{\mathcal{N}}$ is a $\SigmaAlg$-morphism which is surjective, hence by Lemma \ref{bool-props}(ii) it induces an $\AbsMes$-monomorphism from $\X/_{\mathcal{N}}$ to $\X$.  This $\AbsMes$-monomorphism is clearly measure-preserving and can thus be promoted to an $\AbsProb$-monomorphism from $\Inc(\Mes(X))$ to $X$ by Lemma \ref{epimorph}.  It is then a routine matter to verify that this yields the natural monomorphism required for (ii).
\end{proof}

The relationship between the $\ConcMes$-categorical product to the $\AbsMes$- categorical product is subtle: they are not completely compatible with respect to abstraction functor $\Abs$, nevertheless there is a lot of partial compatibility in special cases.

\begin{proposition}[Relation between categorical products in $\ConcMes$ and $\AbsMes$]\label{stomp}\
Let $A$ be a set.
\begin{itemize}
    \item [(i)]  The categorical product $\prod_{A}^\ConcMes$ is contained in the categorical product $\prod_{A}^\ConcMes$ (with respect to the abstraction functor $\Abs$).
    \item[(ii)] The categorical product in $\ConcMes$ does not agree with the categorical product in $\AbsMes$ when $A = \{1,2\}$.
    \item[(iii)] If $\Cat = \Polish, \CH$, then the categorical product in $\Cat$ agrees with the categorical product in $\AbsMes$ for arbitrary $A$, with respect to the casting functor $\Cast_{\Cat \to \AbsMes}$ (which is either $\Abs \circ \BorelFunc$ or $\Abs \circ \BaireFunc$).
    \item[(iv)]  If $X$ is a $\ConcMes$-space and $K$ is a $\CH$-space, then $(X \times^\ConcMes K_\ConcMes)_\AbsMes$ is a categorical product of $X_\AbsMes$ and $K_\AbsMes$.
\end{itemize}
\end{proposition}

\begin{proof} 
The claim (ii) is a corollary of (i), since the $\sigma$-algebra of $\prod^\ConcMes_{\alpha \in A} X_\alpha$ is generated by the projection maps to $X_\alpha$.  For (iii), in \cite[Proposition A.1]{jt19} an example is given of $\AbsMes$-morphisms $y_1 \colon Z \to (Y_1)_\AbsMes$, $y_2 \colon Z \to (Y_2)_\AbsMes$ for some $\AbsMes$-space $Z$ and $\ConcMes$-spaces $Y_1,Y_2$ which do not jointly arise from an $\AbsMes$-morphism $y \colon Z \to (Y_1 \times^{\ConcMes} Y_2)_\AbsMes$.  In contrast, the universal nature of the categorical product in $\AbsMes$ implies that $y_1,y_2$ must jointly arise from an $\AbsMes$-morphism $(y_1,y_2)^{\AbsMes} \colon Z \to Y_1 \times^\AbsMes Y_2$.  The claim (iii) follows.

For Claim (iv), it suffices by \eqref{equiv-prod} to show that
$$ \Hom_\AbsMes\left(Y \to \left(\prod^\Cat_{\alpha \in A} X_\alpha\right)_\AbsMes\right) = \prod_{\alpha \in A} \Hom_\AbsMes\left(Y \to (X_\alpha)_\AbsMes\right)$$
for any $\Cat$-spaces $X_\alpha$ and $\AbsMes$-space $Y$.  For $\Cat=\Polish$ this follows from \cite[Proposition 3.3]{jt19} (and \cite[Remark 1.7]{jt19}; for $\Cat=\CH$ this similarly follows from \cite[Corollary 3.5]{jt19} (extended to arbitrary products as noted in that paper) and \cite[Remark 1.7]{jt19}. The claim (v) similarly follows from \cite[Proposition A.5]{jt19}.
\end{proof}

In Section \ref{loomis-sec} we will use the (functorial form of the) Loomis--Sikorski theorem to give a more explicit description of the categorical product in $\AbsMes$.

Proposition \ref{stomp}(iv) has the following consequence.  Let $\Cat = \Polish,\CH$, and let $K_1,K_2,K_3$ be $\Cat$-spaces.  Then any measurable binary operation $\cdot \colon K_1 \times K_2 \to K_3$ (that is to say, a $\ConcMes$-morphism\footnote{Here we use the casting conventions from Definition \ref{cast}, thus for instance $K_1 \times^\ConcMes K_2$ is shorthand for $(K_1)_\ConcMes \times^\ConcMes (K_2)_\ConcMes$.} from $K_1 \times^\ConcMes K_2$ to $K_3$ induces a ``conditional binary operation''
\begin{equation}\label{cond-op}
\cdot \colon \Hom_\AbsMes(Y \to K_1) \times \Hom_\AbsMes(Y \to K_2) \to \Hom_\AbsMes(Y \to K_3)
\end{equation}
for any $\AbsMes$-space $Y$, since Proposition \ref{stomp}(iv) ensures that the left-hand side is identifiable with $\Hom_\AbsMes(Y \to K_1 \times^\Cat K_2)$, and then one can compose with $(\cdot)_\AbsMes \colon (K_1 \times^\Cat K_2)_\AbsMes \to (K_3)_\AbsMes$ to obtain the desired conditional map.  Thus for instance for any $\AbsMes$-space $Y$ one can give $\Hom_\AbsMes(Y \to \R)$ the structure of a commutative partially ordered unital real algebra, and also $\Hom_\AbsMes(Y \to \C)$ the structure of a commutative unital *-algebra, by constructions of this form (as well as analogues for ternary operations, in order to establish properties such as associativity).  (These observations can also be placed in the more general framework of \emph{conditional analysis}, as developed in \cite{drapeau2016algebra}, particularly if $Y$ arises from a probability algebra; but we will not need this theory in the current paper.)

\section{The canonical model}\label{canonical-sec}

In this section we construct the canonical model via von Neumann and Riesz duality, as per Figure \ref{fig:canon}.  We first need to introduce a category of von Neumann algebras.

  \begin{figure}
    \centering
    \begin{tikzcd}
 \ConcProb \arrow[d, "\Abs"', blue] \arrow[r, "\Linfty"]  & \vonNeumannop \arrow[r, tail, two heads] \arrow[d,   tail, "\Idem", two heads, shift left=0.75ex]
& 
\CStarAlgUnitTraceop \arrow[d, tail,two heads,"\Riesz",shift left=.75ex]  \\
\AbsProb \arrow[r, "\Mes"', blue] \arrow[ur,"\Linfty"]
    & \ProbAlg \arrow[r,"\Stone"',  tail, two heads] \arrow[u,"\Linfty", tail, two heads, shift left=0.75ex] & \CHProb \arrow[u,"\CFunc", two heads,  tail, shift left=.75ex]
\end{tikzcd}
    \caption{Construction of the canonical model functor. The diagram commutes up to natural isomorphisms.}
    \label{fig:canon}
\end{figure}

\begin{definition}[Von Neumann algebra]\label{vonneumann-def}  A \emph{$\vonNeumann$-algebra} $({\mathcal A}, \tau_{\mathcal A})$ is a commutative von Neumann algebra ${\mathcal A}$ equipped with a faithful trace $\tau_{\mathcal A}$, that is to say a $*$-linear functional $\tau_{\mathcal A} \colon {\mathcal A} \to \C$ with $\tau_{\mathcal A}(1)=1$, and $\tau_{\mathcal A}(aa^*) \geq 0$ for any $a \in {\mathcal A}$, with equality if and only if $a=0$.  A \emph{$\vonNeumann$-morphism} $\Phi \colon ({\mathcal A}, \tau_{\mathcal A}) \to ({\mathcal B}, \tau_{\mathcal B})$ between $\vonNeumann$-algebras is a von Neumann algebra homomorphism $\tilde \Phi \colon {\mathcal A} \to {\mathcal B}$ such that $\tau_{\mathcal A} = \tau_{\mathcal B} \circ \tilde \Phi$.
\end{definition}

It is clear that $\vonNeumann$ forms a category.  Every von Neumann algebra is also a unital $C^*$-algebra, and a faithful trace on a commutative von Neumann algebra becomes a state on the associated $C^*$-algebra.  From this it is easy to see that there is a forgetful faithful functor from $\vonNeumann$ to $\CStarAlgUnitTrace$.

The most familiar construction of $\vonNeumann$-algebras comes from $L^\infty$ spaces.  Indeed, if $X = (X_\ConcMes,\mu_X)$ is a $\ConcProb$-space, then the Banach algebra $\Linfty(X)$ of equivalence classes $[f]$ of bounded (concretely) measurable functions $f \colon X \to \C$ up to almost everywhere equivalence, and endowed with the essential supremum norm $\|f\|_{\Linfty(X)}$ and the trace $\tau(f) \coloneqq \int_X f\ d\mu$, is well-known to be a $\vonNeumann$-algebra.  Furthermore, if $T \colon X \to Y$ is a $\ConcProb$-morphism, then the Koopman operator $\Linfty(T) \colon \Linfty(Y) \to \Linfty(X)$ defined by 
$$ L^\infty(T)([f]) \coloneqq [f \circ T]$$
for bounded concretely measurable $f \colon Y \to \C$, is a $\vonNeumann$-morphism.  Thus we see that $\Linfty \colon \ConcProb \to \vonNeumannop$ is in fact a functor.

We can factor this functor through the functor $\Mes \circ \Abs$ from the previous section, by defining an analogous $\Linfty$ functor on the category $\ProbAlg$ of probability algebras.  Indeed, if $X = (\Inc(X)_\SigmaAlg, \mu_X)$ is a $\ProbAlg$-space, we can define $\Linfty(X)$ to be the space of all $\AbsMes$-morphisms $f \colon \Inc(X) \to \C$ (i.e., $f \in \Hom_\AbsMes(\Inc(X) \to \C ) = \Hom_\AbsMes( \Inc(X)_\AbsMes \to \C_\AbsMes )$) which are bounded in the sense that $|f| \leq M$ for some $M \geq 0$ (or equivalently $f_\SigmaAlg( \{ z \in \C: |z| \leq M \} ) = 1$), with $\|f\|_{\Linfty(X)}$ defined to equal the infimum of all such $M$. (Note there is no need to identify functions that agree almost everywhere since the base space is already a probability algebra.) Using conditional operations such as \eqref{cond-op}, one can verify that $\Linfty(X)$ is a $\CStarAlgUnit$-algebra.   Every element $E$ of $\Inc(X)_\SigmaAlg$ generates an idempotent element $1_E$ of $\Linfty(X)$, defined by
setting $1_E^*(F)$ for $F \in \C_\SigmaAlg$ to equal $E$ when $F$ contains $1$ but not $0$, $\overline{E}$ when $F$ contains $0$ but not $1$, $1$ when $F$ contains both $0$ and $1$, and $0$ when $F$ contains neither $0$ and $1$.  We refer to finite linear combinations of idempotents as \emph{simple functions}, it is easy to see that these form a dense subspace of $\Linfty(X)$.  One can then define a state $\tau$ on this algebra by defining
$$ \tau \left( \sum_{n=1}^N c_n 1_{E_n} \right) \coloneqq \sum_{n=1}^N c_n \mu_X(E_n)$$
for any finite sequence of complex numbers $c_n$ and $E_n \in \Inc(X)_\SigmaAlg$, and then extending by density; one can verify that this indeed defines a state (this is essentially an abstraction of the standard construction of the Lebesgue integral that proceeds first by integrating simple functions).  Thus $\Linfty(X)$ can be viewed as an element of $\CStarAlgUnitTrace$.  One can also construct an abstract $L^1(X)$-space (see Remark \ref{rem-lpspaces}) and show that $L^1(X)$ is the predual of $L^\infty(X)$ (see the duality between abstract $L^\infty$ and $L^1$ spaces in \cite[\S 365]{fremlinvol3}). It then follows from  Sakai's characterization of von Neumann algebras \cite{sakai} that $\Linfty(X)$ is indeed a $\vonNeumann$-algebra. 

To emphasize the analogy between $\ConcProb$-spaces and $\ProbAlg$-spaces, we also write
$$ \int_X f \coloneqq \int_X f\ d\mu_X \coloneqq \tau(f)$$
for any $\ProbAlg$-space $X = (\Inc(X)_\SigmaAlg, \mu_X)$ and $f \in \Linfty(X)$.  If $T \colon X \to Y$ is a $\ProbAlg$-morphism, one can define the $\vonNeumann$-morphism $\Linfty(T) \colon \Linfty(Y) \to \Linfty(X)$ by the Koopman operator
$$ \Linfty(T)(f) \coloneqq f \circ \Inc(T)_\AbsMes$$ which 
can be verified to indeed be a $\vonNeumann$-morphism (this is an abstraction of the change of variables formula for the Lebesgue integral).  Some tedious but routine verification (see, e.g., \cite{pavlov2020}) then shows that $\Linfty \colon \ProbAlg \to \vonNeumannop$ is a functor with
$$ \Linfty = \Linfty \circ \Mes \circ \Abs.$$
By abuse of notation we can also write $\Linfty \colon \AbsProb \to \vonNeumannop$ for the composition $\Linfty = \Linfty \circ \Mes$, giving rise to the commutativity of the left half of Figure \ref{fig:canon} (omitting the functor $\Idem$).

Now suppose that $({\mathcal A},\tau_{\mathcal A})$ is a $\vonNeumann$-algebra.  We can form the collection ${\mathcal P}_{\mathcal A}$ of real projections in ${\mathcal A}$, that is to say elements $p \in {\mathcal A}$ such that $p=p^* = p^2$.  As is well-known, these projections have the structure of a $\SigmaAlg$-algebra, with $p \wedge q = pq$, $\overline{p} = 1-p$, $p \vee q = 1 - (1-p)(1-q)$, and with $\bigvee_{n \in \N} p_n = \sum_{n \in \N} p_n$ (in the $L^2$ topology) if the $p_n$ are disjoint.  The trace $\tau_{\mathcal A}$ then becomes a countably additive probability measure on ${\mathcal P}_{\mathcal A}$, and we write $\Idem({\mathcal A},\tau_{\mathcal A})$ for the probability algebra
$$\Idem({\mathcal A},\tau_{\mathcal A}) \coloneqq ( {\mathcal P}_{\mathcal A}, \tau_{\mathcal A}).$$
If $\Phi \colon ({\mathcal A},\tau_{\mathcal A}) \to ({\mathcal B},\tau_{\mathcal B})$ is a $\vonNeumann$-morphism, we observe that the associated von Neumann homomorphism $\tilde \Phi \colon {\mathcal A} \to {\mathcal B}$ maps projections in ${\mathcal P}_{\mathcal A}$ to projections in ${\mathcal P}_{\mathcal B}$, in a manner that preserves the trace as well as being a $\SigmaAlg$-morphism.  We then define $\Idem(\Phi) \colon \Idem({\mathcal B},\tau_{\mathcal B}) \to \Idem({\mathcal A},\tau_{\mathcal A})$ to be the $\ProbAlg$-morphism associated with this $\SigmaAlg$-morphism.  It is then a routine matter to verify that $\Idem$ is a functor from $\vonNeumannop$ to $\ProbAlg$. 

We claim that $\Idem$ and $\Linfty$ form a duality of categories between $\vonNeumann$ and $\ProbAlg$.  First suppose that $X = (\Inc(X)_\SigmaAlg, \mu_X)$ is a $\ProbAlg$-space.  For every $E \in \Inc(X)_\SigmaAlg$, it is easy to see that the indicator function $1_E$ is a projection in $\Linfty(X)$.  Conversely we claim that all projections in $\Linfty(X)$ are of this form.  If $p \in \Linfty(X)$, then since $p-p^2=0$, $p$ is an $\AbsMes$-morphism from $X$ to $\C$ that becomes the zero morphism after concatenation with the map $z \mapsto z-z^2$, viewed as an $\AbsMes$-endomorphism on $\C$.  Pulling back, we conclude that $p_\SigmaAlg(\{0,1\})=1$, and hence $p = 1_E$ where $E \coloneqq p_\SigmaAlg(\{1\})$.  Using this correspondence $E \mapsto 1_E$ it is a routine matter to see that $X$ is $\ProbAlg$-isomorphic to $\Idem(\Linfty(X))$, and further routine verification shows that this isomorphism is natural.  Now let $({\mathcal A}, \tau_{\mathcal A})$ be a $\vonNeumann$-algebra.  By definition, we see that the von Neumann algebra $\Linfty(\Idem({\mathcal A}, \tau_{\mathcal A}))$ is the closure (in $\Linfty$) of formal linear combinations of projections, which one can arrange to be pairwise disjoint.  One can observe (by repeated use of the identity\footnote{The inequality $\max( \| ap \|_{\mathcal A}, \|a(1-p)\|_{\mathcal A} )\leq \|a\|_{\mathcal A}$ is immediate from $\|ab\|\leq \|a\|\|b\|$ for all $a,b\in \mathcal{A}$ and $\|p\|=1$ for all projections $p$. On the other hand, for any projection $p$ and integer $n\geq 1$, we have $\|a^n\|^{1/n}=\|(ap)^n + (a(1-p))^n\|^{1/n}\leq (\|ap\|^n + \|a(1-p)\|^n)^{1/n}$. The converse inequality now follows from applying Gelfand's spectral radius formula $\|b\|=\lim \|b^n\|^{1/n}$ (note that every $b\in \mathcal{A}$ is normal in a commutative von Neumann algebra) to both sides of the previous inequality.}  $\|a\|_{\mathcal A} = \max( \| ap \|_{\mathcal A}, \|a(1-p)\|_{\mathcal A} )$ in a commutative von Neumann algebra for arbitrary $a \in {\mathcal A}$ and projections $p$) 
that the corresponding actual linear combination of these projections in ${\mathcal A}$ has the same norm in ${\mathcal A}$ as the $\Linfty$ norm of the formal linear combination; the two expressions also have the same trace.  Also, from the spectral theorem one can show that any element in ${\mathcal A}$ can be approximated in norm to arbitrary accuracy by finite linear combinations of projections.  From these facts one can show that $\Linfty(\Idem({\mathcal A}, \tau_{\mathcal A}))$ is $\vonNeumann$-isomorphic to $\Idem({\mathcal A}, \tau_{\mathcal A})$, and further routine verification shows that this isomorphism is natural.  This gives the required duality of categories.  In particular $\Idem, \Linfty$ are full and faithful. We refer the interested to the independent work of Pavlov \cite{pavlov2020}, where (see \cite[\S 3.5]{pavlov2020}) this duality (and in fact more general dualities for commutative not necessarily tracial von Neumann algebras) are discussed in depth.  

If we now define
$$ \Stone \coloneqq \Spec \circ \Forget_{\vonNeumannop \to \CStarAlgUnitop} \circ \Linfty$$
then by construction $\Stone$ is a   functor from $\ProbAlg$ to $\CHProb$.  All the three functors used to create $\Stone$ are faithful and full functors, so $\Stone$ is a faithful and full functor as well.  It is now a routine matter to establish

\begin{theorem}[Construction of canonical model]\label{canon-thm}  The categories in Figures \ref{fig:canon} are indeed categories, and the functors in these figures are indeed functors between the indicated categories, with the indicated faithfulness and fullness properties.  Furthermore, the diagram commutes up to natural isomorphisms.
\end{theorem}

Now we establish some basic properties of the canonical model.  Lusin's theorem asserts that $\CFunc(X)$ (after identifying functions that agree almost everywhere) becomes a dense subspace of $\Linfty(X)$ in the $L^2$ topology.  We now consider the following stronger property:

\begin{definition}[Strong Lusin property]  A $\CHProb$-space $X$ has the \emph{strong Lusin property} if every equivalence class $[f]$ in $\Linfty(X) = \Linfty(X_\ConcProb)$ contains precisely one element of $\CFunc(X)$, thus one has an identification $\Linfty(X)
\equiv \CFunc(X)$.
\end{definition}

Most $\CHProb$-spaces will not have this property, but remarkably the canonical models do:

\begin{proposition}[Basic properties of canonical model]\label{model-basic}\ 
\begin{itemize}
\item[(i)] ($\Stone$ is a model) The functor $\Cast_{\CHProb \to \ProbAlg} \circ \Stone \colon \ProbAlg \to \ProbAlg$ is naturally isomorphic to the identity.  In particular, $\Stone$ is full and faithful. 
\item[(ii)]  (Strong Lusin property) For any $\ProbAlg$-space $X$, $\Stone(X)$ has the strong Lusin property.
\end{itemize}
\end{proposition}

\begin{proof}  Let $X$ be a $\ProbAlg$-space.  Then from Figure \ref{fig:canon} we have a natural $\CStarAlgUnitTrace$-isomorphism
$$ \Forget_{\vonNeumannop \to \CStarAlgUnitTraceop}( \Linfty(X) ) \equiv \CFunc( \Stone(X) ).$$
As $\Linfty(X)$ is a tracial von Neumann algebra, it comes with an $L^2$ metric by the Gelfand--Naimark--Segal construction, which by the above isomorphism agrees with the $L^2$ metric on $\Stone(X)$.  By Lusin's theorem, the closure of the closed unit ball of $\CFunc(\Stone(X))$ in the $L^2(\Stone(X))$ topology is the closed unit ball of $\Linfty(\Stone(X))$ (here we apply a forgetful functor to view $\Stone(X)$ as a $\ConcProb$-space).  Also, in the tracial von Neumann algebra $\Linfty(\Stone(X))$, the closed unit ball is also closed in $L^2$.  We conclude that
$$ \CFunc(\Stone(X)) = \Linfty(\Stone(X))$$
which is the strong Lusin property.  This implies the natural $\vonNeumann$-isomorphism
$$ \Linfty(X) \equiv \Linfty( \Stone(X) )$$
which on applying $\Idem$ gives (i).
\end{proof}

Note how Theorem \ref{uncountable-model} is immediate from Theorem \ref{canon-thm} and Proposition \ref{model-basic}(i).

\begin{remark} A measure space with the strong Lusin property is referred to as a ``perfect measure space'' in \cite{segal}, in which a version of the canonical model functor construction just presented is given.  An alternate proof of Proposition \ref{model-basic}(ii) using Banach lattice arguments is given in \cite[Proposition 12.25]{EFHN}.
\end{remark}

Analogously to how the Stone--{\v C}ech compactification $\beta X$ can be viewed as a universal compactification of an $\LCH$-space $X$, one can view $\Stone(X)$ as an ``universal concrete model'' of a $\ProbAlg$-space $X$.  To formalize this claim, we define a \emph{concrete model} for a $\ProbAlg$-space $X$ to be  $(\tilde X, \iota)$, where $\tilde X$ is a $\CHProb$-space and the $\AbsProb$-morphism $\iota \colon \Inc(X) \to \tilde{X}_\AbsProb$ is an $\AbsMes$-monomorphism.  We let $\Model(X)$ be the category of all such models, with a $\Model(X)$-morphism $T \colon (\tilde X, \iota) \to (\tilde X', \iota')$ to be a $\CHProb$-morphism $T_\CHProb \colon \tilde X \to \tilde X'$ such that $\iota' =_\AbsProb T\circ \iota$ (i.e., $\iota' = \iota \circ T_\AbsProb$). (Note that this gives $\Model(X)$ the structure of a partially ordered set.)
 
 \begin{figure} 
    \centering
    \begin{tikzcd}
    \Stone(X) \arrow[r,"\Stone(\pi)"] & \Stone(Y) \\
    \Inc(X) \arrow[u, "\iota'", tail] \arrow[r, "\Inc(\pi)"] & \Inc(Y) \arrow[u, "\iota'", tail]
    \end{tikzcd}
    \caption{Every $\ProbAlg$-morphism $\pi: X \to Y$ gives rise to an $\AbsProb$-morphism $\Inc(\pi) \colon \Inc(X) \to \Inc(Y)$ and a $\CHProb$-morphism $\Stone(\pi) \colon \Stone(X) \to \Stone(Y)$, linked by the above commutative diagram in $\AbsProb$, with $\iota$ the canonical inclusions. Casting functors have been suppressed to reduce clutter.}
    \label{fig:canon-diag}
\end{figure} 

 \begin{proposition}[Universality of the canonical model]\label{canon-model}  Let $X$ be a $\ProbAlg$-space.
 \begin{itemize}
 \item[(i)]  If $(\tilde X, \iota)$ is a concrete model of $X$, then $\tilde X_\ProbAlg$ is $\ProbAlg$-isomorphic to $X$, and $\iota$ is the composition of $\Inc$ applied to that isomorphism, with the natural monomorphism of $\Inc(\tilde X)$ to $X_\AbsProb$.  Conversely, if $\tilde X$ is $\ProbAlg$-isomorphic to $X$, then the pair $(\tilde X, \iota)$ is a concrete model of $X$, where $\iota$ is defined as above.
 \item[(ii)]  $(\Stone(X), \iota_{\Stone(X)})$ is a concrete model of $X$, where $\iota_{\Stone(X)} \colon \Inc(X) \to \Stone(X)_\AbsProb$ is the canonical inclusion formed by applying first $\Inc$ to the natural isomorphism from $X$ to $\Stone(X)_\ProbAlg$ from Proposition \ref{model-basic} and then composing it with the natural monomorphism from $\Inc(\Stone(X)_\ProbAlg)$ to $\Stone(X)_\AbsProb$ from Lemma \ref{ppa}(ii).  (The naturality of this model is then depicted in Figure \ref{fig:canon-diag}.)
 \item[(iii)]  A concrete model $(\tilde X, \iota)$ is an initial object (as defined in Definition \ref{special-morph}) of $\mathbf{Model}(X)$ if and only if $\tilde X$ has the strong Lusin property.  In particular, by Proposition \ref{model-basic}(ii), the concrete model in (ii) is initial in $\Model(X)$.
 \end{itemize}
 \end{proposition}
 
 \begin{proof}  We begin with (i).  If $(\tilde X, \iota)$ is a concrete model, then by duality the morphism
 $$\Forget_{\AbsProb \to \AbsMes}(\iota) \equiv \Forget_{\ProbAlg \to \SigmaAlgop} \circ 
 \Mes(\iota)$$
is a $\SigmaAlg$-epimorphism; from Lemma \ref{bool-props}(iii) it is also a $\SigmaAlg$-monomorphism, hence a $\SigmaAlg$-isomorphism by Lemma \ref{bool-props}(i).  Thus $\iota$ is now invertible in $\AbsMes$ and also measure-preserving (i.e., an $\AbsProb$-morphism), hence it is also invertible in $\AbsProb$.  Applying $\Mes$ we conclude that $\tilde X_\ProbAlg$ is $\ProbAlg$-isomorphic to $X$, with $\iota$ related to this isomorphism as indicated.  The converse implication is routine.

Claim (ii) follows from (i) and Proposition \ref{model-basic}(i), so we turn to (iii).
First suppose that $(\tilde X, \iota)$ obeys the strong Lusin property.  We need to show that for any concrete model $(\tilde X', \iota')$ of $X$ there is precisely one $\CHProb$-morphism $T \colon \tilde X \to \tilde X'$ with $\iota' =_\AbsProb T \circ \iota$.  To show existence, we start with the obvious $\CStarAlgUnit$-map
$$ \Phi \colon \CFunc(\tilde X') \to \Linfty(\tilde X') \equiv \Linfty(X) \equiv \CFunc( \tilde X )$$
and apply $\Spec$ and natural isomorphisms to obtain a $\CH$-morphism
$$ T \colon \tilde X \to \tilde X'$$
with the property that $f \circ T$ and $f$ agree in $\Linfty(X)$ for every $f \in \CFunc(\tilde X')$.   In particular $\int_{\tilde X} f \circ T = \int_{\tilde X'} f$. By Theorem \ref{rrt} this implies that $T$ can be promoted to a $\CHProb$-morphism.  From Lusin's theorem we see that $\CFunc(\tilde X')$ is dense in $\Linfty(\tilde X') \equiv \Linfty(X)$ using the $L^2(X)$ topology, and using this one can show that $1_E \circ T$ and $1_E$ agree in $\Linfty(X)$ for any $E \in \tilde X'_\SigmaAlg$, thus $\iota' =_\AbsProb T \circ \iota$ as desired.  This establishes existence.  To show uniqueness, we see that if $T' \colon \tilde X \to \tilde X'$ is any other $\CHProb$-morphism with $\iota' =_\AbsProb T' \circ \iota$, then for $f \in \CFunc(\tilde X')$, $f \circ T, f \circ T' \in \CFunc(\tilde X)$ agree in $\Linfty(\tilde X)$ and are thus equal as continuous functions.  Thus, for any $\tilde x \in \tilde X$, we have $f(T(\tilde x)) = \CFunc(T'(\tilde x))$ for all $f \in \CFunc(\tilde X')$.  Since the functions in $\CFunc(\tilde X')$ separate points, we obtain $T=T'$, giving uniqueness.

Conversely, if $(\tilde X, \iota)$ is initial in $\Model(X)$, then by the preceding discussion it is $\Model(X)$-isomorphic to $(\Stone(X),\iota(X))$, and then it is straightforward to derive the strong Lusin property of $\tilde X$ from that of $\Stone(X)$.
 \end{proof}

\begin{remark} With a bit more effort, one can show that every concrete model $(\tilde X,\iota)$ of a $\ProbAlg$-space $X$ comes with a canonical identification of $\CFunc(\tilde X)$ as a $\CStarAlg$-subalgebra of $\Linfty(X)$ that is dense in the $L^2$ topology, and conversely any such dense subalgebra gives rise to a concrete model, unique up to $\Model(X)$-isomorphism (cf.~\cite[\S 12.3]{EFHN}); this duality of subalgebras and models is analogous for instance to the fundamental theorem of Galois theory.  The morphisms in $\Model(X)$ then are canonically identified with inclusions maps in $\Linfty(X)$.  When $\tilde X$ has the strong Lusin property, $\CFunc(\tilde X)$ is identified with all of $\Linfty(X)$, which explains the universality.  We leave the verifications of these claims to the interested reader.  The situation can again be compared with the Stone--{\v C}ech compactification, in which the role of the functor $\Linfty$ is instead played by $\CbFunc$.
\end{remark}

\begin{remark}\label{prob-prod} As one quick application of the canonical model functor $\Stone$ one can construct infinite tensor products $\bigotimes^{\ProbAlg}$ on $\ProbAlg$ (in the sense of Definition \ref{def-infinit-prod}) on arbitrarily many factors by starting with infinite tensor products $\bigotimes^{\CHProb}$ on $\CHProb$ (by observing, similarly to Example \ref{concprod-ex}, that we can realize $\CHProb$ as an action category which then comes with a canonical tensor product making $\CHProb$ a semicartesian symmetric monoidal category with arbitrarily infinite tensor products)\footnote{Notice that the $\ConcProb$ and $\CH$ tensor products both agree with the $\ConcMes$ categorical products basically by construction.}, and then defining the tensor product $\bigotimes^\ProbAlg_{\alpha \in A} X_{\alpha}$ of $\ProbAlg$-spaces $X_\alpha$ to be $\bigotimes^{\CHProb}_{\alpha \in A} \Stone(X_\alpha)$ casted back to $\ProbAlg$.  (We caution however that $\bigotimes^{\CHProb}_{\alpha \in A} \Stone(X_\alpha)$ need not obey the strong Lusin property, and thus need not to be equal to $\Stone(\bigotimes^\ProbAlg_{\alpha \in A} X_\alpha)$.)  An alternative way to construct the tensor product is to use probability duality and the standard tensor product operation on von Neumann algebras (which gives $\vonNeumann$ the structure of a cosemicartesian symmetric monoidal category).  We leave it to the reader to verify that these two tensor products are equal up to natural isomorphisms.
\end{remark}

As is well known, every continuous function from an $\LCH$-space $X$ to a $\CH$-space $K$ has a unique continuous extension to the Stone--{\v C}ech compactification $\beta X$, giving an equivalence
$$ \Hom_\LCH(X \to K) \equiv \Hom_\CH(\beta X \to K).$$
In category-theoretic language, $\beta$ is left-adjoint to the forgetful functor from $\CH$ to $\LCH$. There is an analogous property for the canonical model:

\begin{proposition}[Canonical representation]\label{factor-of-stone}  If $X$ is a  $\ProbAlg$-space and $K$ is a $\CH$-space, then to every $\AbsMes$-morphism $f \colon \Inc(X) \to K$ there is a unique $\CH$-morphism $\tilde f \colon \Stone(X) \to K$ which extends (or represents) $f$ in the sense that $f =_\AbsMes \tilde f \circ \iota'$, where $\iota' \colon \Inc(X) \to \Stone(X)$ is the canonical $\AbsProb$-morphism.  In other words, one has an equivalence
$$ \Hom_\AbsMes(\Inc(X) \to K) \equiv \Hom_\CH(\Stone(X) \to K).$$
\end{proposition}

\begin{proof}  We first prove existence.  The $\AbsMes$-morphism $f$ induces a pullback map $f^* \colon C(K) \to \Linfty(X)$, since for any $g \in C(K)$, $g \circ f$ is an $\AbsMes$-morphism from $\Inc(X)$ to a bounded subset of $\C$ and can thus be identified with an element of $\Linfty(X)$.  By construction, $\Linfty(X) \equiv C(\Stone(X))$.  Thus we may apply the functor $\Spec$ to obtain a $\CH$-morphism $\Spec(f^*) \colon \Stone(X) \to K$ (after performing some natural identifications), and the required property $f =_\AbsMes \tilde f \circ \iota'$ can be verified by chasing all the definitions.

To prove uniqueness, suppose we have two $\CH$-morphisms $\tilde f, \tilde f' \colon \Stone(X) \to K$ with $\tilde f \circ \iota' = \tilde f' \circ \iota'$.  Then for any $g \in C(K)$, $g \circ \tilde f$ and $g \circ \tilde f'$ agree in $L^\infty(\Stone(X))$, hence agree in $C(\Stone(X))$ by the strong Lusin property.  Since $C(K)$ separates points, we conclude that $\tilde f = \tilde f'$, giving uniqueness.
\end{proof}

As a corollary of this proposition, we see that $\Stone$ is left-adjoint\footnote{Indeed, since (left-)adjoints are unique up to natural isomorphisms, one could take this fact as a \emph{definition} of the canonical model functor $\Stone$ if desired, although then to verify the remaining properties of the model claimed in this section seems to require an equivalent amount of work to that in the approach presented here.} to the casting functor $\Cast_{\CHProb \to \ProbAlg}$.  As another corollary, if $K$ is a $\CHProb$-space, then by applying the above equivalence to the canonical inclusion $\iota \colon \Inc(K) \to K$ we obtain a $\CHProb$-morphism $\pi \colon \Stone(K) \to K$, which one can check to be a natural transformation from $\Stone \circ \Cast_{\CHProb \to \ProbAlg}$ to $\ident_{\CHProb}$.  Thus one can view any $\CHProb$-space $K$ as a ``factor'' of its canonical model $\Stone(K)$, and one can view the $\AbsProb$-space $\Inc(K)$ as an abstract full measure subspace of both of these $\CHProb$-spaces in which all the null sets have been ``deleted''.

We close this section with a surjectivity property of the morphisms generated the canonical model functor (cf. Lemma \ref{bool-props}(iii)).

\begin{proposition}  If $T\colon X \to Y$ is a  $\ProbAlg$-morphism, then $\Stone(T) \colon \Stone(X) \to \Stone(Y)$ is surjective.
\end{proposition}

\begin{proof} Suppose for contradiction that $\Stone(T)$ is not surjective.  Then from Urysohn's lemma one can find non-zero $g \in \CFunc(\Stone(Y))$ such that $g \circ \Stone(T) = 0$.  By the strong Lusin property, $g$ is non-zero in $L^\infty(\Stone(Y))$, thus by taking sublevel sets there is a positive measure subset of $\Stone(Y)$ whose pullback by $\Stone(T)$ is a null set in $\Stone(X)$.  But this contradicts the measure-preserving nature of $\Stone(T)$.
\end{proof}

\section{Canonical disintegration}\label{disint-sec}

In this section we prove Theorem \ref{canon-disint}.  We begin with existence. Let $X, Y$ be $\ProbAlg$-spaces, and let $\pi \colon X \to Y$ be a $\ProbAlg$-morphism.  Then $\Stone(\pi)$ is a $\CHProb$-morphism from $\Stone(X)$ to $\Stone(Y)$, which gives rise to a Koopman operator $\pi^* \colon L^2(\Stone(Y)) \to L^2(\Stone(X))$ defined in the obvious fashion.  This operator is an $L^2$ isometry, so we can identify
$L^2(\Stone(Y))$ with a closed subspace of $L^2(\Stone(X))$, and similarly identify $\Linfty(\Stone(Y))$ with a subspace of $\Linfty(\Stone(X))$.  We let $f \mapsto \E(f|\Stone(Y))$ be the orthogonal projection from $L^2(\Stone(X))$ to $L^2(\Stone(Y))$.  From construction we see that
\begin{align}
\label{fgm}
 \int_{\Stone(X)} f g\ d\mu_{\Stone(X)} &= \int_{\Stone(X)} \E(f|\Stone(Y)) g\ d\mu_{\Stone(X)} \\ 
 \notag &= \int_{\Stone(Y)} \E(f|\Stone(Y)) g\ d\mu_{\Stone(Y)}
\end{align}
for all $f \in \Linfty(\Stone(X))$ an $g \in \Linfty(\Stone(Y))$ (making heavy use of the above identifications).  By duality and H\"older's inequality we conclude the contractive property
$$ \| \E(f|\Stone(Y))\|_{\Linfty(\Stone(Y))} \leq \| f\|_{\Linfty(\Stone(X))}$$
so in particular $\E(f|\Stone(Y))$ is an element of $\Linfty(\Stone(Y))$.  By Proposition \ref{model-basic}(ii), we can identify $\Linfty(\Stone(Y))$ with $\CFunc(\Stone(Y))$ (and $\Linfty(\Stone(X))$ with $\CFunc(\Stone(X))$), so by abuse of notation we also view $\E(f|\Stone(Y))$ as an element of $\CFunc(\Stone(Y))$ for any $f \in \CFunc(\Stone(X))$.  In particular, for any $y \in \Stone(Y)$, we have a functional $f \mapsto \E(f|\Stone(Y))(y)$ on $\Stone(X)_\CH$, which one can easily verify to be a state.  Applying Theorem \ref{rrt}, one can represent this functional by a Radon probability measure $\mu_y$ on $\Stone(X)_\CH$, thus
$$ \E(f|\Stone(Y))(y) = \int_{\Stone(X)_\CH} f\ d\mu_y$$
for all $f \in \CFunc(\Stone(X))$ and $y \in \Stone(Y)$.  In particular $y \mapsto \int_{\Stone(X)_\CH} f\ d\mu_y$ is continuous and from \eqref{fgm} we conclude \eqref{disint-form}.  This establishes existence.  For uniqueness, let $\mu'_y, y \in \Stone(Y)$ be another candidate disintegration.  Then for any $f \in \CFunc(\Stone(X))$, we see from \eqref{disint-form} that the continuous function $y \mapsto \int_{\Stone(X)_\CH} f\ d\mu_y - \int_{\Stone(X)_\CH} f\ d\mu'_y$ is orthogonal (in $L^2(\Stone(Y))$) to all elements of $\CFunc(\Stone(Y))$, and hence is identically zero (here we view $\CFunc(\Stone(Y)) \equiv \Linfty(\Stone(Y))$ as a subspace of $L^2(\Stone(Y))$).  Thus for every $y \in \Stone(Y)$, we have
$$ \int_{\Stone(X)_\CH} f\ d\mu_y = \int_{\Stone(X)_\CH} f\ d\mu'_y$$
for all $f \in \CFunc(\Stone(X))$.  Applying Theorem \ref{rrt}, we conclude that $\mu_y = \mu'_y$, giving uniqueness.

Finally we need to show that $\mu_y(E)=0$ when $E$ is measurable and disjoint from $\Stone(\pi)^{-1}(\{y\})$.  By inner regularity we may assume that $E$ is compact $G_\delta$.  Then $\Stone(\pi)(E)$ is compact and disjoint from $y$, hence by Proposition \ref{urysohn} one can find $\chi \in \CFunc(\Stone(Y))$ such that $\chi(y')=1$ for $y' \in \Stone(\pi)(E)$ and $\chi(y)=0$.  We also view $\chi$ as an element of $\CFunc(\Stone(X))$, then $\E(\chi|\Stone(Y))=\chi$, in particular
$$ \int_{\Stone(X)_\CH} \chi\ d\mu_y = \chi(y)=0$$
and hence $\mu_y(E)=0$ as required.  This concludes the proof of Theorem \ref{canon-disint}.

By following the construction in \cite[Section 5.5]{furstenberg2014recurrence}, one can use the canonical disintegration to build relative products of probability algebras, but now without the need to impose any regularity hypotheses on the algebras.

\begin{theorem}[Relative products in $\ProbAlg$]\label{rel-prod}  Suppose that one has $\ProbAlg$-morphisms $\pi_1 \colon X_1 \to Y$, $\pi_2 \colon X_2 \to Y$.  Then there exists a $\ProbAlg$-commutative diagram
\begin{center}
\begin{tikzcd}
    & X_1 \otimes_Y X_2 \arrow[dl, "\Pi_1"'] \arrow[dr,"\Pi_2"] \\ X_1 \arrow[dr,"\pi_1"'] && X_2 \arrow[dl,"\pi_2"] \\
    & Y
\end{tikzcd}
\end{center}
for some $\ProbAlg$-space $X_1 \otimes_Y X_2$ and $\ProbAlg$-morphisms $\Pi_1 \colon X_1 \otimes_Y X_2 \to X_1$, $\Pi_2 \colon X_1 \otimes_Y X_2 \to X_2$, which of course also leads to the $\vonNeumann$-commutative diagram
\begin{center}
\begin{tikzcd}
    & \Linfty(X_1 \otimes_Y X_2) \arrow[dl, "\Linfty(\Pi_1)"'] \arrow[dr,"\Linfty(\Pi_2)"] \\ \Linfty(X_1) \arrow[dr,"\Linfty(\pi_1)"'] && \Linfty(X_2) \arrow[dl,"\Linfty(\pi_2)"] \\
    & \Linfty(Y)
\end{tikzcd},
\end{center}
such that one has
\begin{equation}\label{f1f2}
 \int_{X_1 \otimes_Y X_2} f_1 f_2 = \int_Y \E(f_1|Y) \E(f_2|Y) 
\end{equation}
for all $f_1 \in \Linfty(X_1)$, $f_2 \in \Linfty(X_2)$, where we use the above commutative diagram to embed $\Linfty(Y)$ into $\Linfty(X_1), \Linfty(X_2)$, and embed these algebras in turn into $\Linfty(X_1 \otimes_Y X_2)$. Furthermore, $\Inc(X_1 \otimes_Y X_2)_\SigmaAlg$ is generated by $\Inc(X_1)_\SigmaAlg$ and $\Inc(X_2)_\SigmaAlg$ (where we identify the latter with subalgebras of the former in the obvious fashion).
\end{theorem}

\begin{proof}
From the canonical disintegration we have probability measures $\mu_{y,i}$ on $\Stone(X_i)_\CH$ for $y \in \Stone(Y)$ and $i=1,2$ that depend continuously on $y$ in the vague topology, and such that
$$ \E(f_i|\Stone(Y))(y) = \int_{\Stone(X_i)_\CH} f_i\ d\mu_{y,i}$$
for $f_i \in \CFunc(\Stone(X_i))$ and $y \in \Stone(Y)$.  We then define a probability measure $\mu$ on $\Stone(X_1)_\CH \times^\CH \Stone(X_2)_\CH$ by the formula
\begin{align*}
&\int_{\Stone(X_1)_\CH \times^\CH \Stone(X_2)_\CH} f(x_1,x_2)\ d\mu(x_1,x_2)\\
&\quad \coloneqq \int_{\Stone(Y)} \left(\int_{\Stone(X_1)_\CH} \int_{\Stone(X_2)_\CH} f(x_1,x_2)\ d\mu_{y,2}(x_2) d\mu_{y,1}(x_1)\right)\ d\mu_{\Stone(Y)}(y).
\end{align*}
Note from continuity in the vague topology (using Stone-Weierstrass to approximate $f$ uniformly by linear combinations of tensor products $f_1(x_1) f_2(x_2)$ of continuous functions $f_1,f_2$ if desired) that the expression in parentheses is a bounded continuous function on $y$.  The well-definedness of $\mu$ follows from the Riesz representation theorem (Theorem \ref{rrt}).  From construction we have
$$ \int_{\Stone(X_1)_\CH \times^\CH \Stone(X_2)_\CH} f_1(x_1) f_2(x_2)\ d\mu(x_1,x_2) = \int_Y \E(f_1|Y) \E(f_2|Y)$$
for any $f_1 \in \Linfty(X_1)$, $f_2 \in \Linfty(X_2)$, where we identify $\Linfty(X_i)$ with $\CFunc(\Stone(X_i))$.  If we then define 
$$X_1 \otimes_Y X_2 \coloneqq (\Stone(X_1)_\CH \times^\CH \Stone(X_2)_\CH, \mu)_\ProbAlg$$
then we obtain the identity \eqref{f1f2}.  By Stone-Weierstrass, the finite linear combinations of products $f_1 f_2$ with $f_1 \in \Linfty(X_1), f_2 \in \Linfty(X_2)$ are dense in $L^\infty(X_1 \otimes_Y X_2)$ in the $L^2$ topology, hence any element of $\Inc(X_1 \otimes_Y X_2)_\SigmaAlg$ can be approximated to arbitrarily small error by a finite boolean combination of elements of $\Inc(X_1)_\SigmaAlg, \Inc(X_2)_\SigmaAlg$.  Since $X_1 \otimes_Y X_2$ is a probability algebra, every element in $\Inc(X_1 \otimes_Y X_2)_\SigmaAlg$ then lies in the $\SigmaAlg$-algebra generated by $\Inc(X_1)_\SigmaAlg, \Inc(X_2)_\SigmaAlg$.  The claim follows.
\end{proof}

One can show that the relative (tensor) product $\otimes_Y$ gives the slice category $\ProbAlg\downarrow Y$ the structure of a semicartesian symmetric monoidal category with arbitrarily infinite tensor products; we leave the details to the interested reader.  An alternative construction of relative products of probability algebras (in the equivalent form of relative coproducts) is given in \cite[Section 458]{fremlinvol4}.

\section{Alternate construction via the Loomis--Sikorski theorem}\label{loomis-sec}

In this section we provide an alternate construction of the canonical model functor that avoids use of Riesz and probability dualities, proceeding instead via Stone duality.  This alternate construction is lengthier, but reveals more topological features of the canonical model, in particular that it is a Stone space in which the null sets are precisely the Baire-meager sets.  The functor $\Stone$ constructed in this fashion is not strictly speaking identical to the one constructed in Section \ref{canonical-sec}, but will turn out to be equivalent up to natural isomorphism.

 \begin{figure}
    \centering
    \begin{tikzcd}
\StoneCatSigma \arrow[r,blue, tail] \arrow[d,"\ClopenSigma", two heads,  tail, shift left=.75ex] \arrow[rrr,bend left,blue, tail,"\BaireMeagFunc"',two heads] & \StoneCat\arrow[d,"\Clopen", two heads,  tail, shift left=.75ex] \arrow[r,blue, tail,two heads] & \CH \arrow[d,blue,tail,"\BaireFunc"] & \CHDelete \arrow[l,tail,blue] \arrow[d,tail,blue] & \CHProb \arrow[l,blue,tail] \arrow[d,blue,tail] \\
\SigmaAlgop \arrow[u, tail,two heads,"\StoneFuncSigma",shift left=.75ex]  \arrow[r,blue, tail] \arrow[drr,blue, two heads, "\id"'] & \Boolop
\arrow[u, tail,two heads,"\StoneFunc",shift left=.75ex] & \ConcMes \arrow[d,blue,"\Abs"] & \ConcDelete \arrow[l,blue,tail] \arrow[d,blue,"\Abs"] & \ConcProb \arrow[l,blue,tail] \arrow[d,blue,"\Abs"] \\
&& \AbsMes \arrow[uur,"\Loomis", tail, near end,two heads] \arrow[llu,blue, two heads]
& \AbsDelete \arrow[l,blue,tail] \arrow[l,"\ominus",shift left=1.5ex] 
& \AbsProb \arrow[l,blue,tail] \arrow[d,blue,"\Mes"]
\\ 
&&&& \ProbAlg \arrow[u,"\Inc",shift left=1.25ex, tail] \arrow[uuu,bend right,"\Stone"',  tail, two heads, shift right = 1.5ex]
\end{tikzcd}
    \caption{Alternative construction of the canonical model functor $\Stone$. Casting functors (in blue) commute, but the other functors only partially commute with the rest of the diagram.}
    \label{fig:altcon}
\end{figure}
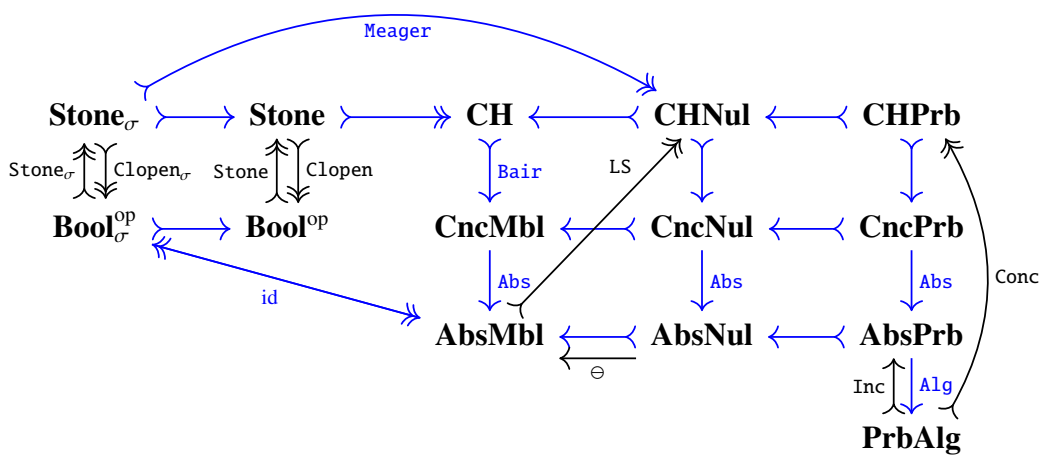

The construction is summarized in Figure \ref{fig:altcon}.  As this figure indicates, it requires several additional categories and functors.  We begin with the categories and functors associated to Stone duality.  Define a \emph{Baire-meager} set to be a Baire set that is also meager (the countable union of nowhere dense sets).

\begin{definition}[Stone duality]\label{stone-dual}\ 
\begin{itemize}
    \item[(i)] $\StoneCat$ is the full subcategory of $\CH$ where the $\StoneCat$-spaces are Stone spaces (i.e., totally disconnected $\CH$-spaces, or equivalently, $\CH$-spaces whose clopen sets form a base for the topology).  
    \item[(ii)] $\StoneCatSigma$ is the subcategory of $\StoneCat$ where the $\StoneCatSigma$-spaces are $\StoneCat$-space whose Baire-measurable sets are equal to clopen sets modulo Baire-meager sets, and whose $\StoneCatSigma$-morphisms are $\StoneCat$-morphism such that pullbacks of Baire-meager sets are Baire-meager.
    \item[(iii)]  There is the obvious forgetful faithful functor from $\StoneCatSigma$ to $\StoneCat$, and the forgetful full faithful functor from $\StoneCat$ to $\CH$.
    \item[(iv)] If $\B$ is a $\Bool$-algebra, $\StoneFunc(\B)$ is the $\StoneCat$-space $$\StoneFunc(\B)\coloneqq \Hom_\Bool(\B \to \{0,1\}),$$ which we view as a compact subspace of the $\StoneCat$-space $\{0,1\}^\B$ and thus is also a $\StoneCat$-space.  If $\Phi \colon \B \to \B'$ is a $\Bool$-homomorphism, we define the $\StoneCat$-morphism $\StoneFunc(\Phi) \colon \StoneFunc(\B') \to \StoneFunc(\B)$ by the formula
    $$ \StoneFunc(\Phi)(\alpha) \coloneqq \alpha \circ \Phi$$
    for all $\alpha \in \StoneFunc(\B)$.
    \item[(v)] If $X$ is a $\StoneCat$-space, $\Clopen(X)$ is the $\Bool$-algebra of clopen subsets of $\StoneFunc(X)$.  If $f \colon X \to Y$ is a $\StoneCat$-morphism, $\Clopen(f) \colon \Clopen(Y) \to \Clopen(X)$ is the $\Bool$-morphism defined by
    $$ \Clopen(f) (E) \coloneqq f^{-1}(E) $$
    for $E \in \Clopen(Y)$.
    \item[(vi)] The functor $\StoneFuncSigma :\SigmaAlg \to \StoneCatSigma$ (resp. $\ClopenSigma \colon \StoneCatSigma \to \SigmaAlg$) is the unique functor that commutes with the corresponding functor $\StoneFunc \colon \Bool \to \StoneCat$ (resp. $\Clopen \colon \StoneCat \to \Bool$) and the faithful functors from $\StoneCatSigma, \SigmaAlg$ to $\StoneCat, \Bool$.
\end{itemize}
\end{definition}

\begin{proposition}[Preliminary Loomis--Sikorski theorem]\label{prelim-ls}\  
\begin{itemize}
\item[(i)]  If $X$ is a $\StoneCat$-space, then the Baire $\sigma$-algebra of $X_\CH$ is generated by the clopen subsets of $X$.
\item[(ii)]  If a subset $E$ of $X$ is equal up to a meager set to a clopen subset of $X$, then the meager and clopen set is determined uniquely by $E$.
\item[(iii)] The categories and functors in Definition \ref{stone-dual} are well-defined and have the faithful and fullness properties indicated in Figure \ref{fig:altcon}.  
\item[(iv)] Also, $\StoneFunc, \Clopen$ form a duality of categories between $\Bool$ and $\StoneCat$, and similarly $\StoneFuncSigma$, $\ClopenSigma$ form a duality of categories between $\SigmaAlg$ and $\StoneCatSigma$.
\end{itemize}
\end{proposition}

\begin{proof} For (i), observe that as the clopen subsets of the $X$ separate points, the linear combinations of indicator functions of these clopen subsets are dense in $\CFunc(X)$ by the Stone-Weierstrass theorem.  The claim follows.

The claim (ii) is immediate from the Baire category theorem (no non-empty clopen set is meager).
Now we turn to (iii), (iv).  The well-definedness of the categories and functors in Definition  \ref{stone-dual}(i)-(v) is clear.  The fact that $\StoneFunc, \Clopen$ give a duality of categories is standard (e.g., see \cite[Chapter 3]{koppelberg89}).  To verify that $\StoneFuncSigma$ is well-defined, we need to show that for a $\SigmaAlg$-algebra $\B$, that the Baire sets of $\StoneFunc(\B_\Bool)$ are clopen modulo Baire-meager sets, and for a $\SigmaAlg$-morphism $\phi : \B \to \B'$ that the $\StoneCat$-morphism $$\StoneFunc(\phi_\Bool) : \StoneFunc(\B'_\Bool) \to \StoneFunc(\B'_\Bool)$$ pulls back Baire-meager sets to Baire-meager sets.  For the first claim, observe from the $\sigma$-completeness of $\B$ that the collection of subsets of $\StoneFunc(\B_\Bool)$ that differ from a clopen set by a Baire-meager set is a $\sigma$-algebra of Baire sets containing the clopen sets, giving the claim.  For the second claim, let us call a subset of a $\StoneCat$-space $X$ \emph{Baire-meager*} if it is Baire measurable and can be covered by countably many nowhere dense compact sets, each of which is the intersection of countably many clopen sets.  Repeating the arguments from the first claim we see that every Baire set is uniquely representable as a clopen set modulo Baire-meager* sets, hence the notions of Baire-meager* and Baire-meager coincide (since trivially every Baire-meager* set is Baire-meager).  It is not difficult to verify that $\StoneFunc(\phi_\Bool)$ pulls back Baire-meager* sets to Baire-meager* sets, giving the second claim.  

To verify that $\ClopenSigma$ is well-defined, we have to show that for a $\StoneCatSigma$-space $X$, that the $\Bool$-algebra $\Clopen(X_\StoneCat)$ is $\sigma$-complete, and that for a $\StoneCatSigma$-morphism $T \colon X \to Y$, the $\Bool$-morphism $\Clopen(T_\StoneCat)$ can be promoted to a $\SigmaAlg$-morphism.  For the first claim, let $E_n, n \in \N$ be an increasing sequence of clopen sets in $X_\StoneCat$, then $\bigcup_{n \in \N} E_n$ is Baire measurable, hence equal modulo a Baire-meager set to a unique clopen set $E$.  By the Baire category theorem, the clopen Baire-meager sets $E_n \backslash E$ are empty, thus $E$ is  the join of the $E_n$ in the clopen $\Bool$-algebra, giving\footnote{We caution however that the $\sigma$-completeness of the clopen algebra does \emph{not} imply that a countable union of clopen sets is clopen, because the countable join of the clopen algebra need \emph{not} be given by countable union.} the first claim.  For the second claim, if $E_n$ is a decreasing sequence of clopen sets in $\Clopen(X_\StoneCat)$ with $\bigwedge_{n \in \N} E_n = 0$, then $\bigcap_{n \in \N} E_n$ is Baire-meager, hence so is the pullback $\bigcap_{n \in \N} T^* E_n$, hence $\bigwedge_{n \in \N} T^* E_n = 0$, giving the second claim. 

If $X$ is a $\StoneCatSigma$-space, then  we see that $\StoneFuncSigma(\ClopenSigma(X))$ is equal in $\StoneCat$ to $\StoneFunc(\Clopen(X_\StoneCat)))$ by chasing the definitions, which by ordinary Stone duality is $\StoneCat$-isomorphic to $X$.  Therefore the $\StoneCatSigma$-spaces  $X$ and $\StoneFuncSigma(\ClopenSigma(X))$  are homeomorphic, hence also $\StoneCatSigma$-isomorphic since the definition of the category $\StoneCatSigma$ is purely topological in nature.  It is then a routine matter to verify that this isomorphism is natural.  Similarly, if $\B$ is a $\SigmaAlg$ space, then $\ClopenSigma(\StoneFuncSigma(\B))$ is $\Bool$-isomorphic to $\B$, hence also $\SigmaAlg$-isomorphic as the definition of $\SigmaAlg$ is purely Boolean algebra-theoretic in nature, and again it is a routine matter to verify that the isomorphism is natural.  The remaining claims in (iii), (iv) then follow from a tedious but routine verification.
\end{proof}

As one quick corollary of the above proposition we see that a categorical product on $\StoneCat$ (resp. $\StoneCatSigma$) exists and agrees with the categorical coproduct on $\Bool$ (resp. $\SigmaAlg$) with respect to $\StoneFunc$, $\Clopen$ (resp. $\StoneFuncSigma$, $\ClopenSigma$).  The $\StoneCat$ product can be verified to agree with the $\CH$ product with respect to the forgetful faithful functor, but the situation with the $\StoneCatSigma$ product is more subtle; as we shall see in Remark \ref{stonecat-remark}, the $\StoneCatSigma$ product is a (non-trivial) quotient of the $\StoneCat$ product.  Similarly for $\AbsMes$ (which is an equivalent category to $\StoneCat_\sigma$, as can be seen from Figure \ref{fig:altcon}).  As another application of the Stone dualities in the above proposition and
Lemma \ref{bool-props}(i), as well as\footnote{One also needs the fact (easily obtained from Zorn's lemma) that any $\Bool$-homomorphism $\phi \colon {\mathcal B}' \to \{0,1\}$ on a $\Bool$-subalgebra ${\mathcal B}'$ of a $\Bool$-algebra ${\mathcal B}$ can be extended (not necessairly uniquely) to ${\mathcal B}$.} Lemma \ref{epimorph}, we see that for $\Cat = \StoneCat, \StoneCatSigma$, that a $\Cat$-morphism is a $\Cat$-monomorphism (resp. a $\Cat$-epimorphism) if and only if it is injective (resp. surjective).

Next, we ``factor'' the forgetful functors from $\CHProb$, $\ConcProb$, $\AbsProb$ to $\CH$, $\ConcMes$, $\AbsMes$ respectively in Figure \ref{fig:prob-mes} by inserting categories intermediate between measurable spaces and measure spaces, in which there is an ideal of null sets, but no actual measure assigned to the space.

\begin{definition}[Null set categories]\label{null-def}\ 
\begin{itemize}
\item[(i)]  An \emph{$\AbsDelete$-space} is a pair $X = (X_\AbsMes, {\mathcal N}_X)$, where $X_\AbsMes = \X$ is an $\AbsMes$-space and ${\mathcal N}_X$ is a $\sigma$-ideal of $\X$ (a downwardly closed subset of $\X$ containing $0$ that is closed under countable joins).  Elements of ${\mathcal N}_X$ will be called \emph{null sets} of the $\AbsDelete$-space, and ${\mathcal N}_X$ itself will be called the \emph{null ideal}.  An \emph{$\AbsDelete$-morphism} $T \colon X \to Y$ between $\AbsDelete$-spaces $X = (X_\AbsMes, {\mathcal N}_X)$,  $Y = (Y_\AbsMes, {\mathcal N}_Y)$ is an $\AbsMes$-morphism $T_\AbsMes \colon X_\AbsMes \to Y_\AbsMes$ such that $T_\SigmaAlg( {\mathcal N}_X ) \subseteq{\mathcal N}_Y$ (i.e., null sets pull back to null sets).  There are  obvious forgetful functors from $\AbsDelete$ to $\AbsMes$, and from $\AbsProb$ to $\AbsDelete$ (where the null ideal is the ideal of sets of measure zero).
\item[(ii)]  A \emph{$\ConcDelete$-space} is a pair $X = (X_\ConcMes, {\mathcal N}_X)$, where $X_\ConcMes$ is a $\ConcMes$-space, and $X_\AbsDelete \coloneqq (X_\AbsMes, {\mathcal N}_X)$ is an $\AbsDelete$-space.  A \emph{$\ConcDelete$-morphism} $T \colon X \to Y$ is a $\ConcMes$-morphism $T_\ConcMes \colon X_\ConcMes \to Y_\ConcMes$ such that $T_\AbsMes \colon X_\AbsMes \to Y_\AbsMes$ can be promoted to an $\AbsDelete$-morphism from 
$X_\AbsDelete$ to $Y_\AbsDelete$.  There are obvious forgetful functors from $\ConcDelete$ to $\ConcMes$ and $\ConcProb$ to $\ConcDelete$, and an abstraction functor $\Abs$ from $\ConcDelete$ to $\AbsDelete$.
\item[(iii)]  A \emph{$\CHDelete$-space} is a pair $X = (X_\CH, {\mathcal N}_X)$, where $X_\CH$ is a $\CH$-space, and $X_\ConcDelete \coloneqq (X_\ConcMes, {\mathcal N}_X)$ is a $\ConcDelete$-space.  A \emph{$\CHDelete$-morphism} $T: X \to Y$ is a $\CH$-morphism $T_\CH :X_\CH \to Y_\CH$ such that $T_\ConcMes : X_\ConcMes \to Y_\ConcMes$ can be promoted to an $\ConcDelete$-morphism from 
$X_\ConcDelete$ to $Y_\ConcDelete$.  There are obvious forgetful functors from $\CHDelete$ to $\CH$, from $\CHDelete$ to $\ConcDelete$, and from $\CHProb$ to $\ConcDelete$.
\item[(iv)]  If $X$ is a $\StoneCatSigma$-space, $\BaireMeagFunc(X) = X_\CHDelete$ is the $\CHDelete$-space $(X_\CH, {\mathcal N}_X)$, where ${\mathcal N}_X$ is the ideal of Baire-meager sets in $X_\CH$.  If $T \colon X \to Y$ is a $\StoneCatSigma$-space, then $\BaireMeagFunc(T) = T_\CHDelete$ is the unique promotion of $T_\CH \colon X_\CH \to Y_\CH$ to a $\CHDelete$-morphism from $X_\CHDelete$ to $Y_\CHDelete$.  (Here it is important that $\StoneCatSigma$-morphisms pull back Baire-meager sets to Baire-meager sets.)
\end{itemize}
\end{definition}

It is easy to verify that the categories and functors in Definition \ref{null-def} are well-defined. This defines all the casting functors (the functors in  blue) in Figure \ref{fig:altcon}, and it is routine to check that these casting functors commute with each other (and with the casting functors in Figure \ref{fig:prob-mes}), and have the indicated faithfulness and fullness properties.

We define the \emph{Loomis--Sikorski functor} $\Loomis \colon \AbsMes \to \CHDelete$ by the formula
$$ \Loomis \coloneqq \BaireMeagFunc \circ \StoneFuncSigma.$$
From the functorial properties already established in Figure \ref{fig:altcon} we see that $\Loomis$ is full and faithful, as depicted in that figure.  This functor can be viewed as an analogue of the canonical model functor $\Stone \colon \ProbAlg \to \CHProb$, but between categories of measurable spaces rather than categories of probability spaces.

Next, we define a \emph{deletion functor} $\ominus$ from $\AbsDelete$ to $\AbsMes$:

\begin{definition}[Deletion functor]\ 
\begin{itemize}
    \item[(i)] If $X = (\X, {\mathcal N}_X)$ is an $\AbsDelete$-space, we define $\ominus(X)$ to be the $\AbsMes$-space 
    $$ \ominus(X) \coloneqq \X / {\mathcal N}_X.$$
    \item[(ii)]  If $T \colon X \to Y$ is an $\AbsDelete$-morphism between $\AbsDelete$-spaces $X = (\X, {\mathcal N}_X)$, $Y = (\Y, {\mathcal N}_Y)$, we let $\ominus(T) \colon \ominus(X) \to \ominus(Y)$ be the $\AbsMes$-morphism defined by setting $\ominus(T)_\SigmaAlg \colon \Sigma_Y / {\mathcal N}_Y \to \Sigma_X / {\mathcal N}_X$ be the descent of $T_\SigmaAlg \colon \Sigma_Y \to \Sigma_X$ by quotienting out the null ideals.
\end{itemize}
\end{definition}

It is not difficult to verify that $\ominus$ is a   functor from $\AbsDelete$ to $\AbsMes$.

\begin{remark}\label{incl} Using Stone duality, one can identify an $\AbsDelete$-space with a $\StoneCatSigma$-space $X$ together with an open subset $U$ of $X$ with the property that the countable join (in the clopen algebra) of any clopen subsets of $U$ remains in $U$.  The deletion functor then corresponds to deleting this open set $U$ from the $\StoneCatSigma$-space $X$ to create a new $\StoneCatSigma$-space $X \backslash U$. This may help explain the term ``deletion functor''.  Related to this, there is a natural monomorphism from $\ominus$ to $\Forget_{\AbsDelete \to \AbsMes}$, where the $\AbsMes$-inclusion $\iota \colon \ominus(X) \to X_\AbsMes$ for an $\AbsDelete$-space $X$ is defined by requiring $\iota_\SigmaAlg \colon \Sigma_X \to \Sigma_X / {\mathcal N}_X$ to be the quotient map.  It is a routine matter to verify that this is indeed a natural monomorphism.
\end{remark}

\begin{remark}  Applying Stone and Gelfand duality to the full and faithful functor of $\StoneCat$ to $\CH$, one expects to have a full faithful functor from $\Bool$ to $\CStarAlgUnit$.  This functor can be described explicitly by mapping a $\Bool$-algebra $\B$ to the associated $\CStarAlgUnit$-algebra $\overline{\C \otimes \B}$ formed by taking the $C^*$-algebra closure of formal complex linear combinations of elements of $\B$ (which can be given the structure of a *-algebra), and also mapping $\Bool$-morphisms accordingly.  We leave the details to the interested reader.
\end{remark}

Now we can give our version of the well-known Loomis--Sikorski theorem that gives a concrete representation to $\SigmaAlg$-algebras (or $\AbsMes$-spaces).

\begin{theorem}[Loomis--Sikorski theorem]\label{ls-thm}   The functor $\ominus \circ \Cast_{\CHDelete \to \AbsDelete} \circ \Loomis$ is naturally isomorphic to $\ident_{\AbsMes}$.  In particular, by Remark \ref{incl}, there is a natural monomorphism from $\ident_{\AbsMes}$ to
$\Cast_{\CHDelete \to \AbsMes} \circ \Loomis$.
\end{theorem}

\begin{proof} If $X$ is an $\AbsMes$-space, we define the associated $\AbsMes$-isomorphism
$$ \phi_X \colon \ominus( \Loomis(X)_\AbsDelete ) = \Loomis(X)_\SigmaAlg / {\mathcal N}_{\Loomis(X)} \to X $$
via its opposite
$$ (\phi_X)_\SigmaAlg \colon \Sigma_X \to \Loomis(X)_\SigmaAlg / {\mathcal N}_{\Loomis(X)}$$
by the formula
$$ (\phi_X)_\SigmaAlg(E) \coloneqq \pi(\{ \alpha \in \Loomis(X) \colon \alpha(E) = 1 \}),$$
where $\pi \colon \Loomis(X)_\SigmaAlg  \to \Loomis(X)_\SigmaAlg  / {\mathcal N}_{\Loomis(X)}$ is the quotient $\SigmaAlg$-morphism.  It is clear that $(\phi_X)_\SigmaAlg$ is a $\SigmaAlg$-morphism; it is injective by Proposition \ref{prelim-ls}(ii), and surjective because $\Loomis(X)_\CH$ can be promoted to a $\StoneCatSigma$-space. By Lemma \ref{bool-props}(ii), $\phi_X$ is an $\AbsMes$-isomorphism, and it is a routine matter to then conclude that $X \mapsto \phi_X$ is a natural isomorphism. 
\end{proof}

\begin{remark}\label{rem-loomissikorski}  The usual formulation of the Loomis--Sikorski theorem (as given for instance in \cite[314M]{fremlinvol3}) completes the Baire $\sigma$-algebra $\Loomis(X)_\SigmaAlg$ on $\Loomis(X)$ by including any set which differs from a clopen set by an \emph{arbitrary} meager set (not just a Baire-meager set), and similarly enlarging the null ideal to contain all meager sets.  From Proposition \ref{prelim-ls}(ii), this does not affect the quotient $\SigmaAlg$-algebra which remains isomorphic to $X_\SigmaAlg$.  However, this modification of $\Loomis(X)$ would no longer lie in $\CHDelete$ as the $\sigma$-algebra no longer is given by the Baire $\sigma$-algebra.  One can view this more traditional Loomis--Sikorski construction as the completion of the one used in this paper, but we have (perhaps surprisingly) found the hypothesis of completeness for the $\sigma$-algebras one encounters to be of little benefit, whereas the use of Baire $\sigma$-algebras is much more compatible with the topological structure of the spaces involved. 
\end{remark}

We now construct an alternate version $\Stone' \colon \ProbAlg \to \CHProb$ of the canonical model functor $\Stone \colon \ProbAlg \to \CHProb$.

\begin{theorem}[Alternate canonical model functor]\label{alt-canon-thm}\ 
\begin{itemize}
    \item[(i)]  There exists a unique   functor $\Stone' \colon \ProbAlg \to \CHProb$ such that 
    $$\Cast_{\CHProb \to \CHDelete} \circ \Stone' = \Loomis \circ \Cast_{\AbsProb \to \AbsMes} \circ \Inc$$
    and the natural monomorphism from $$\Forget_{\AbsProb \to \AbsMes} \circ \Inc$$  to $$\Cast_{\CHDelete \to \AbsMes} \circ \Loomis \circ \Cast_{\AbsProb \to \AbsMes} \circ \Inc$$ can be promoted to a natural monomorphism from $\Inc$ to $\Cast_{\CHProb \to \AbsProb} \circ \Stone'$.
    \item[(ii)]  $\Stone'$ is naturally isomorphic to $\Stone$.
\end{itemize}
\end{theorem}

We remark that a variant of this construction appears implicitly in \cite[\S 3]{doob-ratio}, where in particular the strong Lusin property of the model is noted, which is already also mentioned in \cite{dieudonne-counter,halmos-dieudonne} prior to \cite{doob-ratio}.  

\begin{proof}  For (i), we define $\Stone'(X)$ for a $\ProbAlg$-space $X$ to be the promotion of $\Stone'(X)_\ConcDelete \coloneqq \Loomis(\Inc(X)_\AbsMes)$ to a $\CHProb$-space defined by setting $\mu_{\Stone'(X)}$ to be the pushforward of $\mu_X$ using the natural $\AbsMes$-inclusion from $\Inc(X)_\AbsMes$ to $\Stone'(X)_\AbsMes$, and for any $\ProbAlg$-morphism $T \colon X \to Y$ defining $\Stone'(T) \colon \Stone'(X) \to \Stone'(Y)$ to be the unique promotion of $\Stone'(T)_\ConcDelete = \Loomis(\Inc(T)_\AbsMes)$ to a $\CHProb$-morphism from $\Stone'(X)$ to $\Stone'(Y)$.  It is  a routine matter to show that this defines a   functor.  To verify the properties in (i), the only non-trivial task is to show that the null ideal of $\Stone'(X)$ agrees with the Baire-meager ideal.  By construction all Baire-meager sets have measure zero, hence as $\Stone'(X)$ comes from a $\StoneCatSigma$-space, it suffices to show that non-empty clopen sets have positive measure.  But by construction, the measure that $\mu_{\Stone'(X)}$ assigns to a clopen set is equal to the measure that $\mu_X$ assigns to the corresponding element of $\Inc(X)_\SigmaAlg$ arising from Stone duality, and the claim follows from the probability algebra nature of $X$.  Finally, the uniqueness claim in (i) is easily verified by expanding out all the definitions.

Now we prove (ii).  From (i) we see that for any $\ProbAlg$-space $X$, $(\Stone'(X), \iota_X)$ is a concrete model of $X$, where $\iota_X \colon \Inc(X) \to \Stone'(X)_\AbsProb$ is the natural monomorphism.  By construction, every indicator function in $\Linfty(\Stone'(X))$ is equal (modulo almost everywhere equivalence) to the unique indicator function of a clopen set, which of course lies in $\CFunc(\Stone'(X))$.  By linearity and density we conclude that every function in $\Linfty(\Stone'(X))$ is equivalent to a function in $\CFunc(\Stone'(X))$.  Since the topology of $\Stone'(X)$ is generated by clopen sets, and non-empty clopen sets have positive measure, we see that any two distinct elements in $\CFunc(\Stone'(X))$ also differ in $\Linfty(\Stone'(X))$.  Thus $\Stone'(X)$ obeys the strong Lusin property $\Linfty(\Stone'(X)) = \CFunc(\Stone'(X))$, hence by Proposition \ref{canon-model}(iii) $(\Stone'(X), \iota_X)$ is a initial concrete model of $X$.  By Proposition \ref{model-basic} the same is true for $\Stone$, and it is then a routine matter to construct the natural isomorphism between $\Stone$ and $\Stone'$.
\end{proof}

\begin{remark}
One can dispense with the Loomis--Sikorski functor to construct the canonical model directly from Stone duality\footnote{We thank the anonymous referee for suggesting this alternative proof.} by relying on the Carath\'eodory extension theorem. Indeed, if $(X,\mu)$ is a $\ProbAlg$-algebra, we can apply Stone duality, the Stone--Weierstra\ss \, theorem and the Carath\'eodory--Hahn--Kolmogorov extension theorem to extend the finitely additive measure $\mu$ from $\Clopen(\Stone(X))$ to $\Baire(\Stone(X))$, and then verify as above that the null ideal  coincides with the ideal of Baire meager sets, giving the claim. 
\end{remark}

In view of this natural isomorphism (and also because $\Stone'$ is easily verified to be injective on objects) one can replace $\Stone'$ by $\Stone$ without any substantial change to the statements in this paper if desired.

\begin{remark}[Equivalent forms of the strong Lusin property]\label{lusin-equiv}  From the above equivalences it is not difficult to see that for any $\CHProb$-space $X$, the following claims are equivalent:
\begin{itemize}
\item[(i)]  $X$ has the strong Lusin property.
\item[(ii)]  $X$ is $\CHProb$-isomorphic to $\Stone(Y)$ (or equivalently, $\Stone'(Y)$) for some $\ProbAlg$-space $Y$.
\item[(iii)] $X$ is $\CH$-isomorphic to a $\StoneCatSigma$-space, and the ideal of Baire null sets coincides with the ideal of Baire-meager sets.
\end{itemize}
In particular, the measures on a $\CHProb$-space $X$ with the strong Lusin property are \emph{hyperdiffusive}\footnote{Such measures were also termed \emph{residual} in \cite{armstrong}.} in the sense of Fishel and Parret  \cite{fishel} (all measurable meager sets are null).  The results of \cite{fishel} then imply that such measures are also normal in the sense of Dixmier \cite{dixmier}, in that one has $\sup_\alpha \int_X f_\alpha = \int_X f$ whenever $f_\alpha, \alpha \in A$ is an increasing family in $\CFunc(X)$ indexed by a directed set $A$ that has a least upper bound $f$ in the lattice $\CFunc(X)$.  (Note that this is \emph{not} the same as asserting that $f$ is the pointwise supremum of the $f_\alpha$.)  Also, because the $\SigmaAlg$-algebra associated to a $\ProbAlg$-space is complete, one can show that spaces obeying any of (i), (ii), (iii) are not merely $\StoneCatSigma$-spaces, but are in fact Stonean spaces (extremally disconnected Stone spaces), as the category of such spaces (with open continuous morphisms) is known (see, e.g., \cite{sikorski2013boolean}) to be dual to the category of complete Boolean algebras. We also remark that measures obeying the second conclusion of (iii) are referred to as \emph{category measures} in \cite{oxtoby}. 
We also mention \cite{nagel,nagel1,seever,stone} for variations of these themes in the theory of Banach lattices. 
\end{remark}

\begin{remark}\label{rem-lsmodel}
In \cite{fremlinvol3}, Fremlin employs the concrete model provided by the traditional Loomis--Sikorski representation (see Remark \ref{rem-loomissikorski}) to develop basic results in abstract measure theory (we collected some examples in Remark \ref{rem-lpspaces}). 
As shown in \cite[\S 363 C]{fremlinvol3}, the traditional Loomis--Sikorski concrete model enjoys a strong Lusin property. 
Also it can be used to define arbitrary categorical coproducts in $\SigmaAlg$ and arbitrarily infinite tensor products in $\ProbAlg$ (see \cite[Section 325]{fremlinvol3}). 
However it lacks the functorial properties of our canonical model (as developed in Section \ref{canonical-sec}) and thus the category-theoretical compatibility with the adjacent topological and functional analytic categories. For example we can provide two constructions of our canonical model based on Stone duality and on Riesz duality respectively, whereas the traditional Loomis--Sikorski concrete model rests only on Stone duality. This compatibility is essential in  applications of the canonical model to uncountable ergodic theory (cf., \cite{jt19,jt20,jamneshan2019fz,roth}). 
\end{remark}

\begin{remark}\label{rem-lpspaces}
As demonstrated in \cite{fremlinvol3}, one can develop many basic results in measure theory for measure algebras in abstract form and relate them to their classical counterparts for the traditional Loomis--Sikorski model. 
For example in \cite[Sections 363-366]{fremlinvol3}, abstract $L^p$-spaces on measure algebras are introduced.   
Given a $\Bool$-algebra $X$,  the space of abstract simple functions $\mathfrak{S}(X)$ is defined to be the linear hull of indicator functions $1_E$, where $E$ is in the clopen algebra  $\StoneFunc(X)_\Bool$ (see  \cite[\S 361 D]{fremlinvol3}). Then  the $L^\infty$-space of $X$ is defined to be $\CFunc{(\StoneFunc(X))}$ (see \cite[\S 363 A]{fremlinvol3}), and it shown that $\mathfrak{S}(X)$ is dense in $\CFunc{(\StoneFunc(X))}$ (see \cite[\S 363 C]{fremlinvol3}). 
If $(X,\mu)$ is a measure algebra, the abstract $L^\infty$-space of $X$ can be identified with the concrete $L^\infty$-space of its traditional Loomis--Sikorski space. 
This identification is in the sense of a simultaneous Riesz\footnote{A Riesz space is an ordered vector space in which the order and vector space structures are compatible.} space isomorphism  and Banach space isomorphism. Hence a strong Lusin property starting with  $\CFunc{(\StoneFunc(X))}$ as the definition of an abstract $L^\infty$-space is derived.  

The abstract $L^0$-space of $(X,\mu)$ is defined to be the set of all $\Set$-functions $f:\R\to X$ such that 
\begin{itemize}
    \item[(i)] $f(r)=\bigvee_{r'>r} f(r')$ for all $r\in \R$, 
    \item[(ii)] $\bigwedge_{r\in \R} f(r)=0$, 
    \item[(iii)] $\bigvee_{r\in \R} f(r)=1$,   
\end{itemize}
see \cite[\S 364 A]{fremlinvol3}. (This definition mimics the defining properties of the level set function $r\mapsto \{f> r\}$ for a real measurable function $f:X\to \R$, where now $X$ is a concrete measure space.) This abstract $L^0$-space is isomorphic to the space of $\SigmaAlg$-homomorphisms $\Hom{(\Borel(\R)\to X)}$ as Riesz spaces (see \cite[Theorem 364 D]{fremlinvol3}). Moreover, the abstract  $\mathfrak{S}(X),L^\infty(X)$ spaces are Riesz subspaces of the abstract $L^0(X)$ (see \cite[\S 364 K]{fremlinvol3}).  
Using the level-set description of abstract measurable maps, an abstract $L^1$-norm can be introduced by the traditional Lebesgue integral 
$$\|f\|_1\coloneqq \int_0^\infty \mu(|f|>r) dr,$$ 
which allows to derive a definition of abstract $L^p$-spaces as  
\[
L^p(X)\coloneqq \{f\in L^0(X)\colon \||f|^p\|_1<\infty \}, 
\]
where $\{|f|^p>r\}$ is equal to $\{|f|>r^{1/p}\}$ for $r\geq 0$ and $1\in X$ otherwise (see \cite[\S 366 A]{fremlinvol3}). 
It can be shown that these abstract $L^p$-spaces are isomorphic to the concrete $L^p$-spaces of the Loomis--Sikorski concrete model in the sense of Riesz space and Banach space isomorphies (see \cite[\S 365 B, 366 B]{fremlinvol3}).  
One can check that the definition of $L^\infty, L^2$ in \cite{fremlinvol3}, when applied to a $\ProbAlg$-space, agrees (up to natural identifications) with the one given here.

It is remarkable that several basic results such as the Radon-Nikod\'ym theorem, the $L^p$-$L^q$-duality and existence of conditional expectations have proofs in abstract $L^p$-spaces without using a concrete representation (cf., \cite[\S 366 D, \S 365 E, \S 365 R]{fremlinvol3}). The Radon-Nikod\'ym theorem can be used to construct relative products (see  \cite[Section 458]{fremlinvol4}).
\end{remark}

Next we give an explicit description of the categorical product in $\AbsMes$.  This will be done in terms of a categorical product on $\CHDelete$:

\begin{definition}[$\ConcDelete$, $\CHDelete$ and $\AbsMes$ products]\label{delete-prod}\text{}  
\begin{itemize}
    \item[(i)] Let $X_\alpha = ((X_\alpha)_\ConcMes, {\mathcal N}_{X_\alpha})$, $\alpha \in A$ be a family of $\ConcDelete$-spaces.  We define the product $\prod^\ConcDelete_{\alpha \in A} X_\alpha$ to be the $\ConcDelete$-space
$$X \coloneqq ( X_\ConcMes, {\mathcal N}_X ),$$
where $X_\ConcMes \coloneqq \prod^\ConcMes_{\alpha \in A} (X_\alpha)_{\ConcMes}$, and ${\mathcal N}_X$ is the $\sigma$-ideal of $X_\ConcMes$ generated by $\bigcup_{\beta \in A} ({\pi_\beta})_{\SigmaAlg}( \mathcal{N}_{X_\beta}),$ where $({\pi_\beta})_\ConcMes \colon X_\ConcMes \to (X_\beta)_\ConcMes$ are the canonical $\ConcMes$-projections.  We also promote the ${\pi_\beta}_\ConcMes$ to $\ConcDelete$-morphisms $\pi_\beta \colon X \to X_\beta$ in the obvious fashion.
    \item[(ii)]  If $X_\alpha$, $\alpha \in A$ are a family of $\CHDelete$-spaces, we define $\prod^\CHDelete_{\alpha \in A} X_\alpha$ to be the unique $\CHDelete$-space whose cast to $\Cat$ is $\prod^\Cat_{\alpha \in A} (X_\alpha)_\Cat$ for $\Cat = \CH, \ConcDelete$, and define the projections $\pi_\beta \colon \prod^\CHDelete_{\alpha \in A} X_\alpha \to X_\beta$ similarly.
    \item[(iii)]  Let $X_\alpha, \alpha \in A$ be a family of $\AbsMes$-spaces.  We define the product $\prod^\AbsMes_{\alpha \in A} X_\alpha$ as
    $$ \prod^\AbsMes_{\alpha \in A} X_\alpha \coloneqq \ominus \circ \Cast_{\CHDelete \to \AbsDelete} \left(  \prod^\CHDelete_{\alpha \in A} \Loomis(X_\alpha) \right) $$
    with the $\AbsMes$-projection morphisms ${\pi_\beta} \colon \prod^\AbsMes_{\alpha \in A} X_\alpha \to X_\beta$ defined analogously.
\end{itemize}
\end{definition}

\begin{proposition}[Universality of $\ConcDelete$-product and $\AbsMes$-product]\label{absmes-prod} \text{}  
\begin{itemize}
    \item [(i)]  For $\Cat = \ConcDelete, \CHDelete, \AbsMes$, the $\Cat$-product defined in Definition \ref{delete-prod} is universal.
    \item[(ii)]  The $\CH$, $\CHDelete$, $\ConcMes$, and $\ConcDelete$ products agree with each other with respect to forgetful functors, and the categorical product in $\ConcDelete$ agrees with the categorical product in $\AbsMes$ with respect to $\ominus \circ \Cast_{\ConcDelete \to \AbsDelete}$.
\end{itemize}
\end{proposition}

\begin{proof}  We first prove (i) for $\Cat = \ConcDelete$.  Thus suppose we have $\ConcDelete$-morphisms $f_\alpha \colon Y \to X_\alpha$, and we wish to lift these to a common $\ConcDelete$-morphism $f \colon Y \to X$ with $X \coloneqq \prod^\ConcDelete_{\alpha \in A} X_\alpha$ and $f_\beta = \pi_\beta \circ f$ for all $\beta \in A$.  Uniqueness follows easily from the universality of the $\ConcMes$-product; but existence also follows easily from observing that the $\ConcMes$-morphism $(f_\alpha)_{\alpha \in A}^\ConcMes \colon Y_\ConcMes \to X_\ConcMes$ can be promoted to a $\ConcDelete$-morphism from $Y$ to $X$.  The claim (i) for $\Cat = \CHDelete$ is established similarly.  The claims (ii) are then routinely verified (using Proposition \ref{prod-top}(viii)).

It remains to verify the $\Cat = \AbsMes$ case of (i).  Let let $f_\alpha \colon Y \to X_\alpha$, $\alpha \in A$ be $\AbsMes$-morphisms for various $\AbsMes$-spaces $Y, X_\alpha$.  We wish to show that there is a unique $\AbsMes$-morphism $f \colon Y \to \prod^\AbsMes_{\alpha \in A} X_\alpha$ whose projections to $X_\beta$ equal $f_\beta$ for all $\beta \in A$.  For existence, we apply the categorical product in $\CHDelete$ to $\Loomis(f_\alpha) \colon \Loomis(Y) \to \Loomis(X_\alpha)$ to obtain a map $(\Loomis(f_\alpha))_{\alpha \in A}^\CHDelete \colon \Loomis(Y) \to \prod_{\alpha \in A}^\CHDelete \Loomis(X_\alpha)$.  Applying $\ominus \circ {\Cast}_{\CHDelete \to \AbsDelete}$ and the Loomis--Sikorski theorem, we obtain an $\AbsMes$-morphism $f$ with the required properties.  To obtain uniqueness, it suffices to show that the pullbacks of $(X_\alpha)_\SigmaAlg$ to $(\prod^\AbsMes_{\alpha \in A} X_\alpha)_\SigmaAlg$ generate the entire $\SigmaAlg$-algebra.  The pullbacks of $\Loomis(X_\alpha)_\SigmaAlg$ to $(\prod^{\CHDelete}_{\alpha \in A} \Loomis(X_\alpha))_\SigmaAlg$ generate the entire $\SigmaAlg$-algebra. The claim follows by applying
$\ominus \circ \Cast_{\CHDelete \to \AbsDelete}$
and the Loomis--Sikorski theorem.
\end{proof}

Now we can complete the proof of Lemma \ref{bool-props}.

\begin{corollary}\label{bool-cat-extra}  The categorical $\Bool$-coproduct is contained in the categorical $\SigmaAlg$-coproduct.
\end{corollary}

\begin{proof}  For $\SigmaAlg$-algebras $\X_\alpha$, our task is to show that the natural $\Bool$-map from $\coprod^{\Bool}_{\alpha \in A} \X_\alpha$ to $\coprod^{\SigmaAlg}_{\alpha \in A} \X_\alpha$ is injective.  We may assume that none of the $\X_\alpha$ are the $2^0$-element Boolean algebra $\{0=1\}$, as the claim is trivial in this case.  As is well-known, one can explicitly write down a Boolean coproduct $\coprod^{\Bool}_{\alpha \in A} \X_\alpha$ as the $\Bool$-algebra of formal finite joins of disjoint ``rectangles'' $E_{\alpha_1} \otimes \dots \otimes E_{\alpha_n}$ with $E_{\alpha_i} \in \X_{\alpha_i}$ and $\alpha_1,\dots,\alpha_n \in A$, so it suffices to show that the image of any such ``rectangle'' in $\coprod^{\SigmaAlg}_{\alpha \in A} \X_\alpha$ is non-zero if all of the $E_{\alpha_i}$ are non-zero.

For each $\SigmaAlg$-space $\X_\alpha$ we can form a $\CHDelete$-space $\tilde X_\alpha \coloneqq \Loomis((\X_\alpha)_{\AbsMes})$, which is non-empty since the $\X_\alpha$ are not $2^0$-element algebras.  We then form the categorical product in $\CHDelete$
$$ \tilde X \coloneqq \prod^\CHDelete_{\alpha \in A} \tilde X_\alpha.$$
Each element $E_{\alpha_i}$ then has a counterpart $\tilde E_{\alpha_i} \in (\tilde X_{\alpha_i})_{\SigmaAlg}$ defined by
$$ \tilde E_{\alpha_i} \coloneqq \phi_{\alpha_i}(E_{\alpha_i})$$
where $\phi_{\alpha_i}$  is the natural $\Bool$-isomorphism between $\Clopen( \StoneFunc( (X_{\alpha_i})_\Bool) )$ and $(X_{\alpha_i})_\Bool$.  Since $E_{\alpha_i}$ is non-zero, $\tilde E_{\alpha_i}$ is not in the null ideal of $\tilde X_{\alpha_i}$.  From the axiom of choice, we then see that the product set
$$ \prod_{i=1}^n \tilde E_{\alpha_i} \times \prod_{\alpha \neq \alpha_1,\dots,\alpha_n} \tilde X_{\alpha_i}$$
does not lie in the null ideal of $\tilde X$, and the claim follows.
\end{proof}

\begin{remark}\label{stonecat-remark} The above theory also provides a reasonably explicit, albeit strange, description of the categorical product in $\StoneCatSigma$.  Namely, one can identify $\prod^\StoneCatSigma_{\alpha \in A} X_\alpha$ with the space of $\Bool$-morphisms from $\prod^\ConcMes_{\alpha \in A} (X_\alpha)_\ConcMes$ onto the trivial algebra $\{0,1\}$ that annihilate all the Baire-meager sets in $\prod^\CH_{\alpha \in A} (X_\alpha)_\CH$ (this is a $\StoneCat$-space, but not necessarily a $\StoneCatSigma$-space).  By restricting these $\Bool$-morphisms to clopen sets we obtain a $\StoneCat$-morphism from $(\prod^\StoneCatSigma_{\alpha \in A} X_\alpha)_{\StoneCat}$ to $\prod^{\StoneCat}_\alpha (X_\alpha)_{\StoneCat}$, which is surjective (a $\StoneCat$-epimorphism) by the dual of Corollary \ref{bool-cat-extra}.  We leave the verification of these claims to the interested reader.
\end{remark}

We now combine the above product theory with the Loomis--Sikorski functor and the Riesz representation theorem to give a version of the Kolmogorov extension theorem in the category $\AbsMes$ of abstract measurable spaces.  Unlike the classical Kolmogorov extension theorem, no regularity properties (such as standard Borel properties) on the underlying measurable spaces are required; on the other hand, the measures constructed live in the categorical product in $\AbsMes$ rather than the categorical $\ConcMes$-product.

\begin{theorem}[Abstract Kolmogorov extension theorem]\label{kolmo}  Let $(X_\alpha)_{\alpha \in A}$ be a family of $\AbsMes$-spaces indexed by some (possibly uncountable) set $A$.  Suppose that for each finite subset $F$ of $A$, one has a probability measure $\mu_F$ on the $\AbsMes$-space $X_F \coloneqq \prod_{\alpha \in F} X_\alpha$, thus promoting this $\AbsMes$-space $X_F$ to an $\AbsProb$-space $(X_F,\mu_F)$.  Suppose furthermore that whenever $F \subseteq F' \subseteq A$ are finite, one has $(\pi_{X_{F'} \to X_F})_* \mu_{F'} = \mu_F$ where $\pi_{X_{F'} \to X_F} \colon X_{F'} \to X_F$ is the canonical $\AbsMes$-projection, thus $\pi_{X_{F'} \to X_F}$ can be promoted to an $\AbsProb$-morphism from $(X_{F'}, \mu_{F'})$ to $(X_F, \mu_F)$.  Then there exists a unique probability measure $\mu_A$ on the $\AbsMes$-space $X_A \coloneqq \prod_{\alpha \in A} X_\alpha$ such that $(\pi_{X_A \to X_F})_* \mu_A = \mu_F$ for all finite $F \subseteq A$.
\end{theorem}

\begin{proof}  We begin with existence.  By Proposition \ref{absmes-prod}, one can identify $X_F$ with $\ominus \circ \Cast_{\CHDelete \to \AbsDelete}( \tilde X_F )$ where $\tilde X_F$ is the $\CHDelete$-space $\tilde X_F \coloneqq \prod^\CHDelete_{\alpha \in F} \Loomis(X_\alpha)$, and similarly for $X_A$.  The probability measure $\mu_F$ on $X_F$ then induces a probability measure $\tilde \mu_F$ on $\tilde X_F$ which annihilates the null ideal of this $\CHDelete$-space.  For $F \subseteq F' \subseteq A$ finite, one easily checks that
$$ \pi_{\tilde X_{F'} \to \tilde X_F} \tilde \mu_{F'} = \tilde \mu_F$$
where $\pi_{\tilde X_{F'} \to \tilde X_F} \colon \tilde X_{F'} \to \tilde X_F$ is the canonical $\CHDelete$-projection.  By the Riesz representation theorem, each $\tilde \mu_F$ represents a state $\lambda_F \colon \CFunc(\tilde X_F) \to \C$ on $\tilde X_F$.  If we identify $\CFunc(\tilde X_F)$ with a subalgebra of $\CFunc(\tilde X_{F'})$ and of $\CFunc(\tilde X_A)$ for $F \subseteq F' \subseteq A$, we see that $\lambda_F$ and $\lambda_{F'}$ agree on $\CFunc(\tilde X_F)$ for all finite $F \subseteq F' \subseteq A$.  But from the Stone-Weierstrass theorem, the union of the $\CFunc(\tilde X_F)$ for $F \subseteq A$ finite is dense in $\CFunc(\tilde X_A)$.  Thus we see that the states $\lambda_F$ on $\tilde X_F$ extend to a state $\lambda_A \colon \CFunc(\tilde X_A) \to \C$ on $\tilde X_A$.  By the Riesz representation theorem (Theorem \ref{rrt}), this state is represented by a probability measure $\tilde \mu_A$ on $\tilde X_A$, and the uniqueness aspect of this theorem we have
$$ \pi_{\tilde X_{A} \to \tilde X_F} \tilde \mu_{A} = \tilde \mu_F$$
for any finite $F \subseteq A$, where $\pi_{\tilde X_{A} \to \tilde X_F} \colon \tilde X_{A} \to \tilde X_F$ is the canonical $\CHDelete$-projection.  In particular, $\tilde \mu_A$ annihilates the pullback of any null ideal of an individual factor $\tilde X_\alpha$ of $\tilde X_F$, and hence annihilates the entire null ideal.  As such, $\tilde \mu_A$ descends to a probability measure $\mu_A$ on $X_A$, which has the required properties.  This establishes existence.

For uniqueness, suppose there is another measure $\mu'_A$ on $X_A$ with the stated properties.  Then as before this induces a measure $\tilde \mu'_A$ on $\tilde X_A$ that annihilates the null ideal.  This represents a functional $\lambda'_A$ on $\tilde X_A$ that agrees with $\lambda_F$ on $\CFunc(\tilde X_F)$ for every finite $F \subseteq A$, and hence is identically equal to $\lambda_A$ by density.  From the uniqueness aspect of the Riesz representation theorem, we then have $\tilde \mu'_A = \tilde \mu_A$, hence $\mu'_A = \mu_A$, giving uniqueness.
\end{proof}

It is a classical fact \cite{andersen-jessen} that the analogue of Theorem \ref{kolmo} for $\ConcMes$ fails without additional hypotheses on the factor spaces.  However, the analogue of Theorem \ref{kolmo} for $\CH$ (using the Baire $\sigma$-algebra) follows easily from the Riesz representation theorem by a variant of the argument used to prove Theorem \ref{kolmo}.

\begin{remark}
We sketch an alternative "dual" proof of our abstract Kolmogorov theorem suggested to us by the anonymous referee.  
Suppose that $(X_\alpha)_{\alpha\in A}$ is a direct system of $\SigmaAlg$-algebras. 
Then we want to verify that any consistent family $(\mu_\alpha)_{\alpha\in A}$ of abstract probability measures in the sense as stated in Theorem \ref{kolmo}, where $\mu_\alpha$ is a probability measure on $X_\alpha$ such that $(X_\alpha,\mu_\alpha)$ promotes to an $\AbsProb$-space, extends to a unique probability measure $\mu$ on the $\SigmaAlg$-direct limit $\varinjlim_{\alpha\in A} X_\alpha$. 
As for existence, we observe that the direct limit $\varinjlim_{\alpha\in A} \Mes(X_\alpha,\mu_\alpha)$ exists in $\ProbAlg$ (see, e.g., \cite[328 H]{fremlinvol3}), and by the universal property of direct limits, there is a unique $\SigmaAlg$-morphism from $\varinjlim_{\alpha\in A} X_\alpha$ to the cast of $\varinjlim_{\alpha\in A} \Mes(X_\alpha,\mu_\alpha)$ in $\SigmaAlg$ (with respect to the forgetful functor).  
Uniqueness follows from the fact that the $\SigmaAlg$-algebra $\varinjlim_{\alpha\in A} X_\alpha$ is generated by the images of the $\SigmaAlg$-algebras $X_\alpha$.  
\end{remark} 

\appendix

\section{Review of category theory}\label{category-appendix}

In this appendix we review some concepts and notations in category theory that we will need.  
A prominent role will be played by various categorical and non-categorical notions of products and coproducts. 
The interested reader is additionally referred to standard introductory textbooks on category theory, e.g. \cite{leinster,maclane-categories}, and to \cite[Chapter 1]{heunen} for an introduction to (symmetric) monoidal categories which formalize relevant notions of Fubini type products (resp. coproducts) that are prevalent in probabilistic categories. 

\subsection{Categories and functors}\label{categories}

\begin{definition}[Category] A \emph{category} $\Cat$ is a class of objects (which we refer to as \emph{$\Cat$-objects}, \emph{$\Cat$-spaces} or \emph{$\Cat$-algebras}), together with a set $\Hom_\Cat(X \to Y)$ associated to any pair $X,Y \in \Cat$ of $\Cat$-objects, whose elements we call \emph{$\Cat$-morphisms} $f \colon X \to Y$ from the domain $X$ to the codomain $Y$.  The category $\Cat$ is equipped with a composition operation $\circ \colon \Hom_\Cat(Y \to Z) \times \Hom_\Cat(X \to Y) \to \Hom_\Cat(X \to Z)$ for any three $\Cat$-objects $X,Y,Z$ which is associative in the sense that
$$ (f \circ g) \circ h = f \circ (g \circ h)$$
whenever $f \colon Z \to W$, $g \colon Y \to Z$, $h \colon X \to Y$ are $\Cat$-morphisms.  We also assume that to every $\Cat$-object $X \in \Cat$ there is an identity $\Cat$-morphism $\id_X \colon X \to X$ such that 
$$ f = f \circ \id_{X} = \id_{Y} \circ f$$
for every $\Cat$-morphism $f \colon X \to Y$ from one $\Cat$-object $X$ to another $Y$. 
\end{definition}

As a general convention, when the ambient category $\Cat$ is clear from context, we will drop the prefix $\Cat$-, for instance $\Cat$-morphisms will also be referred to a ``morphism in $\Cat$'', or simply a ``morphism'' if it is clear which category one is working in.

We now give two fundamental examples of categories: the categories of sets and groups.

\begin{example}[The category $\Set$]\label{set-def} A \emph{$\Set$-object} (or \emph{$\Set$-space}) is a set $X$.  A \emph{$\Set$-morphism} is a function $f \colon X \to Y$ between two sets.  Composition of two $\Set$-morphisms $f \colon X \to Y$, $g \colon Y \to Z$ is given by the usual composition law $(g \circ f)(x) \coloneqq g(f(x))$ for $x \in X$.
\end{example}

\begin{example}[The category $\Group$]\label{grp-def} A \emph{$\Group$-object} is a group $G = (G_\Set,\cdot)$.  A \emph{$\Group$-morphism} $f \colon G \to H$ is a group homomorphism $f_\Set \colon G_\Set \to H_\Set$ between the underlying sets.  Composition of two $\Group$-morphisms is given by the $\Set$-composition law.
\end{example}

One can take a category $\Cat$ and ``reverse all its arrows'' to obtain a new category $\Cat^\op$:

\begin{definition}[Opposite category]\label{opp-cat} Let $\Cat$ be a category.  We define the \emph{opposite category} $\Cat^\op$ as follows.
\begin{itemize}
    \item[(i)] A $\Cat^\op$-object is the same as a $\Cat$-object.
    \item[(ii)] A $\Cat^\op$-morphism $f \colon X \to Y$ of two $\Cat^\op$-objects $X, Y$ is a $\Cat$-morphism $f \colon Y \to X$. 
    \item[(iii)]  The composition $g \circ f$ of two $\Cat^\op$-morphisms $f$ and $g$ is defined by the composition of $\Cat$-morphisms $f \circ g$,  
    see Figure \ref{fig:op}.
\end{itemize}
By abuse of notation we identify $(\Cat^\op)^\op$ with $\Cat$ in the obvious fashion.
\end{definition}

\begin{figure}
    \centering
    \begin{tikzcd}
    X \arrow[r,"g"]  & Z & X \arrow[d, "f"'] & Z \arrow[l, "g"'] \arrow[dl, "f \circ g"] \\
    Y \arrow[u, "f"] \arrow[ur, "g \circ f"'] & & Y 
    \end{tikzcd}
    \caption{A diagram in $\Cat$ on the left, and its counterpart in $\Cat^\op$ on the right. Note the reversed direction of all the arrows.}
    \label{fig:op}
\end{figure}
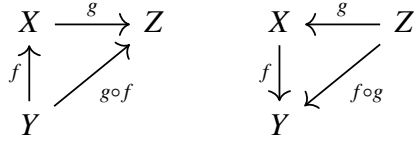

\begin{definition}[Slice category]\label{def-slice}
Let $\Cat$ be a category. 
We denote by $\Cat\downarrow X$ the \emph{slice category} over an object $X\in \Cat$, where $(\Cat\downarrow X)$-objects are $\Cat$-morphisms $Y\to X$, i.e., $\Cat$-morphisms whose codomain is $X$, and $(\Cat\downarrow X)$-morphisms are $\Cat$-morphisms $g\colon Y\to Y'$ from a $(\Cat\downarrow X)$-object $f\colon Y\to X$ to another $(\Cat\downarrow X)$-object $f'\colon Y'\to X$ such that the identity $f'\circ g=f$ holds in $\Cat$. 
See Figure \ref{fig:slicecat}, for a diagram of how composition is defined in $\Cat\downarrow X$. 

Dually, we can define the \emph{coslice category} $X\downarrow\Cat$ with respect to an $X\in \Cat$, whose objects are $\Cat$-morphisms $X\to Y$, i.e., $\Cat$-morphisms whose domain is $X$, and whose morphisms are $\Cat$-morphisms $g\colon Y'\to Y$ from a $(X\downarrow\Cat)$-object $f'\colon X\to Y'$ to another $(X\downarrow\Cat)$-object $f\colon X\to Y$ such that the identity $g\circ f'=f$ holds in $\Cat$. 
\end{definition}

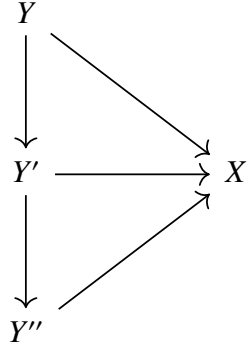
\begin{figure}
    \centering
    \begin{tikzcd}
	Y \\
	\\
	{Y'} && X \\
	\\
	{Y''}
	\arrow[from=1-1, to=3-3]
	\arrow[from=3-1, to=3-3]
	\arrow[from=5-1, to=3-3]
	\arrow[from=1-1, to=3-1]
	\arrow[from=3-1, to=5-1]
\end{tikzcd}
    \caption{The composition of morphisms in a slice category.}
    \label{fig:slicecat}
\end{figure}

We isolate some special types of morphisms and objects:

\begin{definition}[Special morphisms and objects]\label{special-morph}  Let $\Cat$ be a category.
\begin{itemize}
    \item [(i)]  A $\Cat$-morphism $\pi \colon X \to Y$ is a \emph{$\Cat$-epimorphism} if whenever $f, f' : Y \to Z$ are $\Cat$-morphisms with $f \circ \pi = f' \circ \pi$, one has $f = f'$.
    \item [(ii)]  Dually, a $\Cat$-morphism $\iota \colon Y \to X$ is a \emph{$\Cat$-monomorphism} if whenever $f, f' \colon Z \to Y$ are $\Cat$-morphisms with $\iota \circ f = \iota \circ f'$, one has $f = f'$.
    \item[(iii)]  A $\Cat$-\emph{bimorphism} $\phi$ is a $\Cat$-epimorphism that is also a $\Cat$- monomorphism.
    \item[(iv)]  A $\Cat$-morphism $\phi \colon X \to Y$ is a \emph{$\Cat$-isomorphism} if there is an inverse $\Cat$-morphism $\phi^{-1} \colon Y \to X$ such that $\id_X =  \phi^{-1}\circ \phi$ and $\id_Y = \phi \circ \phi^{-1}$.
    \item[(v)]  A $\Cat$-morphism $\phi \colon X \to X$ is a \emph{$\Cat$-endomorphism} if the domain and codomain are the same object, and a \emph{$\Cat$-automorphism} if it is a $\Cat$-endomorphism and a $\Cat$-isomorphism.  
        \item[(vi)]  A $\Cat$-object $X$ is \emph{terminal} (resp. \emph{initial}) if for every $\Cat$-object $Y$ there is a unique $\Cat$-morphism from $Y$ to $X$ (resp. from $X$ to $Y$).  
\end{itemize}
See also Figure \ref{fig:mono}.
\end{definition}

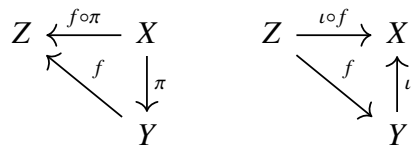
\begin{figure}
    \centering
    \begin{tikzcd}
    Z  & X \arrow[l, "f \circ \pi"'] \arrow[d, "\pi"] & Z \arrow[r, "\iota \circ f"] \arrow[dr, "f"] & X \\
    & Y \arrow[ul, "f"'] & & Y \arrow[u, "\iota"']
    \end{tikzcd}
    \caption{If $\pi$ is an epimorphism, then $f$ is uniquely determined by $f \circ \pi$.  If $\iota$ is a monomorphism, then $f$ is uniquely determined by $\iota \circ f$.}
    \label{fig:mono}
\end{figure}

Clearly the composition of two $\Cat$-monomorphisms is again a $\Cat$-monomorphism, and similarly for $\Cat$-epimorphisms. Every $\Cat$-isomorphism is a $\Cat$-bimorphism.  The converse is true for some of the categories we will study here (e.g., $\Set$, $\Bool$, $\Bool_\sigma$, $\CH$, $\StoneCat$, $\AbsMes$), but not all (for instance, the inclusion map from $(0,1)$ to $[0,1]$ is a $\Polish$-bimorphism but not a $\Polish$-isomorphism).  

\begin{example}\label{set-example}  It is easily verified that a function $X \to Y$ between two sets $X,Y \in \Set$ (i.e., a $\Set$-morphism) is a $\Set$-monomorphism if and only if it is injective, a $\Set$-epimorphism if and only if it is surjective, and a $\Set$-isomorphism (or $\Set$-bimorphism) if and only if it is bijective.  The analogous claims for the category $\Group$ are also true, but not as easy to demonstrate; the difficult step is to show that for any proper subgroup $H$ of $G$ there is a group homomorphism $f \colon H \to K$ into a third group $K$ that admits more than one extension to a group homomorphism of $G$. A canonical choice of such a $K$ is provided by the amalgamated free product $G \ast_H G$ (which, in the categorical language used in this paper, is the categorical colimit of the diagram $G \leftarrow H \rightarrow G$) .  On the other hand, not all $\ConcMes$-epimorphisms are surjective; for instance, the inclusion of $\{1\}$ into $\{1,2\}$, where we endow $\{1,2\}$ with the trivial $\sigma$-algebra $\{ \emptyset, \{1,2\}\}$, is a $\ConcMes$-epimorphism which is not surjective.  (The existence of non-surjective $\ConcMes$-epimorphisms causes difficulty when trying to represent abstract measurable maps by concrete ones; see \cite[\S 5]{jt19} for further discussion.)
\end{example}

A pair of categories $\Cat, \Cat'$ can be related to each other by functors\footnote{We will work exclusively with covariant functors in this paper.}. 

\begin{definition}[Functor]  Let $\Cat,\Cat'$ be categories.
 A \emph{functor} $$\Func = \Func_{\Cat \to \Cat'} \colon \Cat \to \Cat'$$ assigns to each $\Cat$-object $X$ a $\Cat'$-object $\Func(X)$, and to each $\Cat$-morphism $f = f_{X \to Y}$ a $\Cat$'-morphism $\Func(f) = \Func(f_{X \to Y}) = \Func(f)_{\Func(X) \to \Func(Y)}$ such that 
    $$ \Func( f_{Y \to Z} \circ g_{X \to Y} ) = \Func(f)_{\Func(Y) \to \Func(Z)} \circ \Func(g)_{\Func(X) \to \Func(Y)}$$
    for any $\Cat$-functors $f_{Y \to Z}, g_{X \to Y}$ between the $\Cat$-objects $X,Y,Z$. 
\end{definition}

\begin{example}[Identity functor]  For every category $\Cat$ there is the identity functor $\ident_{\Cat} \colon \Cat \to \Cat$ that acts trivially on the objects and morphisms of the category. 
\end{example}

\begin{example}[Forgetful functors]  For every unlabeled arrow in the diagrams of categories in this paper between two functors $\Cat$, $\Cat'$, there is an obvious  \emph{forgetful functor} $\Forget_{\Cat \to \Cat'}$, which is a   functor that takes any $\Cat$-object $X$ and ``forgets'' some structure on it to produce a $\Cat'$-object (which by abuse of notation we often also call $X$), and usually leaves the $\Cat$-morphisms unchanged (but now interpreted as $\Cat'$-morphisms).  For instance, there is a forgetful functor $\Forget_{\Group \to \Set}$ formed by taking a group $K = (K,\cdot)$ and forgetting the group structure, to only retain the underlying set $K$.  We consider the composition of two or more forgetful functors to again be a forgetful functor, thus for instance $\Forget_{\CH \to \Set}$ is the forgetful functor $\Forget_{\ConcMes \to \Set} \circ \Forget_{\CH\to \ConcMes}$.  In most cases (particularly when the forgetful functor is deemed to be a casting functor, see Definition \ref{cast}) we will not need to explicitly refer to such functors by name. 
\end{example}

One can compose functors together in the obvious fashion to obtain further functors.  
We record some special types of functors:

\begin{definition}[Special functors]\label{special-func}  Let $\Cat, \Cat'$ be categories.
\begin{itemize}
    \item [(i)]  A functor $\Func \colon \Cat \to \Cat'$ is \emph{faithful} (resp. \emph{full}) if for any two $\Cat$-objects $X,Y$, the map $$\Func \colon \Hom_\Cat(X \to Y) \to \Hom_{\Cat'}(\Func(X) \to \Func(Y))$$ is injective (resp. surjective). In the diagram of categories in this paper, we use arrows with tails $\rightarrowtail$ between categories to indicate faithful functors (resp. arrows with two heads in one direction $\twoheadrightarrow$  to indicate full functors). We use arrows with tails and two heads $\taildoubleheadrightarrow$ between categories to indicate a functor which is both faithful and full. 
    
    When $\Func \colon \Cat \to \Cat'$ is a faithful functor, then we also call $\Cat$ a \emph{subcategory} of $\Cat'$. If the faithful functor is full, we call $\Cat$ a \emph{full subcategory} of $\Cat'$. 
    \item[(iii)]  A   functor $\Func \colon \Cat \to \Cat'$ is \emph{invertible}\footnote{It would strictly speaking be more natural from a category theory perspective to work with equivalence of categories here rather than invertible functors.} if it has an inverse $\Func^{-1} \colon \Cat' \to \Cat$ that is also a functor.  Invertible functors are indicated by arrows with heads in both directions $\doubleheadrightleftarrow$. 
\end{itemize}
\end{definition}

\begin{example}  The forgetful functor from $\SigmaAlg$ to $\Bool$ is by definition faithful.  In fact, all the forgetful functors we use in this paper are faithful.  On the other hand, the abstraction functor $\Abs \colon \ConcMes \to \AbsMes$ is not even faithful, even if it arguably deserves to be classified as a forgetful functor; for instance, if $X$ is a $\ConcMes$-space with the trivial $\sigma$-algebra, then any permutation on $X$ is a $\ConcMes$-morphism that becomes the identity $\AbsMes$-morphism when applying $\Abs$.  
\end{example}

We point out that the following example is not usually considered in category theory. 

\begin{example}[Range of a functor]\label{range-ex}  If $\Func \colon \Cat \to \Cat'$ is a functor that is injective on objects, then we can define the category $\Func(\Cat)$ to be the category whose  $\Func(\Cat)$-objects are of the form $\Func(X)$ for some $\Cat$-object $X$, and whose $\Func(\Cat)$-morphisms are of the form $\Func(f)$ for some $\Cat$-morphism $f$, with the obvious composition law.  This is then a subcategory of $\Cat'$ with the obvious faithful functor.  
\end{example}

The following lemma is trivial but useful:

\begin{lemma}[Faithful functors, epimorphisms, and monomorphisms]\label{epimorph} 
If a functor $\Func \colon \Cat \rightarrowtail \Cat'$ is faithful and $f$ is a $\Cat$-morphism with $\Func(f)$ a $\Cat'$-monomorphism (resp. $\Cat'$-epimorphism), then $f$ is also a $\Cat$-monomorphism (resp. $\Cat$-epimorphism).  In particular, if $\Cat$ is a concrete category (so that there is a faithful forgetful functor to $\Set$), every injective $\Cat$-morphism is monomorphic, and every surjective $\Cat$-morphism is epimorphic. 
\end{lemma}

\subsection{Natural transformations}

We now recall the notion of a natural transformation between two functors. This notion helps us capture what it means for a given construction (such as a categorical or monoidal product or coproduct) in a category to be ``functorial'', and what it means for one such construction to be ``contained in'' another, even when the underlying category is abstract rather than concrete.  It also makes precise the (often vaguely defined) concept of what it means for a certain morphism to be ``canonical''.

\begin{definition}[Natural transformation]  Let $\Func, \Func' \colon \Cat \to \Cat'$ be two   functors between categories $\Cat, \Cat'$.  A \emph{natural transformation} $\Nat \colon \Func \to \Func'$ from $\Func$ to $\Func'$ is an assignment of a $\Cat'$-morphism $\Nat(X)$ to each $\Cat$-object $X$ such that the diagram
\begin{center}
    \begin{tikzcd}
    \Func(X) \arrow[r, "\Nat(X)"] \arrow[d, "\Func(f)"] & \Func'(X) \arrow[d, "\Func'(f)"] \\
    \Func(Y) \arrow[r, "\Nat(Y)"] & \Func'(Y)
\end{tikzcd}
\end{center}
commutes for every $\Cat$-morphism $f\colon X \to Y$.  We say that $\Nat$ is a \emph{natural isomorphism} (resp. \emph{natural monomorphism}, \emph{natural epimorphism}) if $\Nat(X)$ is a $\Cat'$-isomorphism (resp. $\Cat'$-monomorphism, $\Cat'$-epimorphism) for every $\Cat$-object $X$.

 An \emph{equivalence of categories}  between two categories $\Cat, \Cat'$ is a pair of functors $\Func \colon \Cat \to \Cat', \Func' \colon \Cat' \to \Cat$ such that $\Func' \circ \Func$ is naturally isomorphic to $\ident_\Cat$ and $\Func \circ \Func'$ is naturally isomorphic to $\ident_{\Cat'}$.  
 A \emph{duality of categories} between two categories $\Cat, \Cat'$ is a pair of functors $\Func \colon \Cat^\op \to \Cat', \Func' \colon \Cat' \to \Cat^\op$ such that $\Func' \circ \Func$ is naturally isomorphic to $\ident_{\Cat^\op}$ and $\Func \circ \Func'$ is naturally isomorphic to $\ident_{\Cat'}$.

We will refer to a \emph{canonical} $\Cat$-map (resp. $\Cat$-monomorphism, $\Cat$-epimorphism, $\Cat$-isomorphism) between two $\Cat$-objects $X, Y$ to be the morphism given by the ``obvious'' natural transformation (resp. natural monomorphism, epimorphism, isomorphism) that can relate the two objects, in those cases where the ``obvious'' choice of natural transfomation is clear from context.
\end{definition}

\begin{example}\label{nat-inv}  If $\Nat \colon \Func_1 \to \Func_2$ is a natural isomorphism then so is its inverse $\Nat^{-1} \colon \Func_2 \to \Func_1$, defined in the obvious fashion.
\end{example} 

\begin{example} The identity functor establishes a duality of categories between an arbitrary category $\Cat$ and its opposite category $\Cat^\op$.  Further examples of dualities of categories are given in Figure \ref{fig:duality}.
\end{example}

\begin{figure}
    \centering
    \begin{tikzcd}
      \CStarAlgUnitInfop \arrow[d,tail,two heads,"\Spec",shift left=.75ex] & \CStarAlgUnitop \arrow[d,tail,"\Spec",two heads,shift left=.75ex]  & \CStarAlgNdop \arrow[d,tail,two heads,"\Spec",shift left=.75ex] & \CStarAlgMultop \arrow[d,tail,"\Spec",two heads,shift left=.75ex] &  \\ 
    \CHpt \arrow[u,"\CFunc", two heads,  tail, shift left=.75ex,blue]  & 
     \CH \arrow[u,"\CFunc", tail, two heads, shift left=.75ex,blue]  & 
    \LCHpr  \arrow[u, "\CoFunc", tail,two heads,shift left=.75ex,blue] & \LCH \arrow[u, "\CoFunc", tail, two heads,shift left=.75ex,blue] \\
    \CStarAlgUnitInfTraceop \arrow[d,  tail,two heads,"\Riesz",shift left=.75ex] & \CStarAlgUnitTraceop \arrow[d,  tail,"\Riesz",two heads,shift left=.75ex]  & \CStarAlgNdTraceop \arrow[d,  tail,two heads,"\Riesz",shift left=.75ex] & \CStarAlgMultTraceop \arrow[d,  tail,"\Riesz",two heads,shift left=.75ex] &  \\ 
    \CHptProb \arrow[u,"\CFunc", two heads,   tail, shift left=.75ex,blue]  & 
     \CHProb \arrow[u, "\CFunc", tail,two heads,shift left=.75ex,blue]  & 
    \LCHprProb  \arrow[u,  "\CoFunc", tail,two heads,shift left=.75ex,blue] & \LCHProb \arrow[u,  "\CoFunc", tail, two heads,shift left=.75ex,blue] \\
    \vonNeumannop \arrow[d, tail,two heads,"\Idem",shift left=.75ex]  \\ 
    \ProbAlg \arrow[u,"\Linfty", two heads,   tail, blue, shift left=.75ex]  \\
    \Boolop \arrow[d,  tail,two heads,"\StoneFunc",shift left=.75ex] & \AbsMes \arrow[d,  tail,"\StoneFuncSigma",two heads,shift left=.75ex]  \\ 
    \StoneCat \arrow[u,"\Clopen", two heads,   tail, shift left=.75ex]  & 
     \StoneCatSigma \arrow[u, "\ClopenSigma", tail,two heads,shift left=.75ex]  
\end{tikzcd}
    \caption{The dualities of categories that appear in this paper.  The rows correspond to Gelfand dualities, Riesz dualities, probability dualities, and Stone dualities respectively. Various additional functors between these categories have been omitted for clarity.}
    \label{fig:duality}
\end{figure}
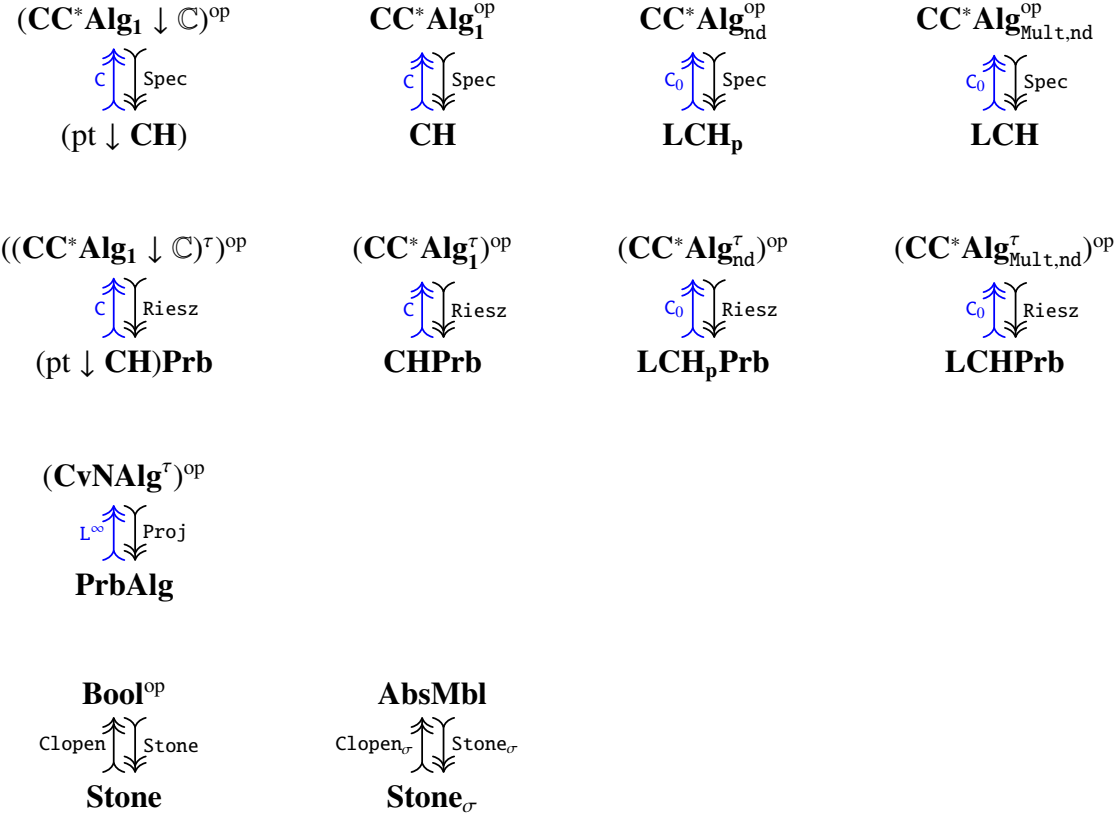

\subsection{Categorical products, inverse limits, and tensor products}\label{sec-product}

In this section, we recall the concepts of categorical products and coproducts, inverse and direct limits, and symmetric monoidal categories (the latter category-theoretically formalizes a general notion of tensor products). 
We then discuss how to combine and category-theoretically compare these concepts in order to apply them to relate various product constructions for the topological, probabilistic and functional analytic objects introduced in this paper. 

\begin{definition}[Categorical products and coproducts]
Let $\Cat$ be a category. 
\begin{itemize}
    \item[(i)] A \emph{categorical product} of a family $X_\alpha, \alpha \in A$ of objects in $\Cat$ is an object $X\in \Cat$ such that there is an indexed family  $(\pi_\alpha)_{\alpha \in A}$ of $\Cat$-morphisms $\pi_\alpha \colon X \to X_\alpha$ satisfying the following universal property: If $Y$ is another object in $\Cat$ and $(f_\alpha)_{\alpha\in A}$ is another indexed family of $\Cat$-morphisms $f_\alpha\colon Y\to X_\alpha$, then there exists a unique $\Cat$-morphism $\phi\colon Y\to X$ such that $\pi_\alpha\circ \phi=f_\alpha$ for each $\alpha\in A$. Note that if a categorical product of the $X_\alpha$ exists, then it is unique up to $\Cat$-isomorphisms.  In this case, we denote the categorical product by $\prod_{\alpha \in A} X_\alpha=\prod_{\alpha \in A}^\Cat X_\alpha$. 
    \item[(ii)]  Dually, a \emph{categorical coproduct} of a family $X_\alpha, \alpha \in A$ of objects in $\Cat$ is an object $X\in \Cat$ such that there is an indexed family  $(\iota_\alpha)_{\alpha \in A}$ of $\Cat$-morphisms $\iota_\alpha \colon X_\alpha \to X$ satisfying the following universal property: If $Y$ is another object in $\Cat$ and $(f_\alpha)_{\alpha\in A}$ is another indexed family of $\Cat$-morphisms $f_\alpha\colon X_\alpha\to Y$, then there exists a unique $\Cat$-morphism $\phi\colon X\to Y$ such that $\phi \circ \iota_\alpha=f_\alpha$ for each $\alpha\in A$. Note that if a categorical coproduct of the $X_\alpha$ exists, then it is unique up to $\Cat$-isomorphisms.  In this case, we denote the categorical coproduct by $\coprod_{\alpha \in A} X_\alpha=\coprod_{\alpha \in A}^\Cat X_\alpha$. 
\end{itemize}
For products (resp. coproducts) of two objects in $\Cat$ we use $X_1 \times X_2=X_1 \times^{\Cat} X_2$ (resp.~$X_1 \sqcup X_2=X_1 \sqcup^\Cat X_2$) as shorthand for $\prod_{\alpha \in \{1,2\}}^\Cat X_\alpha$ (resp. $\coprod_{\alpha \in \{1,2\}}^\Cat X_\alpha$).  See Figure \ref{fig:univ}.
\end{definition}

\begin{figure}
    \centering
    \begin{tikzcd}
    Y \arrow[ddr, bend right, "f_1"'] \arrow[dr, dotted, "\phi"] \arrow[drr, bend left, "f_2"] \\
    & X_1 \times X_2 \arrow[r,"\pi_2"'] \arrow[d,"\pi_1"] & X_2 \\
    & X_1
    \end{tikzcd}
    \caption{A diagram of the universal property of the categorical product of a pair $X_1,X_2$ of objects in a category $\Cat$.  A similar diagram exists for the coproduct $X_1 \sqcup X_2$ (after reversing all the arrows).}
    \label{fig:univ}
\end{figure}
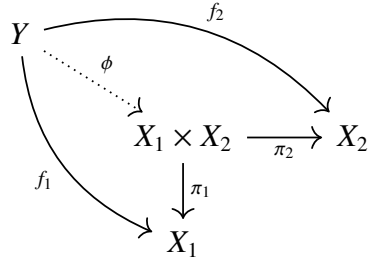

\begin{example}\label{coterm}    Given a family $X_\alpha, \alpha \in A$ of sets $X_\alpha$, the Cartesian product $\prod_{\alpha \in A} X_\alpha$ is a categorical product of the $X_\alpha$ in $\Set$, and the disjoint union $\biguplus_{\alpha \in A} X_\alpha$ is a categorical coproduct in $\Set$.  
\end{example}

\begin{remark}\label{hom-prod}  Let $(X_\alpha)_{\alpha \in A}$ be a family of objects in some category $\Cat$.  If a categorical product $\prod_{\alpha \in A} X_\alpha$ exists, then (after making some obvious canonical identifications) one has the identity
$$ \Hom_\Cat\left( Y \to \prod^\Cat_{\alpha \in A} X_\alpha \right) = \prod^\Set_{\alpha \in A} \Hom_\Cat( Y \to X_\alpha )$$
for any $\Cat$-object $Y$; indeed this can be viewed as an alternate definition of a categorical product in $\Cat$.  Similarly, if a categorical coproduct $\coprod^\Cat_{\alpha \in A} X_\alpha$ exists, then one has the identity
$$ \Hom_\Cat\left( \coprod^\Cat_{\alpha \in A} X_\alpha \to Y \right) = \coprod^\Set_{\alpha \in A} \Hom_\Cat( X_\alpha \to Y ),$$
after making the obvious canonical identifications.
\end{remark}

\begin{example}\label{exp-groups}
A categorical product $\prod_{\alpha \in A} K_\alpha$ of groups $K_\alpha$ can be constructed by taking the Cartesian product $\prod^\Set_{\alpha \in A} K_\alpha$ and endowing it with the group operation in the obvious fashion.  A categorical coproduct $\coprod^\Group_{\alpha \in A} K_\alpha$ can be formed by the free product construction.
\end{example}

\begin{definition}[Inverse and direct limits]\label{def-limits}
Let $\Cat$ be a category. 
\begin{itemize}
    \item[(i)] Let $(X_\alpha)_{\alpha\in A}$ be a directed\footnote{A partially ordered set $(A,\leq)$ is said to be \emph{directed} if for each pair $\alpha,\beta\in A$ there is $\gamma\in A$ such that $\alpha\leq \gamma$ and $\beta\leq \gamma$.} family of objects $X_\alpha\in \Cat$ such that there is a family of $\Cat$-morphisms $f_{\alpha,\beta}\colon X_\beta\to X_\alpha$ for all $\alpha\leq \beta$ satisfying 
    \begin{itemize}
        \item $f_{\alpha,\alpha}=\id_{X_\alpha}$ for all $\alpha$;
        \item $f_{\alpha,\gamma}=f_{\alpha,\beta}\circ f_{\beta,\gamma}$ for all $\alpha\leq \beta\leq \gamma$. 
    \end{itemize}
    Then we call the tuple $(X_\alpha,f_{\alpha,\beta})$ an \emph{inverse system} of objects and morphisms in $\Cat$. 
    A $\Cat$-object $X$ is called an \emph{inverse limit} of the inverse system $(X_\alpha,f_{\alpha,\beta})$ if there exists a family of $\Cat$-morphisms $\pi_\alpha \colon X\to X_\alpha$ for each $\alpha$ satisfying $\pi_\alpha=f_{\alpha,\beta}\circ \pi_\beta$ for all $\alpha\leq \beta$. The pair $(X,\pi_\alpha)$ must satisfy the following universality property. If $(Y,\psi_\alpha)$ is another pair of a $\Cat$-object $Y$ and $\Cat$-morphisms  $\psi_\alpha \colon X\to X_\alpha$ for all $\alpha$ such that $\psi_\alpha=f_{\alpha,\beta}\circ \psi_\beta$ for all $\alpha\leq \beta$, then there must exist a unique $g\colon Y\to X$ such that the diagram
\[\begin{tikzcd}
	&& Y \\
	\\
	&& X \\
	\\
	{X_\beta} &&&& {X_\alpha}
	\arrow["{f_{\alpha,\beta}}", from=5-1, to=5-5]
	\arrow["{\psi_\beta}"', from=1-3, to=5-1]
	\arrow["{\psi_\alpha}", from=1-3, to=5-5]
	\arrow["g"', dashed, from=1-3, to=3-3]
	\arrow["{\pi_\beta}", from=3-3, to=5-1]
	\arrow["{\pi_\alpha}"', from=3-3, to=5-5]
\end{tikzcd}\]
    commutes for all $\alpha\leq \beta$. 
    If the inverse system $(X_\alpha,f_{\alpha,\beta})$ possesses an inverse limit $X$ in $\Cat$, then we denote this inverse limit by $X=\varprojlim_{\alpha\in A} X_\alpha$.  By the universality property, if an inverse limit exists, then its is unique up to $\Cat$-isomorphisms. 
    \item[(ii)] Let $(X_\alpha)_{\alpha\in A}$ be a directed family of objects $X_\alpha\in \Cat$ such that there is a family of $\Cat$-morphisms $f_{\alpha,\beta}\colon X_\alpha\to X_\beta$ for all $\alpha\leq \beta$ satisfying 
    \begin{itemize}
        \item $f_{\alpha,\alpha}=\id_{X_\alpha}$ for all $\alpha$;
        \item $f_{\alpha,\gamma}=f_{\beta,\gamma}\circ f_{\alpha,\beta}$ for all $\alpha\leq \beta\leq \gamma$. 
    \end{itemize}
    Then we call the tuple $(X_\alpha,f_{\alpha,\beta})$ a \emph{direct system} of objects and morphisms in $\Cat$. 
    A $\Cat$-object $X$ is called an \emph{direct limit} of the direct system $(X_\alpha,f_{\alpha,\beta})$ if there exists a family of $\Cat$-morphisms $\iota_\alpha \colon X_\alpha\to X$ for each $\alpha$ satisfying $\iota_\alpha= \iota_\beta\circ f_{\alpha,\beta}$ for all $\alpha\leq \beta$. The pair $(X,\iota_\alpha)$ must satisfy the following universality property. If $(Y,\lambda_\alpha)$ is another pair of a $\Cat$-object $Y$ and $\Cat$-morphisms  $\lambda_\alpha \colon X\to X_\alpha$ for all $\alpha$ such that $\lambda_\alpha=\lambda_\beta\circ f_{\alpha,\beta}$ for all $\alpha\leq \beta$, then there must exist a unique $h\colon X\to Y$ such that the diagram
\[\begin{tikzcd}
	{X_\alpha} &&&& {X_\beta} \\
	&& X \\
	\\
	&& Y
	\arrow["{\iota_\alpha}"', from=1-1, to=2-3]
	\arrow["{\iota_\beta}", from=1-5, to=2-3]
	\arrow["{f_{\alpha,\beta}}", from=1-1, to=1-5]
	\arrow["{\lambda_\alpha}"', from=1-1, to=4-3]
	\arrow["{\lambda_\beta}", from=1-5, to=4-3]
	\arrow["h", dashed, from=4-3, to=2-3]
\end{tikzcd}\]
    commutes for all $\alpha\leq \beta$. 
    If the direct system $(X_\alpha,f_{\alpha,\beta})$ possesses a direct limit $X$ in $\Cat$, then we denote this direct limit by $X=\varinjlim_{\alpha\in A} X_\alpha$. By the universality property, if a direct limit exists, then its is unique up to $\Cat$-isomorphisms. 
\end{itemize}
\end{definition}

\begin{example}
The inverse limit of a $\Set$-inverse system $(X_\alpha,f_{\alpha,\beta})$ is a particular subset of the Cartesian product $\prod_\alpha^\Set X_\alpha$ determined by the family of $\Set$-morphisms $f_{\alpha,\beta}$. 
On the other hand, the direct limit of a $\Set$-direct system $(X_\alpha,f_{\alpha,\beta})$ is a particular quotient of the disjoint union $\coprod_\alpha^\Set X_\alpha$ determined by the family of $\Set$-morphisms $f_{\alpha,\beta}$. 
We have similar constructions in the category $\Group$ of groups (cf.~Example \ref{exp-groups}). 
\end{example}

There is a link between inverse and direct limits: 

\begin{remark}\label{hom-invers}  Let $(X_\alpha, f_{\alpha,\beta})$ be an inverse system of objects and morphisms in some category $\Cat$.  If the inverse limit $\varprojlim_\alpha X_\alpha$ exists, then (after making some obvious canonical identifications) one has the identity
$$ \Hom_\Cat\left( Y \to \varprojlim_\alpha X_\alpha \right) = \varprojlim_\alpha \Hom_\Cat( Y \to X_\alpha )$$
for any $\Cat$-object $Y$; indeed this can be viewed as an alternate definition of inverse limits.  Similarly, if $(X_\alpha, f_{\alpha,\beta})$ is a direct system of objects and morphisms in some category $\Cat$ and its direct limit $\varinjlim_\alpha X_\alpha$ exists, then 
$$ \Hom_\Cat\left( \varinjlim_\alpha X_\alpha \to Y \right) = \varprojlim_\alpha \Hom_\Cat( X_\alpha \to Y ),$$
after making the obvious canonical identifications.
\end{remark}

In measure theory, there are natural product constructions which are not categorical as the following example discusses. 

\begin{example}\label{no-univ-prob}   Let $X = (X,\X,\mu)$ and $Y = (Y,\Y,\nu)$ be $\ConcProb$-spaces (i.e., concrete probability spaces), as defined in Definition \ref{top-prob-cat}.  Then the usual probability space product (also known as the \emph{Fubini product}),
$$ X \times^\ConcProb Y = (X \times Y, \X \otimes \Y, \mu \times \nu)$$
will almost never be categorical.  For instance, if $X=Y=[0,1]$ with Lebesgue measure, then the diagonal set $[0,1]^\Delta \coloneqq \{ (x,x): x \in [0,1]\} \subseteq X \times Y$ equipped with Lebesgue probability measure is another product of $X$ and $Y$ (it projects via $\ConcProb$-morphisms to both $X$ and $Y$), but has no $\ConcProb$-morphism to $X \times^\ConcProb Y$.  Indeed, categorical products almost never exist in $\ConcProb$, because of the non-uniqueness of joinings. In fact, the area of optimal transport would be completely trivial if there existed categorical products in $\ConcProb$! 
\end{example}

The lack of categorical products in probabilistic categories is dually reflected in the lack of categorical coproducts in corresponding categories of tracial commutative $C^*$- and von Neumann algebras.
However in both cases there are natural notions of products resp. coproducts. In the probabilistic categories, these are the previously mentioned Fubini products (a similar construction is available in the category of probability algebras, see Remark \ref{prob-prod}), and in the algebraic categories, we have the dual notion of tensor products. 
The formalism of \emph{(symmetric) monoidal categories} allows to capture these non-categorical product (resp. coproduct) constructions, and \emph{monoidal functors} help to relate them. 

\begin{definition}[Symmetric monoidal categories and functors]\label{def-monoidal}
Let $\Cat$ be a category. 
\begin{itemize}
    \item[(i)] A \emph{symmetric monoidal structure} on a category $\mathcal{C}$ is defined by the following data: 
\begin{itemize}
    \item[(1)] A bifunctor $\otimes\colon \mathcal{C}\times \mathcal{C}\to \mathcal{C}$ called the \emph{tensor product};
    \item[(2)] an object $\mathtt{I}$ called the \emph{identity object} or \emph{unitor};
    \item[(3)] and four natural isomorphisms, called the \emph{structure isomorphisms}, satisfying the following coherence conditions:
    \begin{itemize}
        \item[(a)] the \emph{associator} $\alpha$ with components $$\alpha_{X,Y,Z}\colon (X\otimes Y)\otimes Z \to X\otimes (Y\otimes Z)$$ satisfying the pentagon identity which can be expressed via the commutative diagram
    \begin{center}
    \begin{tikzcd}
	{} & {(X\otimes Y)\otimes (W\otimes Z)} & {}  \\
	 {X\otimes (Y\otimes (W\otimes Z))}  & {} & {((X\otimes Y) \otimes W) \otimes Z}  	\\
	 {X\otimes ((Y\otimes W)\otimes Z)} & {} &  {(X\otimes (Y\otimes W))\otimes Z} \\
	\arrow["{\alpha_{X,Y,W\otimes Z}}", from=2-1, to=1-2]
	\arrow["{\alpha_{X\otimes Y,W,Z}}", from=1-2, to=2-3]
	\arrow["{\alpha_{X,Y,W}\otimes \id_Z}"', from=3-3, to=2-3]
	\arrow["{\id_X\otimes \alpha_{Y,W,Z}}"', from=2-1, to=3-1]
	\arrow["{\alpha_{X,Y\otimes W,Z}}"', from=3-1, to=3-3]
    \end{tikzcd}
    \end{center}
        \item[(b)] the \emph{left unitor} $\lambda$ with components $$\lambda_X\colon \mathtt{I}\otimes X\to X$$ and the \emph{right unitor} $\rho$ with components $$\rho_X\colon X\otimes \mathtt{I}\to X$$ satisfying the triangle identity which can be expressed via the commutative diagram
        \begin{center}
              \begin{tikzcd}
	{A\otimes (\mathtt{I}\otimes B)} && {(A\otimes \mathtt{I})\otimes B} \\
	\\
	& {A\otimes B}
	\arrow["{\id_X\otimes \lambda_Y}"', from=1-1, to=3-2]
	\arrow["{\rho_X\otimes \id_Y}", from=1-3, to=3-2]
	\arrow["{\alpha_{X,\mathtt{I},Y}}", from=1-1, to=1-3]
\end{tikzcd}
\end{center}
        \item[(c)] the \emph{braiding} $\beta$ with components $$\beta_{X,Y}\colon X\otimes Y\to Y\otimes X$$ satisfying the hexagon identity which can be expressed via the commutative diagram
        \begin{center}
            \begin{tikzcd}
	{(X\otimes Y)\otimes Z} && {X\otimes (Y \otimes Z)} && {(Y\otimes Z)\otimes X} \\
	\\
	{(Y\otimes X)\otimes Z} && {Y \otimes (X \otimes Z)} && {Y \otimes (Z \otimes X)}
	\arrow["{\alpha_{X,Y,Z}}", from=1-1, to=1-3]
	\arrow["{\beta_{X,Y\otimes Z}}", from=1-3, to=1-5]
	\arrow["{\alpha_{Y,Z,X}}", from=1-5, to=3-5]
	\arrow["{\beta_{X,Y\otimes \mathtt{id}_Z}}"', from=1-1, to=3-1]
	\arrow["{\alpha_{Y,X,Z}}"', from=3-1, to=3-3]
	\arrow["{\mathtt{id}_Y\otimes \beta_{X,Z}}"', from=3-3, to=3-5]
\end{tikzcd}
        \end{center}
        Moreover, we require that $$\beta_{Y,X}\circ \beta_{X,Y}=\id_{X\otimes Y}.$$
    \end{itemize} 
    A \emph{symmetric monoidal category} is a tuple $$(\mathcal{C},\otimes,\mathtt{I})=(\mathcal{C},\otimes,\mathtt{I},\alpha,\lambda,\rho,\beta)$$ where $\mathcal{C}$ is a category equipped with a symmetric monoidal structure given by the data  $\otimes,\mathtt{I},\alpha,\lambda,\rho,\beta$. 
    We call $X\otimes Y$ the \emph{tensor product} of $X,Y\in \Cat$. 
    If $A$ is a finite index, and $X_\alpha,\alpha\in A$ are $\Cat$-objects, then we denote by $\bigotimes_{\alpha\in A} X_\alpha=\bigotimes_{\alpha\in A}^\Cat X_\alpha$ their tensor product. 
    \end{itemize}
    \item[(ii)] A symmetric monoidal category $(\mathcal{C},\otimes,\mathtt{I})$ is said to be \emph{semicartesian} (resp.~\emph{cosemicartesian}) if the unitor $\mathtt{I}$ is a terminal (resp.~initial) object. 
    \item[(iii)] Let $(\mathcal{C},\otimes_\mathcal{C},\mathtt{I}_\mathcal{C},\alpha_\mathcal{C},\lambda_\mathcal{C},\rho_\mathcal{C},\beta_\mathcal{C})$ and $(\mathcal{D},\otimes_\mathcal{D},\mathtt{I}_\mathcal{D},\alpha_\mathcal{D},\lambda_\mathcal{D},\rho_\mathcal{D},\beta_\mathcal{D})$ be symmetric monoidal categories. 
    A functor $F\colon \mathcal{C}\to \mathcal{D}$ is said to be a \emph{braided monoidal functor} if there are 
    \begin{itemize}
        \item[(1)] a natural transformation from the bifunctor $\otimes_\mathcal{D}\circ F\times F$ to the bifunctor $F\circ \otimes_\mathcal{C}$ with components denoted by $\phi_{X,Y}$,
        \item[(2)] and a morphism $\phi\colon \mathtt{I}_\mathcal{D}\to F(\mathtt{I}_\mathcal{C})$ in $\mathcal{D}$,
    \end{itemize}  
    such that for all $X,Y,Z\in \mathcal{C}$ the following diagrams commute in $\mathcal{D}$:
    \begin{center}
        \begin{tikzcd}
	{(F(X) \otimes_\mathcal{D} F(Y))\otimes_\mathcal{D} F(Z)} & {F(X) \otimes_\mathcal{D} (F(Y) \otimes_\mathcal{D} F(Z))} \\
	\\
	{F(X\otimes_\mathcal{C}Y)\otimes_\mathcal{D}F(Z)} & {F(X)\otimes_\mathcal{D} F(X\otimes Z)} \\
	\\
	{F((X\otimes_\mathcal{C}Y)\otimes_\mathcal{C}Z)} & {F(X\otimes_\mathcal{C}(Y\otimes_\mathcal{C}Z))}
	\arrow["{\alpha_\mathcal{D}}", from=1-1, to=1-2]
	\arrow["{\mathtt{id}_X\otimes_\mathcal{D}\phi_{Y,Z}}", from=1-2, to=3-2]
	\arrow["{\phi_{X,Y}\otimes_\mathcal{D}\mathtt{id}_Z}"', from=1-1, to=3-1]
	\arrow["{\phi_{X\otimes_\mathcal{C} Y,Z}}"', from=3-1, to=5-1]
	\arrow["{\phi_{X, Y\otimes_\mathcal{C} Z}}", from=3-2, to=5-2]
	\arrow["{F\circ \alpha_\mathcal{C}}"', from=5-1, to=5-2]
\end{tikzcd}

    \begin{tikzcd}
	{F(X)\otimes_\mathcal{D}\mathtt{I}_\mathcal{D}} & {F(X)\otimes_\mathcal{D}F(\mathtt{I}_\mathcal{C})}   \\
	{} \\
	{F(X)} & {F(X\otimes_\mathcal{C} \mathtt{I}_\mathcal{C})}   \\
	\arrow["{\phi_{X,\mathtt{I}_\mathcal{C}}}", from=1-2, to=3-2]
	\arrow["{\mathtt{id}_{F(X)}\otimes_\mathcal{D}\phi}", from=1-1, to=1-2]
	\arrow["{F(\rho_\mathcal{C})}", from=3-2, to=3-1]
	\arrow["{\rho_\mathcal{D}}"', from=1-1, to=3-1]
\end{tikzcd}

    \begin{tikzcd}
	 {\mathtt{I}_\mathcal{D}\otimes_\mathcal{D}F(Y)} & {F(\mathtt{I}_\mathcal{C})\otimes_\mathcal{D} F(Y)} \\
	{} \\
	 {F(Y)} & {F(\mathtt{I}_\mathcal{C}\otimes_\mathcal{C}Y)} \\
	\arrow["{\phi\otimes_\mathcal{D}\mathtt{id}_{F(Y)}}", from=1-1, to=1-2]
	\arrow["{\lambda_\mathcal{D}}"', from=1-1, to=3-1]
	\arrow["{\phi_{\mathtt{I}_\mathcal{C},Y}}", from=1-2, to=3-2]
	\arrow["{F(\lambda_\mathcal{C})}", from=3-2, to=3-1]
\end{tikzcd}
\end{center}
\[
\begin{tikzcd}
	{F(X)\otimes_\mathcal{D}F(Y)} && {F(Y)\otimes_\mathcal{D}F(X)} \\
	&&  \\
	{F(X\otimes_\mathcal{C}Y)} && {F(Y\otimes_\mathcal{C}X)}
	\arrow["{\beta_\mathcal{D}}", from=1-1, to=1-3]
	\arrow["{F(\beta_\mathcal{C})}", from=3-1, to=3-3]
	\arrow["{\phi_{X,Y}}"', from=1-1, to=3-1]
	\arrow["{\phi_{Y,X}}", from=1-3, to=3-3]
\end{tikzcd}
    \]
    \item[(iv)] A functor from a semicartesian symmetric monoidal category $(\c{C},\otimes_{\c{C}},\mathtt{I}_{\c{C}})$ to another one $(\c{D},\otimes_{\c{D}},\mathtt{I}_{\c{D}})$ is said to be \emph{braided} if it is braided as a functor of symmetric monoidal categories and additionally the following diagrams commute for all $X,Y\in \mathcal{C}$ (where $\phi_{X,Y}$ are the components of the natural transformation in the definition of a braided functor above):
\[
    \begin{tikzcd}
	{F(X)\otimes_\mathcal{D}F(Y)} &    \\
	{} \\
	{F(X)} & {F(X\otimes_\mathcal{C} Y)}   \\
	\arrow["{\phi_{X,Y}}", from=1-1, to=3-2]
	\arrow["{\pi_{F(X)}}", from=1-1, to=3-1]
	\arrow["{F(\pi_X)}", from=3-2, to=3-1]
\end{tikzcd}
    \begin{tikzcd}
	  & {F(X)\otimes_\mathcal{D} F(Y)} \\
	{} \\
	{F(X \otimes_\mathcal{C} Y)}  & {F(Y)} \\
	\arrow["{\pi_{F(Y)}}", from=1-2, to=3-2]
	\arrow["{\phi_{X,Y}}"', from=1-2, to=3-1]
	\arrow["{F(\pi_Y)}", from=3-1, to=3-2]
\end{tikzcd}
\]
Dually, one can define braided functors between cosemicartesian symmetric monoidal categories, the details of which we leave to the reader. 
    \end{itemize}
    If $F\colon \c{C}\to \c{D}$ is braided monoidal functor between symmetric monoidal categories $\c{C}$ and $\c{D}$, then we say that tensor products in $\c{C}$ are \emph{related to} tensor products in $\c{D}$ (with respect to $F$). 
    If the natural transformation with components $\phi_{X,Y}$ is a natural monomorphism (resp.~natural isomorphism), then we say that tensor products in $\c{C}$ are \emph{contained in} (resp.~\emph{agree with}) tensor products in $\c{D}$ (with respect to $F$). 
\end{definition}

\begin{example}[Cartesian and cocartesian monoidal categories]\label{exp-cartesian}
Let $\Cat$ be a category with finite categorical products. In particular, $\Cat$ has a terminal object, namely the empty categorical product (this is a consequence of the universal property for the empty categorical product). 
We can equip $\Cat$ with a semicartesian symmetric monoidal structure, where the tensor product is the categorical product, the unitor is the terminal object, and the structure isomorphisms are defined in the obvious way.   
Semicartesian symmetric monoidal categories arising from categories with finite categorical products are called \emph{Cartesian monoidal categories}. 

Dually, if $\Cat$ is category with finite categorical coproducts (such a category has always an initial object, namely the empty coproduct), then we can equip $\Cat$ with a cosemicartesian symmetric monoidal structure, where the tensor product is the categorical coproduct, the unitor is the initial object, and the structure isomorphisms are defined in the obvious way.     
Cosemicartesian symmetric monoidal categories arising from categories with finite categorical coproducts are called \emph{cocartesian monoidal categories}.
\end{example}

\begin{remark}
By \cite[Theorem 3.5]{gerhold}, $\Cat$ is semicartesian if and only if there are natural transformations from the functor $-\otimes Y$ to the identity functor $\id_\Cat$ and from the functor $X\otimes -$ to $\id_\Cat$ for all $X,Y\in \Cat$ with components $\pi_X$ and $\pi_Y$, that is, the diagrams 
\begin{equation}\label{eq-semicartesian}
\begin{tikzcd}
	{X\otimes Y} && X && {X\otimes Y} && Y \\
	\\
	{X'\otimes Y} && {X'} && {X\otimes Y'} && {Y'} \\
	&&& {}
	\arrow["{\pi_X}", from=1-1, to=1-3]
	\arrow["{\pi_{X'}}", from=3-1, to=3-3]
	\arrow["{f\otimes \mathtt{id}_Y}"', from=1-1, to=3-1]
	\arrow["f", from=1-3, to=3-3]
	\arrow["{\pi_Y}", from=1-5, to=1-7]
	\arrow["g", from=1-7, to=3-7]
	\arrow["{\mathtt{id}_{X}\otimes g}"', from=1-5, to=3-5]
	\arrow["{\pi_{Y'}}"', from=3-5, to=3-7]
\end{tikzcd}
\end{equation}
commute for all $\Cat$-morphisms $f\colon X\to X'$ and $g\colon Y\to Y'$. 
Moreover, these natural transformations are required to be compatible with the left and right unitor in the obvious way. 
We call the components $\pi_X$ and $\pi_Y$ the \emph{projections} or \emph{marginalizations}. 

Dually, $\Cat$ is cosemicartesian if and only if there are natural transformations from $\id_\Cat$ to $-\otimes Y$  and from $\id_\Cat$ to $X\otimes -$ for all $X,Y\in \Cat$ with components $\iota_X$ and $\iota_Y$, that is, the diagrams 
\begin{equation}\label{eq-cosemicartesian}
\begin{tikzcd}
	 X && {X\otimes Y} && Y &&  {X\otimes Y}\\
	\\
	{X'} &&  {X'\otimes Y} && {Y'} &&  {X\otimes Y'}\\
	&&& {}
	\arrow["{\iota_X}", from=1-1, to=1-3]
	\arrow["{\iota_{X'}}", from=3-1, to=3-3]
	\arrow["{f\otimes \mathtt{id}_Y}", from=1-3, to=3-3]
	\arrow["f"', from=1-1, to=3-1]
	\arrow["{\iota_Y}", from=1-5, to=1-7]
	\arrow["g"', from=1-5, to=3-5]
	\arrow["{\mathtt{id}_{X}\otimes g}", from=1-7, to=3-7]
	\arrow["{\iota_{Y'}}"', from=3-5, to=3-7]
\end{tikzcd}
\end{equation}
commute for all $\Cat$-morphisms $f\colon X\to X'$ and $g\colon Y\to Y'$. 
Moreover, these natural transformations are required to be compatible with the left and right unitor in the obvious way. We call the components $\iota_X$ and $\iota_Y$ the \emph{inclusions}. 
\end{remark}

Symmetric monoidal categories help to formalize \emph{finite} categorical and non-categorical notions of products. To formalize infinite products in semicartesian (resp.~semicocartesian) categories, we can combine finite monoidal tensor products with inverse (resp.~direct) limits.

\begin{definition}[Infinite tensor products]\label{def-infinit-prod}
(cf.~\cite[Definition 3.1]{fritz}) 
Let $(\Cat,\otimes,\mathtt{I})$ be a semicartesian symmetric monoidal category. Let $(X_j)_{j\in J}$ be a family of objects in $\Cat$. Let $A$ be the directed set of finite subsets of $J$ ordered by inclusion. For $\alpha \in A$, put $X_F\coloneqq \bigotimes_{\alpha\in F} X_\alpha$, and for $\alpha\leq \beta$ in $A$, consider the $\Cat$-morphism $f_{\alpha,\beta}\colon X_{\beta}\to X_\alpha$ (given by \eqref{eq-semicartesian}). The pair $(X_\alpha,f_{\alpha,\beta})$ forms an inverse system in $\Cat$. 
An inverse limit $\varprojlim_\alpha X_\alpha$, if it exists, is said to be an \emph{infinite tensor product} of the $X_j$, if it preserves the functor $-\otimes Y$ for every $\Cat$-object $Y$, that is,  
$$(\varprojlim_\alpha X_\alpha)\otimes Y = \varprojlim_\alpha (X_\alpha\otimes Y)$$
(up to canonical identifications).  

Dually, we can define infinite tensor products in a cosemicartesian symmetric monoidal category replacing inverse systems and limits by direct systems and limits.  

We denote by $\bigotimes_{j\in J} X_j=\bigotimes^\Cat_{j\in J} X_j$ the infinite tensor product of the $X_j$. 

Let $(\c{C},\otimes_\mathcal{C},\mathtt{I}_\mathcal{C})$ and  $(\c{D},\otimes_\mathcal{D},\mathtt{I}_\mathcal{D})$ be semicartesian (resp.~cocartesian) symmetric monoidal categories both admitting infinite tensor products.  
A braided functor $F\colon \mathcal{C}\to \mathcal{D}$ is said to \emph{relate} infinite tensor products if $F$ preserves inverse limits (resp.~direct limits). We can then also speak of that infinite tensor products in $\mathcal{C}$ are \emph{contained in} (\emph{agree with}) infinite tensor products in $\mathcal{D}$. 
\end{definition}

\begin{remark}\label{rem-cartesian}
Notice that in a category where infinite categorical products (resp.~coproducts) exist, the infinite tensor products in the associated cartesian symmetric monoidal category coincide with infinite categorical products (resp.~coproducts).  
\end{remark}

\begin{example}\label{concprod-ex} It follows from the construction of (Fubini type) product measure spaces in \cite[Chapter 3.5]{bogachev2006measure} or \cite[Chapter 254]{fremlinvol2} (note that no separability or standard Borel hypotheses are needed for these product space constructions on arbitrary probability spaces) that $\ConcProb$ has the structure of a cocartesian symmetric monoidal category which admits infinite tensor products.  However, this tensor product does not give $\ConcProb$ the structure of a Cartesian symmetric monoidal category, as already noted in Example \ref{no-univ-prob}.
\end{example}

If $\Cat$ is a semicartesian symmetric monoidal category and $\Cat'$ cartesian monoidal category, then any functor $\Func \colon \Cat \to \Cat'$ relates the two tensor products; however, the two tensor products only agree with respect to $\Func$ if one has the relation
\begin{equation}\label{equiv-prod}
\Hom_{\Cat'}\left( Y \to \Func\left( \bigotimes_{\alpha \in A}^\Cat X_\alpha \right)
\right) = \prod^\Set_{\alpha \in A} \Hom_{\Cat'}( Y \to \Func(X_\alpha) )
\end{equation}
for all $\Cat'$-objects $Y$ and $\Cat$-objects $X_\alpha$, in the sense that the natural map from the left-hand side to the right-hand side is bijective.  Similarly, $\Cat$-tensor product is only contained in the $\Cat'$-tensor product if one has
$$ \Hom_{\Cat'}\left( Y \to \Func\left( \bigotimes_{\alpha \in A}^\Cat X_\alpha \right)
\right) \subseteq \prod^\Set_{\alpha \in A} \Hom_{\Cat'}( Y \to \Func(X_\alpha) )$$
for all $\Cat'$-objects $Y$ and $\Cat$-objects $X_\alpha$, in the sense that the natural map from the left-hand set to the right-hand set is injective.  There are similar equivalences for categorical coproducts which we leave to the reader. 

The following examples may help illustrate these relations:

\begin{example}\label{exp-grp-prod}  The categorical product $\prod^\Group$ agrees with the categorical product $\prod^\Set$ (the direct product of groups uses the Cartesian product of the underlying sets), but the categorical coproduct $\coprod^\Group$, while canonically related to the categorical $\Set$-coproduct, does not agree with it or even contain it (the canonical $\Set$-morphism between the two coproducts maps all of the identity elements of each group to a single point).  The categorical coproduct $\coprod^{\SigmaAlg}$ does not agree with the categorical coproduct $\coprod^\Bool$, but does at least contain it (for instance, if $\X, \X'$ are $\SigmaAlg$-algebras, then the categorical coproduct  $\X \otimes^{\SigmaAlg} \X'$ contains the categorical coproduct $\X \otimes^\Bool \X'$ as a Boolean subalgebra). 
\end{example}

We review a special case of the construction of Grothendieck categories, called action categories, which will be useful for us to associate to a category of spaces (resp.~algebras) a corresponding category of probability spaces (resp.~tracial algebras). 

\begin{definition}[Action category]\label{def-action}
Let $\Cat$ be a category and $\mathfrak{P}\colon \Cat\to \Set$ be a functor. 
The objects of the associated \emph{action category} $\Cat\ltimes \mathfrak{P}$ consist of pairs $(X,\mu)$, where $X\in \Cat$ and $\mu\in \mathfrak{P}(X)$, and a $\Cat\ltimes \mathfrak{P}$-morphism from a  $\Cat\ltimes \mathfrak{P}$-object $(X,\mu)$ to another  $\Cat\ltimes \mathfrak{P}$-object $(Y,\nu)$ is a $\Cat$-morphism $f\colon X\to Y$ with the property that $\mathfrak{P}(f)(\mu)=\nu$.  
\end{definition}

\begin{remark}\label{tensor-action}
Suppose that $\Cat$ is a Cartesian symmetric monoidal category and $\mathfrak{P}\colon \Cat\to \Set$ is a braided functor such that categorical products in $\Cat$ agree with product of sets with respect to $\mathfrak{P}$. By chasing definitions, one can verify that the action category $\Cat\ltimes \mathfrak{P}$ has the structure of a semicartesian symmetric monoidal category with the induced tensor product $(X,\mu)\otimes^{\Cat\ltimes \mathfrak{P}} (Y,\nu)\coloneqq (X\otimes^\Cat Y,\mu\times \nu)$, where $\mu\times \nu$ is the unique element in $\mathfrak{P}(X\otimes^\Cat Y)$ which corresponds to $(\mu,\nu)\in \mathfrak{P}(X)\times^\Set \mathfrak{P}(Y)$. 

Moreover if $\Cat$ admits arbitrary categorical products and $\mathfrak{P}$ preserves inverse limits (resp.~direct limits), then $\Cat\ltimes \mathfrak{P}$ admits infinite tensor products with respect to the induced tensor product.    
\end{remark}

\begin{example}\label{concprob-ex}
Consider the category of concrete measurable spaces $\ConcMes$. 
The functor $\mathtt{Prb}\colon \ConcMes\to \Set$ sends an object $(X,\Sigma_X)\in \ConcMes$ to the set $\mathtt{Prb}(X,\Sigma_X)$ of probability measures on $(X,\Sigma_X)$ and a morphism $f\colon (X,\Sigma_X)\to (Y,\Sigma_Y)$ in $\ConcMes$ to the pushforward map $\mathtt{Prb}(f)\colon \mathtt{Prb}(X,\Sigma_X)\to \mathtt{Prb}(Y,\Sigma_Y)$ defined by $\mathtt{Prb}(f)(\mu)=f_\ast\mu$. 
The action category $\Cat\ltimes \mathfrak{P}$ can be identified with the category $\ConcProb$ of concrete probability spaces and measure-preserving maps. 
Since $\ConcMes$ admits arbitrary categorical products, $\ConcProb$ admits infinite tensor products (cf.~Remark \ref{concprod-ex}).  
\end{example}

\section{A counterexample to a claim of Halmos}\label{halmos-sec}

Suppose that $X = (\Inc(X)_\SigmaAlg, \mu_X)$ is a $\ProbAlg$-space, and let $\Gamma$ be a discrete group acting on $X$ by $\ProbAlg$-isomorphisms, thus for each $\gamma \in \Gamma$ one has an $\ProbAlg$-isomorphism $T^\gamma \colon X \to X$ such that $T^\gamma T^{\gamma'} = T^{\gamma \gamma'}$ for $\gamma,\gamma' \in \Gamma$.  We can form the invariant factor $\Inv(X)$ by replacing the $\SigmaAlg$-algebra $\Inc(X)_\SigmaAlg$ by the invariant subalgebra
$$ \Inc(X)_\SigmaAlg^\Gamma \coloneqq \{ E \in \Inc(X)_\SigmaAlg: (T^\gamma)^* E = E \hbox{ for all } \gamma \in \Gamma \}.$$
There is then an obvious factor $\ProbAlg$-morphism $\pi \colon X \to \Inv(X)$.  Applying the canonical model functor $\Stone$, we obtain a $\CHProb$-morphism $\Stone(\pi) \colon \Stone(X) \to \Stone(\Inv(X))$, and applying canonical disintegration (Theorem \ref{canon-disint}), we may disintegrate $\mu_{\Stone(X)}$ into probability measures $\mu_y$ supported on fibers $\Stone(\pi)^{-1}(\{y\})$ for all $y \in \Stone(\Inv(X))$.  From the uniqueness of the disintegration it follows that the measures $\mu_y$ are invariant with respect to the continuous action $\gamma \mapsto \Stone(T^\gamma)$.  It was asserted without proof in \cite[\S 4]{halmos-dieudonne} that these measures are furthermore ergodic with respect to this action, that is to say all invariant Baire-measurable subsets on $\Stone(X)$ have measure $1$ or $0$ with respect to $\mu_y$.  The purpose of this appendix is to supply a counterexample to this claim in which $\Gamma$ is uncountable.  The failure is somewhat dramatic in the sense that \emph{every} fiber measure $\mu_y$ is non-ergodic.  We are unable to determine whether the claim may still hold for countable $\Gamma$, either for all $y$ or for almost all $y$.

To build the counterexample, we first define the $\CHProb$-space
$$ Y \coloneqq \Stone([0,1])$$
where the unit interval $[0,1]$ is given the usual Lebesgue measure.  Clearly for any natural number $n$, $[0,1]$ can be partitioned into $n$ measurable subsets of measure $1/n$, so the same is true for $\Stone([0,1])$.  In particular we see that every point $y$ in $Y$ has zero outer measure, in the sense that it can be contained in sets of arbitrarily small measure.

Next, we define the product space
$$ X \coloneqq Y \times^\CHProb \prod_{y \in Y}^\CHProb \{-1,1\}$$
where $\{-1,1\}$ is the discrete two-element multiplicative group with probability Haar measure.  In particular, each element of $X$ takes the form
$$ (y_1, (\alpha_{y_2})_{y_2 \in Y})$$
with $y_1 \in Y$ and $\alpha_{y_2} \in \{-1,+1\}$ for $y_2 \in Y$.  We abbreviate $\prod_{y \in Y}^\CHProb \{-1,1\}$ as $\{-1,1\}^Y$; this is a compact Hausdorff abelian group equipped with Haar probability measure $\mu_{\{-1,1\}^Y}$.  For any $f \in L^1(X)$ we have a conditional expectation $\E(f|Y) \in L^1(Y)$ defined by
$$ \E(f|Y)(y) \coloneqq \int_{\{-1,1\}^Y} f(y,\alpha)\ d \mu_{\{-1,1\}^Y}(\alpha).$$

Observe that for any $h \in \CFunc(Y)$ and any finite subset $I \subseteq Y$, the function $h \otimes \prod_{y \in I} \epsilon_y \colon X \to \C$ defined by
$$ h \otimes \prod_{y \in I} \epsilon_y\left( y_1, (\alpha_{y_2})_{y_2 \in Y}\right) \coloneqq h(y_1) \prod_{y \in I} \alpha_y $$
is an element of $\CFunc(X)$.  If $I$ is empty we write $h \otimes \prod_{y \in I} \epsilon_y$ as $h \otimes 1$.  From the Stone--Weierstrass theorem we see that the space of finite linear combinations of such functions $h \otimes \prod_{y \in I} \epsilon_y$ is dense in $\CFunc(X)$ in the uniform topology, and hence also dense in $L^2(X)$ in the $L^2$ topology.

We let $\Gamma$ be the discrete abelian multiplicative group
$$ \Gamma = \Hom_\CH( Y \to \{-1,1\} )$$
of continuous maps $\gamma \colon Y \to \{-1,+1\}$ from $Y$ to $\{-1,1\}$.  We define a $\CHProb$-action $T^\gamma \colon X \to X$ of $\Gamma$ on $X$ by the formula
$$ T^\gamma( y_1, (\alpha_{y_2})_{y_2 \in Y} ) \coloneqq 
(y_1, (\alpha_{y_2} \gamma(y_1) \gamma(y_2))_{y_2 \in Y} ).$$
It is not difficult to verify (using Fubini's theorem) that this is a $\CHProb$-action.

The intuition here is that the set
$$ \{ (y_1, (\alpha_{y_2})_{y_2 \in Y}): \alpha_{y_1} = +1 \}$$
appears to be an invariant subset of $X$ ``of measure $1/2$''.  However, it turns out that this set is not Baire measurable and so will not show up in the invariant factor of $X_\ProbAlg$.  However, when passing to the canonical models, there are analogues of this invariant set which are measurable with respect to the individual $\mu_y$ and can be used to contradict the ergodicity of these measures.

We turn to the details.  For any $h \in \CFunc(Y)$ we see that $h \otimes 1$ is an invariant element of $L^2(X)$; on taking closures we see that $L^2(Y)$ is contained in the invariant factor of $L^2(X)$.  We claim that this is the entire invariant factor.  To show this, it suffices by the Alaoglu-Birkhoff theorem \cite{ab40abstract} to check that for every $f \in L^2(X)$, that the element of minimal norm in the closed convex hull of the orbit $\{ f \circ T^\gamma: \gamma \in \Gamma \}$ is equal to $\E(f|Y)$.  By density and linearity it suffices to verify this when $f$ is of the form $f = h \otimes \prod_{y \in I} \epsilon_I$ for some $h \in \CFunc(Y)$ and finite $I \subseteq Y$.  If $I$ is empty this follows from the invariance of the $h \otimes I$, so now suppose that $I$ is non-empty.  Then $\E(f|Y)=0$ and the objective is now to show that convex combinations of $f \circ T^\gamma$ can have arbitrarily small $L^2(X)$ norm.

Direct calculation shows that for any $\gamma \in \Gamma$ we have
$$
T^\gamma f = \left(\prod_{y \in I} \gamma(y)\right) (h\gamma^{|I|}) \otimes \prod_{y \in I} \epsilon_I,$$
where $|I|>0$ denotes the cardinality of $I$.  If $|I|$ is even, we can use Urysohn's lemma to choose $\gamma \in \Gamma$ so that $\prod_{y \in I} \gamma(y)=-1$, and thus 
$$ T^\gamma\left(h \otimes \prod_{y \in I} \epsilon_y\right) = - h \otimes \prod_{y \in I} \epsilon_y,$$
giving the claim in this case.  If instead $|I|$ is odd, we use the fact that $I$ has zero outer measure and Urysohn's lemma to find $\gamma \in \Gamma$ such that $\gamma=+1$ on a neighbourhood of $I$ of arbitrarily small measure, and $-1$ otherwise, then we see that the function
$$\frac{1+T^\gamma}{2} \left(h \otimes \prod_{y \in I} \epsilon_y\right)$$ 
is bounded and supported on a set of arbitrarily small measure, hence is arbitrarily small in $L^2(X)$ norm as required.  This establishes that $Y$ is the invariant factor of $X$: $\Inv(X_\ProbAlg) \equiv Y_\ProbAlg$.  Thus the projection map $\pi \colon X_\ProbAlg \to \Inv(X_\ProbAlg)$ can be identified with the obvious projection map from $X$ to $Y$, casted to $\ProbAlg$.

Now we pass to the canonical models $\Stone(X), \Stone(Y)$.  For $y \in \Stone(Y)$, we have from construction that
\begin{equation}\label{stone-e}
\int_{\Stone(X)} f\ d\mu_y = \E(f|Y)(y)
\end{equation}
for $f \in \CFunc(\Stone(X)) \equiv \Linfty(X)$, where we view $\E(f|Y) \in \Linfty(Y)$ as an element of $\CFunc(\Stone(Y))$.  We claim that the function $1 \otimes \epsilon_y \in \Linfty(X) \equiv \CFunc(\Stone(X))$ is invariant but non-constant in $L^\infty(\Stone(X), \mu_y)$, which will demonstrate the non-ergodicity of $\mu_y$.

 To do this it will suffice to establish the identities
\begin{align*}
\int_{\Stone(X)} (1 \otimes \epsilon_y)\ d\mu_y &= 0\\
\int_{\Stone(X)} |(1 \otimes \epsilon_y)|^2\ d\mu_y &= 1 \\
\int_{\Stone(X)} \left|(1 \otimes \epsilon_y) \circ T^\gamma - (1 \otimes \epsilon_y)\right|^2\ d\mu_y &= 0
\end{align*}
for all $\gamma \in \Gamma$.  
But direct calculation shows that
\begin{align*}
\E(1 \otimes \epsilon_y|Y)(y') &= 0 \\
\E\left(|1 \otimes \epsilon_y|^2 |Y\right)(y') &= 1 \\
\E\left(\left|(1 \otimes \epsilon_y) \circ T^\gamma - (1 \otimes \epsilon_y)\right|^2 |Y\right)(y') &= |\gamma(y)-\gamma(y')|^2
\end{align*}
for any $y' \in Y$, and the claim now follows from \eqref{stone-e}.

\subsection*{Acknowledgements}
AJ was supported by DFG-research fellowship JA 2512/3-1. 
TT was supported by a Simons Investigator grant, the James and Carol Collins Chair, the Mathematical Analysis \& Application Research Fund Endowment, and by NSF grant DMS-1764034.  We thank Balint Farkas, Tobias Fritz, Markus Haase, and David Roberts for helpful comments and references. We are indebted to an anonymous referee for a careful reading and many useful comments and suggestions.


\end{document}